\documentclass[11pt]{article}
\usepackage{amsmath,amssymb,graphicx}
\usepackage{subcaption}
\usepackage{amsfonts, amsmath, amssymb,latexsym,amsthm, mathrsfs, mathtools}
\usepackage[margin=0.9in]{geometry}
\usepackage{url}
\usepackage[english]{babel}
\usepackage{epsfig}
\usepackage{pifont}
\usepackage{tikz-cd}
\usepackage{authblk}

\newtheorem{thm}{Theorem}[section]
\newtheorem*{thm*}{Theorem}
\newtheorem*{prop*}{Proposition}
\newtheorem*{corol*}{Corollary}
\newtheorem{prop}[thm]{Proposition}
\newtheorem{corol}[thm]{Corollary}
\newtheorem{lemma}[thm]{Lemma}
\usepackage[all]{xy}
\newtheorem*{claim*}{Claim}

\usepackage{chngcntr}
\counterwithin{figure}{section}

\usepackage{enumitem}
\theoremstyle{definition}
\newtheorem{defi}[thm]{Definition}
\newtheorem*{defi*}{Definition}

\newtheorem{remark}[thm]{Remark}

\theoremstyle{plain}
\newtheorem*{namedthm}{\namedthmname}
\newcounter{namedthm}
\makeatletter
\newenvironment{named}[1]
{\def\namedthmname{#1}%
	\refstepcounter{namedthm}%
	\namedthm\def\@currentlabel{#1}}
{\endnamedthm}
\makeatother

\usepackage[maxbibnames=99, backend=bibtex, style=alphabetic, sorting=nyt]{biblatex}
\usepackage{fancyhdr}

\usepackage[hidelinks]{hyperref}

\title{A model for boundary dynamics of Baker domains}
\author[1,2]{N\'uria Fagella\thanks{This work is partially supported by the (a) Spanish State Research Agency, through the Severo Ochoa and María de Maeztu Program for Centers and Units of Excellence in R\&D (CEX2020-001084-M, and PID2020-118281GB-C32; (b) Generalitat de Catalunya through the grants 2017SGR1374 and ICREA Academia 2020.}}
\author[1]{Anna Jov\'e\thanks{Supported by the Spanish government grant FPI PRE2021-097372.}}
\affil[1]{\small Departament de Matemàtiques i Informàtica, Universitat de Barcelona, Barcelona, Spain}
\affil[2]{\small Centre de Recerca Matemàtica, Barcelona, Spain}

\begin{document}
\maketitle
\begin{abstract}
	We consider the transcendental entire function $ f(z)=z+e^{-z} $, which has a doubly parabolic Baker domain $ U $ of degree two, i.e. an invariant stable component for which all iterates converge locally uniformly to infinity, and for which the hyperbolic distance between successive iterates converges to zero. It is known from general results that the dynamics on the boundary is ergodic and recurrent and that the set of points in $ \partial U $ whose orbit escapes to infinity has zero harmonic measure. For this model we show that stronger results hold, namely that this escaping set is non-empty, and it is organized in curves encoded by some symbolic dynamics, whose closure is precisely $ \partial U $.
	We also prove that nevertheless, all escaping points in $ \partial U $ are non-accessible from $ U $, as opposed to points in $ \partial U $ having a bounded orbit, which are all accessible. Moreover, repelling periodic points are shown to be dense in $ \partial U $, answering a question posted in \cite{bfjk19}. None of these features are known to occur for a general doubly parabolic Baker domain.
\end{abstract}
\section{Introduction}

We consider a transcendental entire function $ f\colon\mathbb{C}\to\mathbb{C} $ and denote by $ \left\lbrace f^n\right\rbrace _{n\in\mathbb{N}} $ its iterates, which generate a discrete dynamical system in $ \mathbb{C} $. Then, the complex plane is divided into two totally invariant sets: the {\em Fatou set} $ \mathcal{F}(f) $, defined to be the set of points $ z\in\mathbb{C} $ such that $ \left\lbrace f^n\right\rbrace _{n\in\mathbb{N}} $ forms a normal family in some neighbourhood of $ z$; and the {\em Julia set} $ \mathcal{J}(f) $, its complement. Another dynamically relevant set is the escaping set $ \mathcal{I}(f) $, where points converge to infinity, the essential singularity of the function.
For background on the iteration of entire functions see e.g. \cite{bergweiler}.

The Fatou set is open and consists typically of infinitely many connected components, called {\em Fatou components}. Due to the invariance of the Fatou and the Julia sets, Fatou components are periodic, preperiodic or wandering. For entire functions, periodic ones are always simply connected \cite{baker1984}, so the Riemann map can be used as a uniformization. More precisely, let $ U $ be a invariant Fatou component of $ f $ and let $ \varphi $ be a Riemann map from the open unit disk $ \mathbb{D} $ onto $ U $. Then,  \[
g\colon\mathbb{D}\longrightarrow\mathbb{D},\hspace{1cm}g\coloneqq \varphi^{-1}\circ f\circ\varphi
\] is an analytic self-map of $ \mathbb{D} $, and  $ f_{|U} $ and $ g_{|\mathbb{D}} $ are conformally conjugate by $ \varphi $. Therefore, the study of holomorphic self-maps of $ \mathbb{D} $ is a good approach to analyze the dynamics of $ f_{|U} $. 

The Denjoy-Wolff Theorem (see Sect. \ref{sect-2-prelim}) asserts that, whenever a holomorphic self-map $ g $ of $ \mathbb{D} $ is not conjugate to a rotation, all orbits converge to the same point $ p\in\overline{\mathbb{D}} $ (the \textit{Denjoy-Wolff point} of $ g $). From this celebrated result,   the classification theorem of invariant Fatou components of entire maps can be deduced, which was proved earlier by Fatou (\cite{fatou1920}) using different techniques. Indeed, a given invariant Fatou component is either a \textit{Siegel disk} (when it is conjugate to an irrational rotation), an \textit{attracting basin }(when all orbits converge to the same point in $ U $)
 or a \textit{parabolic basin} or a \textit{Baker domain} (when all orbits converge to the same point in $\partial U $). The difference between the last two possibilities comes from the nature of the convergence point: for Baker domains it is the essential singularity, so $ f $ is not defined at it; whereas for parabolic basins, it is a fixed point of multiplier 1.

One may ask if the previous conjugacy with a holomorphic self-map of $ \mathbb{D} $ can be used to describe the dynamics of $ f $ in the boundary of $  U $. First, from the fact that $ f(\partial U)\subset \partial U $, it can be deduced that $ g $ is an {\em inner function}, i.e. an analytic self-map of $ \partial\mathbb{D} $ such that the radial limits  belong to $\partial \mathbb{D}$ for  almost every point in $\partial \mathbb{D}$. Hence, a boundary extension \[g^*\colon E\subset\partial\mathbb{D}\to\partial\mathbb{D} \]can be defined using radial limits, where $ E $ is a set of full measure in $ \partial\mathbb{D} $, and it induces a dynamical system defined almost everywhere on $ \partial \mathbb{D} $. One may expect a priori that $ f_{|\partial U} $ and $ g^*_{|\partial \mathbb{D}} $ share dynamical properties.
Nevertheless, this is not always the case. The main obstacle is that the Riemann map cannot be assumed to extend continuously to the boundary. In fact, this is the usual case for unbounded Fatou components of transcendental entire functions (compare \cite{baker-dominguez, bargmann}).  Therefore, $ \varphi $ is no longer a conjugacy in $  \partial \mathbb{D}$ and properties of $ g^*_{|\partial \mathbb{D}} $ do not transfer to $ f_{|\partial U} $ in general. However, successful results have been obtained in some cases.

 First, Devaney and Goldberg   studied the exponential family $ \lambda e^z $ with $ 0<\lambda<\frac{1}{e} $, \cite{goldberg-devaney},  whose Fatou set consists of a totally invariant attracting basin $ U $. From the explicit computation of the inner function, accesses to infinity were characterized, and the boundary of $ U $, which is precisely the Julia set, was shown to be organized in curves of escaping points and their endpoints, the latter being the only accessible points from $ U $. Such results were generalized to a larger family of functions having a totally invariant attracting basin \cite{baranski, baranski-karpinska}. 
 
On the basis of this successful example, inner functions have been used systematically to understand the dynamics on the boundary of  Fatou components. On the one hand, results of \cite{baker-dominguez,bargmann, bfjk15} describe the topology of the boundary of unbounded Fatou components and their accesses to infinity.  On the other hand, the revealing work in \cite{doering-mañé}, further developed in \cite{rippon-stallard, bfjk19}, describe their ergodic properties. 

We focus on a precise type of periodic Fatou components, {\em Baker domains}, in which iterates converge locally uniformly to infinity. Maps possessing Baker domains are not hyperbolic, nor bounded type (i.e. the set of singularities of the inverse branches of the function is unbounded \cite{eremenko-lyubich}). In contrast with the other periodic Fatou components, in which the dynamics around the convergence point can be conjugate to some predetermined normal form, three different asymptotics are possible for Baker domains (see Thm. \ref{teo-cowen} and Rmk. \ref{remark-bakers}). This leads to a further classification according to their internal dynamics into \textit{doubly parabolic}, \textit{hyperbolic} and \textit{simply parabolic} Baker domains, which also present different boundary properties.

Even though all orbits in a Baker domain tend to infinity, it is still unknown whether a single escaping point  always exists in $     \partial U $. For hyperbolic and simply parabolic univalent Baker domains this question was answered affirmatively by Rippon and Stallard \cite{rippon-stallard}, who showed that the set of boundary escaping points has full harmonic measure with respect to the Baker domain. This result was generalized to finite degree Baker domains and to infinite degree under certain assumptions \cite[Thm. A]{bfjk19}.

On the contrary, for doubly parabolic Baker domains of finite degree the set of  escaping boundary points is known to have zero harmonic measure \cite[Thm. B]{bfjk19}. This connects with the fact that, for the corresponding inner function, no point in $ \partial\mathbb{D} $ converges to the Denjoy-Wolff point. However, the boundaries of such Baker domains are always non-locally connected \cite[Thm. 3.1]{bargmann}, so the Riemann map cannot be used to rule out the existence of  escaping boundary points.  Other unanswered questions about the boundaries of such Baker domains concern periodic points, which are not known to exist in general, or the connection between the accessibility of boundary points and their dynamics.

In this paper, we present a detailed analysis of the dynamics of the transcendental entire function $ f(z)=z+e^{-z} $, which possesses countably many doubly parabolic Baker domains of degree two. It is our belief that a good understanding of this model will throw some light about the correspondence between the inner function and the boundary map, in a more explicit way than the abstract existence of measurable sets. In our work, other interesting properties of both the inner function and the boundary of the Baker domain arise, and are susceptible to hold for a wider family of functions. 

The function we consider, $ f(z)=z+e^{-z} $,  is one of the few explicit examples having doubly parabolic Baker domains of finite degree. Because of that, it was studied previously in \cite{baker-dominguez, fh, bfjk19}. However, many aspects concerning boundary dynamics are still unexplored, and are the object of this paper. 

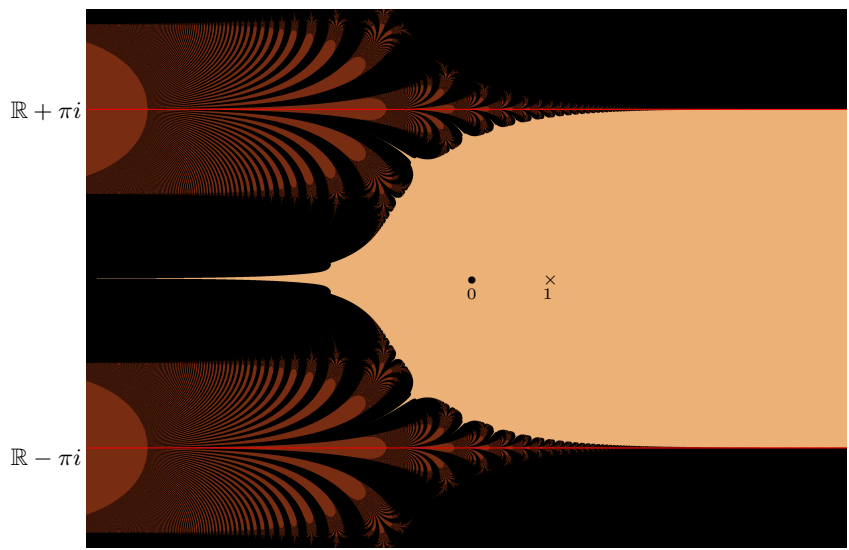
\begin{figure}[htb!]\centering
	\includegraphics[width=10cm]{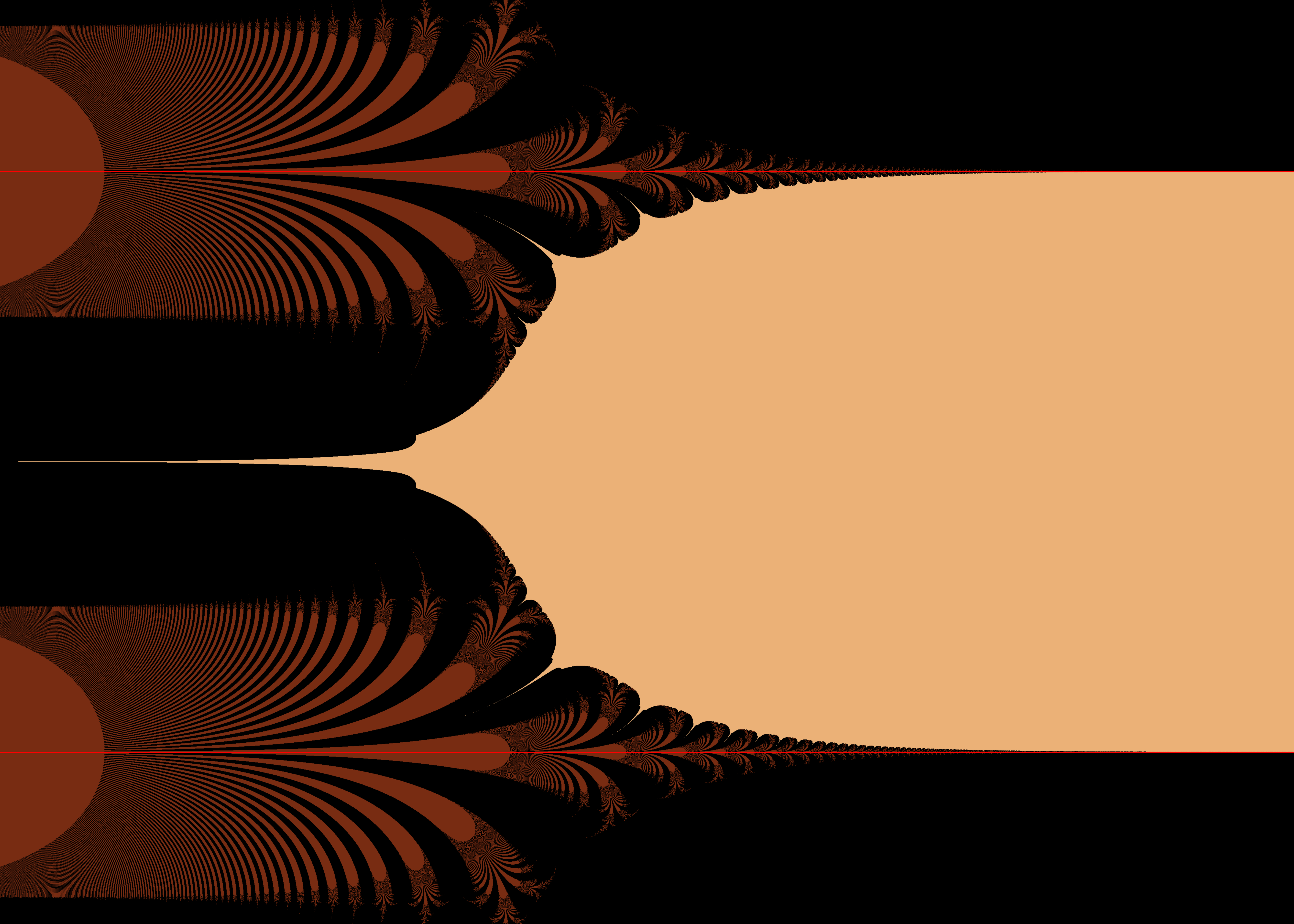}
		\setlength{\unitlength}{10cm}
	\put(-1.1, 0.57){\footnotesize $ \mathbb{R}+\pi i $}
\put(-1.1, 0.11){\footnotesize $ \mathbb{R}-\pi i $}
\put(-0.5, 0.35){\tiny $ \bullet$}
\put(-0.4, 0.35){\tiny $ \times$}
\put(-0.5, 0.33){\tiny $ 0$}
\put(-0.4, 0.33){\tiny $ 1$}
	\caption{\footnotesize Dynamical plane for $ f(z)=z+e^{-z} $. In red, the Julia set of $ f $. In beige, the Baker domain contained in the strip $ \left\lbrace -\pi<\textrm{Im }z<\pi\right\rbrace  $. In black, the rest of the Fatou set of $ f $. The only critical point on the strip (0) is also marked, as well as the corresponding critical value (1).}
\end{figure}

First, Baker and Domínguez \cite[Thm. 5.1]{baker-dominguez} and Fagella and Henriksen \cite[Example 3]{fh} proved, using different arguments, the existence of a doubly parabolic Baker domain $ U_k $ of degree two in each strip $ S_k\coloneqq\left\lbrace (2k-1)\pi\leq\textrm{Im }z\leq(2k+1)\pi\right\rbrace  $, for all $ k\in \mathbb{Z} $. Since the dynamics in all of them are the same, we consider only the Baker domain $ U \coloneqq U_0 $ in the strip $S\coloneqq S_0=\left\lbrace -\pi\leq\textrm{Im }z\leq\pi\right\rbrace  $.  In \cite[Thm. 5.2]{baker-dominguez}, the associated inner function is computed explicitly.

The topology of $ \partial U $ is addressed in \cite[Section 6]{baker-dominguez}, where  it is deduced that $ \partial U $ is  non-locally connected and preimages of infinity by the Riemann map $ \varphi $ (in the sense of radial limits) are dense in the unit circle. Going one step further,  they proved that the impression of the prime end corresponding to 1 is precisely $ \partial S \cup \left\lbrace  \infty\right\rbrace $. Using this, accesses to infinity from $ U $ were characterized in terms of the inner function.

Finally, in \cite[Example 1.2]{bfjk19}, they describe some dynamical sets in $ \partial U$ in terms of measure, as an application of a general theorem (\cite[Thm. B]{bfjk19}). More precisely, they show  that almost every point with respect to the harmonic measure has a dense orbit in $ \partial U $. Therefore, the escaping points in $ \partial U $ have zero harmonic measure. Moreover, 
they conjectured that all escaping points in $ \partial U $ are non-accessible from $ U $ and accessible repelling periodic points are dense in $ \partial U $. In this paper we prove both conjectures.

\subsection*{Statement of results}
Following the approach of \cite{bfjk19}, we aim to give an explicit description of the sets of full and zero harmonic measure which appear as a result of the general ergodic theorems. 

First we describe the escaping set in $ \partial U $. Recall that it is known to have zero harmonic measure so, {a priori}, it is unknown whether it is non-empty. We prove that escaping points do exist in $ \partial U $ and are organized in curves (known as \textit{dynamic rays} or \textit{hairs}) encoded by some symbolic dynamics, as it is not uncommon for transcendental entire functions.  All escaping points are proved to belong to such curves, while non-escaping points are in their accumulation sets. This leads to the following description of the boundary of $ U $.

\begin{named}{Theorem A}\label{teo:A} {\bf (The boundary of $ U $)} 
	Every escaping point in $ \partial U $ can be connected to $ \infty $ by a unique curve of escaping points in $ \partial U $. Moreover, $ \partial U $ is the closure of such curves.
\end{named}

The existence of these dynamic rays follows from general results of \cite{r3s} applied to $ h(w)=we^{-w} $, semiconjugate to $ f $ by $ w=e^{-z} $. From these general results it is deduced that $ h $, and therefore $ f $, are \textit{criniferous functions}, i.e. that all points in the escaping set can be connected to infinity by a curve of escaping points: the dynamic ray. Criniferous functions were introduced in \cite{benini-rempe}, and further studied in \cite{leti}. Nevertheless, in order to have a better control on the geometry of the dynamic rays and their relation with the boundary of $ U $, we choose to prove \ref{teo:A} with an explicit construction, which gives us additionally a parametrization and certain continuity properties.

 Other remarkable properties are observed, such as that all points in $ \partial U $ escape to $ \infty $ in a different ``direction'' than that of the dynamical access. This connects with the fact that, for the inner function, there is no escaping point (in the sense that there are no boundary orbits converging to the Denjoy-Wolff point, apart from the preimages of itself).
Moreover, escaping orbits in $ \partial U $ converge to $ \infty $ exponentially fast, while points in $ U$ do so in a slower fashion, being the map close to the identity.

Next,  we study the landing properties of the dynamic rays mentioned above. More precisely, we prove the following.

\begin{named}{Theorem B}\label{teo:Indec}{\bf (Landing and non-landing dynamic rays)} There exist uncountably many dynamic rays which land at a finite end-point, and there exists uncountably many dynamic rays which do not land. The accumulation set (on the Riemann sphere) of such a non-landing ray is an indecomposable continuum which contains the ray itself.
\end{named}
This contrasts with the exponential maps $ \lambda e^z $, with $ 0<\lambda<\frac{1}{e} $, where all dynamic rays land, due to hyperbolicity. 

On the other hand, 
indecomposable continua were shown to exist in the Julia set of some non-hyperbolic exponential maps $ E_\kappa(z)= e^z+\kappa $, for some values of $ \kappa $, first in \cite{devaney-indecomposable} and later on  \cite{devaney-jarque-indecomposable, devaney-jarque-rocha-indecomposable}, although not as the accumulation set of a dynamic ray. It was shown by Rempe \cite{tesi-lasse-rempe,rempe2007} that indecomposable continua appear as the accumulation set of a dynamic ray in exponential maps $ E_\kappa $, for some values of $ \kappa $. More precisely, he proves that if the singular $ \kappa  $ is on a dynamic ray, then there exist uncountably many dynamic rays whose accumulation set is an indecomposable continuum. However, for the exponential maps $ E_\kappa $, if the singular value is on a dynamic ray, then $ \mathcal{J}(E_\kappa)=\mathbb{C} $ or the Fatou set consists of Siegel disks and preimages of them (see e.g. \cite{DevaneyExp}). Contrastingly, we find these indecomposable continua in the boundary of a Baker domain (and in the boundary of the projected parabolic basin, see Sect. \ref{sec-3-dynamics-f}).

We also address the problem of relating the previous sets, of escaping and non-escaping points, with the set of accessible boundary points from $ U $. Again, symbolic dynamics play an important role, in this case to connect the dynamics in the unit circle with the behaviour in  $\partial U $.

\begin{named}{Theorem C}\label{teo:B}{\bf (Accessible points)}
	Escaping points in $ \partial U $ are non-accessible from $ U $, while points in $ \partial U $ having a bounded orbit are all accessible from $ U $.
\end{named}

Finally, we study periodic points in $ \partial U $. We show  that $ g_{|\partial\mathbb{D}} $ is conjugate to the doubling map (see Sect. \ref{sect-6-accessibility}), so periodic points for $ g $ are dense in $ \partial\mathbb{D} $. Moreover, \ref{teo:B} asserts that periodic points in $ \partial U $, if they exist, are accessible. Both things suggest that periodic points might be dense in  $ \partial U $, which is indeed proven in the following theorem.

\begin{named}{Theorem D}\label{teo:C}{\bf (Periodic points)}
Periodic points are dense in $ \partial U $.
\end{named}
 
 We observe that first statement in \ref{teo:B} corresponds to the first part of the conjecture in \cite{bfjk19}, while the second statement together with \ref{teo:C} provide a positive answer to the second part.
 
 \vspace{0.2cm}
{\bf Structure of the paper.}
Section \ref{sect-2-prelim} is devoted to reviewing some general results on the dynamics of Fatou components, and in particular, of Baker domains, as well as to state other preliminary results needed in the rest of the paper. In Section \ref{sec-3-dynamics-f} one finds the auxiliary results about the dynamics of $ f(z)=z+e^{-z} $, which are used recurrently in the following sections. For completeness, a sketch of the general dynamics of $ f $ is included, summarizing the ideas of \cite{baker-dominguez} and \cite{fh}. Section \ref{sect-4-escaping} is devoted to studying the escaping set and its organization in dynamic rays, proving \ref{teo:A}.  The landing properties of such  rays are discussed in Section \ref{sect-5-non-escaping}, together with the proof of \ref{teo:Indec}. \ref{teo:B} and \ref{teo:C} are proved in Section \ref{sect-6-accessibility} and \ref{sect-periodic-points}, respectively. 

\vspace{0.2cm}
{\bf Notation.} Throughout this article, $ \mathbb{C} $ and $ \widehat{\mathbb{C}} $ denote the complex plane and the Riemann sphere, respectively. The positive and negative real axis are indicated by $ \mathbb{R}_+$ and $ \mathbb{R}_-$, respectively; while the upper and the lower half-plane are indicated by $ \mathbb{H}_+$ and $ \mathbb{H}_-$, respectively. The following notation is used for the horizontal strips of width $ 2\pi $
\[ S_k\coloneqq\left\lbrace (2k-1)\pi\leq\textrm{Im }z\leq(2k+1)\pi\right\rbrace . \] We denote by $ U_k $ the unique Baker domain contained in the strip $ S_k $. We shall denote by $ S $ the central strip $ S_0 $, and by $ U $ its Baker domain, just to lighten the notation.

Given a set $ A\subset \mathbb{C} $, we denote by $ \overline{ A } $ and $ \partial A $, its closure and its boundary taken in $ \mathbb{C} $; and by $ \widehat{\partial} A $,  its boundary when considered in $ \widehat{\mathbb{C}} $.  We denote by $ \textrm{dist}(\cdot, \cdot) $ the Euclidean distance between two points, and for $ z\in\mathbb{C}$ and $ X,Y \subset \mathbb{C}$ we write
\[ \textrm{dist}(z, X)=\inf\limits_{x\in X} \textrm{dist}(z,x),\hspace{0.7cm}\textrm{dist}(X,Y)=\inf\limits_{x\in X, y\in Y} \textrm{dist}(x,y).\]

\vspace{0.2cm}
{\bf Acknowledgments.} We would like to thank Vasiliki Evdoridou, Xavier Jarque, Leticia Pardo-Simón, Lasse Rempe, Phil Rippon and Gwyneth Stallard for interesting discussions and comments, and specially Arnaud Chéritat, for suggesting the presence of indecomposable continua in the boundary of the Baker domain. We are also grateful to the referee for the careful reading and the suggestions that made the paper much better.

\section{Preliminaries}\label{sect-2-prelim}

\subsection*{Inner function associated to a Fatou component}
Let $ U $ be an invariant Fatou component of a transcendental entire function $ f\colon\mathbb{C}\to\mathbb{C} $. Such component is always simply connected \cite{baker1984}, so one may consider a Riemann map $ \varphi\colon \mathbb{D}\to U $. Then, \[
g\colon\mathbb{D}\longrightarrow\mathbb{D},\hspace{1cm}g\coloneqq \varphi^{-1}\circ f\circ\varphi.
\]
is an {\em inner function}, i.e. a holomorphic self-map of the unit disk $ \mathbb{D}$ such that, for almost every $ \theta\in \left[ 0, 2\pi\right)  $, the radial limit $ g^*(e^{i\theta}) $ belongs to $\partial \mathbb{D}$ (see e.g. \cite[Sect. 2.3]{efjs}). Then, $ g $ is called the \textit{inner function associated} to $ U $. Although $ g $ depends on the choice of $ \varphi $, inner functions associated to the same Fatou components are conformally conjugate, so we can ignore the dependence on the Riemann map. 

The dynamics of $ g_{|\mathbb{D}} $  are completely described by the results of Denjoy, Wolff and Cowen, being valid not only for inner functions, but for any holomorphic self-map of $ \mathbb{D} $. The Denjoy-Wolff Theorem describes the asymptotic behaviour of iterates (see e.g. \cite[Thm. IV.3.1.]{carlesongamelin}).

\begin{thm}{\bf (Denjoy-Wolff)}\label{teo-dw}
	Let $ g$ be a holomorphic self-map of  $ \mathbb{D} $, not conjugate to a rotation. Then, there exists $ p\in\overline{\mathbb{D}} $, such that for all $ z\in\mathbb{D} $, $ g^n(z)\rightarrow p $.
	The point $ p $ is called the {\em Denjoy-Wolff point} of $ g $.
\end{thm}

For a more precise description of the dynamics, the concepts of {\em fundamental set} and {\em absorbing domain} are needed.
Although sometimes both notions are used interchangeably, we shall make a distinction between them.
	
\begin{defi}{\bf (Absorbing domain)}
	A domain $ V\subset U $ is said to be an \textit{absorbing domain} for $ f $ in $ U $ if  $ f(V)\subset V $ and for every compact set $ K\subset U $ there exists $ n\geq 0 $ such that $ f^n(K)\subset V $.
\end{defi}
\begin{defi}{\bf (Fundamental set)}
	Let $ U $ be a domain in $ \mathbb{C} $ and let $ f\colon U \to U $ be a holomorphic map.
	An absorbing domain $ V\subset U $ is said to be a \textit{fundamental set} for $ f $ in $ U $ if it is  simply connected and $ f_{|V} $ is univalent.
\end{defi}

Clearly, fundamental sets are absorbing domains, but the converse is not true. The existence of fundamental sets (and, therefore, of absorbing domains) is ensured by Cowen's theorem. Moreover,  his results leads to a classification of the self-maps of $ \mathbb{D} $ having the Denjoy-Wolff point in $\partial \mathbb{D} $  in terms of the dynamics in the fundamental sets.
\begin{thm}{\bf (Cowen's classification of self-maps of $ \mathbb{D} $, {\normalfont \cite{cowen}})}\label{teo-cowen}
	Let $ g$ be a holomorphic self-map of  $ \mathbb{D} $ with Denjoy-Wolff point $ p\in\partial \mathbb{D} $. Then, there exists a set $ V\subset\mathbb{D} $, a domain $ \Omega $ equal to $ \mathbb{C} $ or $ \mathbb{H}=\left\lbrace \textrm{\em Re }z>0\right\rbrace  $, a holomorphic map $ \psi\colon \mathbb{D}\to\Omega  $, and a Möbius transformation $ T\colon\Omega\to\Omega $, such that: \begin{enumerate}[label={\em (\alph*)}]
		\item $ V $ is a fundamental set for $ g $ in $ \mathbb{D} $,
		\item $ \psi(V) $ is a fundamental set for $ T $ in $ \Omega $,
		\item $ \psi\circ g=T\circ \psi $ in $ \mathbb{D} $,
		\item $ \psi $ is univalent in $ V $.
	\end{enumerate}
	
	Moreover, up to a conjugacy of $ T $ by a Möbius transformation preserving $ \Omega $, one of the following three cases holds:\begin{itemize}
		\item $ \Omega= \mathbb{C} $, $ T=\textrm{id}_\mathbb{C} +1 $ {\em (doubly parabolic type)},
				\item $ \Omega= \mathbb{H} $, $ T=\lambda\textrm{id}_\mathbb{H}$, for some $ \lambda>1 $ {\em (hyperbolic type)},
		\item $ \Omega= \mathbb{H} $, $ T=\textrm{id}_\mathbb{H} \pm1 $ {\em (simply parabolic type)}.
	\end{itemize}
\end{thm}

In view of this theorem, we say that a Baker domain $ U $ (or $ f_{|U} $) is of \textit{doubly parabolic}, \textit{hyperbolic} or \textit{simply parabolic} type if the same holds for the associated inner function.

Assuming that $ f $ has finite degree on $ U $, as in our example, $ g $ is always finite Blaschke product, so $ g $ is well-defined and holomorphic in $ \partial \mathbb{D} $, and $ g(\partial\mathbb{D})=\partial\mathbb{D} $. 

When the degree of $ f_{|U} $ is infinite, the associated inner function does not extend holomorphically to all points in $ \partial \mathbb{D} $, which complicates considerably the definition of a dynamical system in $ \partial\mathbb{D} $. Since this is not the case of our example, we shall skip the details and refer to \cite{doering-mañé} for a wide exposition on the topic.

\subsection*{Behaviour of the Riemann map}
Let $ U\subsetneq\mathbb{C} $ be a simply connected domain and let $ \varphi\colon\mathbb{D}\to U $ be a Riemann map. The behaviour of the Riemann map in the boundary plays a crucial role for describing $ \widehat{\partial} U $, and hence $ \partial U $. Here we state some basic definitions and results which we use. A general exposition on the topic can be found in \cite{pommerenke}, \cite[Sect. 17]{milnor}.

By Carathéodory's Theorem, $ \varphi $ extends continuously to $ \overline{\mathbb{D}} $ if and only if $ \widehat{\partial} U $ is locally connected. When this is not the case, radial limits, radial cluster sets and cluster sets replace the notion of image for points in $ \partial \mathbb{D}$, and are the key tools to study the behaviour of  $ \varphi $ in $ \partial\mathbb{D} $.

\begin{defi}
	Let $ \varphi\colon \mathbb{D}\to U $ be a Riemann map and let $ e^{i\theta}\in\partial \mathbb{D} $. 
	\begin{itemize}
		\item The {\em radial limit} of $ \varphi $ at $ e^{i\theta} $ is defined to be $\varphi^* (e^{i\theta} )\coloneqq\lim\limits_{r\to 1^-}\varphi(re^{i\theta})$.
		\item The {\em radial cluster set} $ Cl_\rho(\varphi, e^{i\theta}) $ of $ \varphi $ at $ e^{i\theta} $ is defined as the set of values $ w\in\widehat{\mathbb{C}}$ for which there is an increasing sequence $ \left\lbrace t_n\right\rbrace _n \subset (0,1) $ such that $ t_n\to 1 $ and $\varphi (e^{i\theta}t_n)\to w $, as $ n\to\infty $.
		\item The {\em cluster set} $ Cl(\varphi, e^{i\theta}) $ of $ \varphi $ at $ e^{i\theta} $ is the set of values $ w\in\widehat{\mathbb{C}} $ for which there is a sequence $ \left\lbrace z_n\right\rbrace _n \subset\mathbb{D}$ such that $ z_n\to e^{i\theta} $ and $\varphi (z_n)\to w $, as $ n\to\infty $.
	\end{itemize}
\end{defi}

The well-known Fatou, Riesz and Riesz Theorem on radial limits (\cite[Thm. 17.4]{milnor}) states that the radial limit $ \varphi^*(e^{i\theta}) \in \widehat{\partial} U $ exists for Lebesgue almost every $ \theta $;  but, if we fix any particular point $ v\in \widehat{\partial} U $, then the set of $ e^{i\theta}\in\partial\mathbb{D} $ such that $ \varphi^*(e^{i\theta}) =v$ has Lebesgue measure zero.

The cluster set $ Cl(\varphi, e^{i\theta}) $ can be seen to be equivalent to the impression of the prime end of $ U $ corresponding to  $ e^{i\theta} $ by the Riemann map $ \varphi $ (\cite[Thm. 2.16]{pommerenke}). Therefore, we use both notions indistinguishably.

When $ \widehat{\partial} U $ is non-locally connected, accessible points and accesses play an important role, since not all points in $ \partial U $ can be reached from inside $ U $. They can be characterized by means of the Riemann map.
\begin{defi}{\bf (Accessible point)}
	Given an open subset $ U\subset\widehat{\mathbb{C}} $, 
	a point $ v\in\widehat{\partial} U $ is \textit{accessible} from $ U $ if there is a path $ \gamma\colon \left[ 0,1\right) \to U $ such that $ \lim\limits_{t\to 1} \gamma(t)=v $. We also say that $ \gamma $ \textit{lands} at $ v $.
\end{defi}

Moreover, we say that a curve $ \gamma\colon\left[ 0,1\right) \to U $ lands at $ +\infty $ (resp., $ -\infty $), if  $ \textrm{Re }\gamma(t) \to+\infty$ (resp. $ -\infty $), as $ t\to 1 $, and  $ \textrm{Im }\gamma (t) $ is bounded for $ t\in \left[ 0,1\right)  $.

\begin{defi}{\bf (Access)}
	Let $ z_0\in U $ and let $ v\in \widehat{\partial} U $ be an accessible point. A homotopy class (with fixed endpoints) of curves $ \gamma\colon \left[ 0,1\right] \to\widehat{\mathbb{C}} $ such that $ \gamma(\left[ 0,1\right))\subset U $, $ \gamma(0)=z_0 $ and $ \gamma(1)=v $ is called an \textit{access} from $ U $ to $ v $.
\end{defi}
\begin{thm}{\bf (Correspondence Theorem, \normalfont{\cite{bfjk15}}\bf )}\label{correspondence-theorem} Let $ U\subset\widehat{\mathbb{C}} $ be a simply-connected domain, $ \varphi\colon\mathbb{D}\to U $ a Riemann map, and let $ v\in\widehat{\partial} U $. Then, there is a one-to-one correspondence between accesses from $ U $ to $ v $ and the points $ e^{i\theta}\in \mathbb{D} $ such that $ \varphi^*(e^{i\theta})=v  $. The correspondence is given as follows.\begin{enumerate}[label={\em (\alph*)}]
		\item If $ \mathcal{A} $ is an access to $ v\in\widehat{\partial} U $, then there is a point $ e^{i\theta}\in\partial \mathbb{D} $ with $ \varphi^*(e^{i\theta})=v  $. Moreover, different accesses correspond to different points in $ \partial\mathbb{D} $.
		\item If, at a point $ e^{i\theta}\in\partial\mathbb{D} $, the radial limit $ \varphi^* $ exists and it is equal to $ v\in \widehat{\partial} U $, then there exists an access $ \mathcal{A} $ to $ v $. Moreover, for every curve $ \eta\subset \mathbb{D} $ landing at $ e^{i\theta} $, if $ \varphi(\eta) $ lands at some point $ w\in\widehat{\mathbb{C}} $, then $ w=v $ and $ \varphi(\eta)\in \mathcal{A} $.
	\end{enumerate}
\end{thm}

\subsection*{Harmonic measure}
The Riemann map $ \varphi\colon \mathbb{D}\to U $ induces a measure in $ \widehat{\partial} U $, the harmonic measure, which is the appropriate one when dealing with the boundaries of Fatou components. Indeed, we now define harmonic measure in $ \widehat{\partial} U $ in terms of the pullback under a Riemann map of the normalized measure on the unit circle $ \partial\mathbb{D} $, following the approach of \cite[Chapter 7]{mesuraharmonica}.
\begin{defi}{\bf (Harmonic measure)}
	Let $ U\subsetneq{\mathbb{C}} $ be a simply connected domain, $ z\in U $, and let $ \varphi\colon\mathbb{D}\to U $ be a Riemann map, such that $ \varphi(0)=z\in U $. Let $ (\partial\mathbb{D}, \mathcal{B}, \lambda) $ be the measure space on $ \partial \mathbb{D} $ defined by $ \mathcal{B} $, the Borel $ \sigma $-algebra of $ \partial \mathbb{D} $, and $ \lambda $, its normalized Lebesgue measure. Consider the measurable space $ (\widehat{\partial}U, \mathcal{B}_U) $, where  $ \mathcal{B}_U $ is the $ \sigma $-algebra defined as\[\mathcal{B}_U\coloneqq \left\lbrace B\subset\partial U  \colon (\varphi^*)^{-1}(B)\in\mathcal{B} \right\rbrace .\] Then, given $ B\in  \mathcal{B}_U$, the \textit{harmonic measure at $ z $ relative to $ U $} of the set $ B $ is defined as:\[\omega_U(z, B)\coloneqq\lambda ((\varphi^*)^{-1}(B)).\]
\end{defi}

We note that the $ \sigma $-algebra $ \mathcal{B}_U $ defined in $ \widehat{\partial} U$ does not depend on the chosen Riemann map $ \varphi $ nor on the base point $ z $. Indeed, any two Riemann maps $ \varphi_1,\varphi_2\colon\mathbb{D}\to U $ are equal up to precomposition with an automorphism of $ \mathbb{D} $, and automorphisms of $ \mathbb{D} $ send Borel sets of $\partial \mathbb{D}  $ to Borel sets of $\partial \mathbb{D}  $. Hence, if for some Riemann map $ \varphi\colon\mathbb{D}\to U $, it holds $ (\varphi^*)^{-1}(B)\in\mathcal{B} $, then it also holds for all Riemann maps $ \varphi\colon\mathbb{D}\to U $.

We also note that the definition of $ \omega_U(z, \cdot) $ is independent of the choice of $ \varphi$, provided it satisfies $ \varphi(0)=z $, since $ \lambda $ is invariant under rotation.

We refer to \cite{harmonicmeasure2, pommerenke, mesuraharmonica} for equivalent definitions and further properties of the harmonic measure. We only need the following simple fact.

\begin{lemma}{\bf (Sets of zero and full harmonic measure)}
	Let $ U\subsetneq{\mathbb{C}} $ be a simply connected domain. Consider the measure space $ (\widehat{\partial} U, \mathcal{B}_U) $ defined above, and let $ B\in  \mathcal{B}_U$. If there exists $ z_0\in U $ such that $ \omega_U(z_0, B)=0 $ (resp. $ \omega_U(z_0, B)=1 $), then $ \omega_U(z, B)=0 $ (resp. $ \omega_U(z, B)=1 $) for all $ z\in U $.
	In this case, we say that the set $ B $ has {\em zero} (resp. {\em full}) {\em harmonic measure relative to $ U $}, and we write $ \omega_U(B)=0 $ (resp. $ \omega_U(B)=1 $).
\end{lemma}

Finally, we note that, by the Fatou, Riesz and Riesz Theorem stated above, it holds $ \omega_U(\left\lbrace \infty\right\rbrace )=0 $. Hence, for every measurable set $ B\subset\widehat{\partial} U $ and $ z\in U $, we have \[\omega_U(z,B)=\omega_U(z, B\cap\mathbb{C}).\]
\subsection*{Dynamics on the boundary of Baker domains}
For transcendental entire functions, \textit{Baker domains} are defined as periodic Fatou components in which iterates converge uniformly towards infinity, the essential singularity of the function. Such Fatou components have been widely studied
\cite{rippon-stallard-familiesbakersI, rippon-stallard-familiesbakersII, fh,konig, rippon-bakerdomainsmeromorphic,rippon-bakerdomains, bfjk15-absorbing, bfjk19}, although here we only state the results we need to deal with our example.

Without loss of generality, let us assume that $ U $ is an invariant Baker domain. 
Since all points in  $ U $ escape to $ \infty $ under iteration, $ U $ is clearly unbounded and infinity is accessible from it. Indeed, given any point $ z\in U $ and a curve joining $ z $ and $ f(z) $ within $ U $, then the curve $ \gamma\coloneqq\cup_{n\geq0} f^n(\gamma) $ is unbounded and lands at infinity, defining an access which is called the {\em dynamical access to infinity}. 

\begin{remark}{\bf (Classification of Baker domains)}\label{remark-bakers}
	Cowen's classification (Thm. \ref{teo-cowen}) implies that dynamics inside a Baker domain can be eventually conjugate to $ \textrm{id}_\mathbb{C}+1 $, or to $ \lambda\textrm{id}_\mathbb{H} $, $ \lambda>1 $, or to $ \textrm{id}_\mathbb{H}+1 $. Indeed, all types are possible (\cite{konig}) and thus, this gives a classification of Baker domains. Let us remark that this is not the case for parabolic basins, which can be proved to be always of doubly parabolic type using extended Fatou coordinates (see e.g. \cite[Sect. 10]{milnor}).
\end{remark}

The following results describe the boundary of Baker domains, both from a topological and dynamical points of view.
\begin{thm}{\bf (Boundary of doubly parabolic Baker domains, {\normalfont \cite{baker-dominguez, bargmann}})}
Let $ f\colon\mathbb{C}\to\mathbb{C} $ and $ U $ be a Baker domain of $ f $ of doubly parabolic type. Let $ \varphi\colon\mathbb{D}\to U $ be a Riemann map. Then, \[\overline{\left\lbrace e^{i\theta}\colon \varphi^*(e^{i\theta})=\infty\right\rbrace} =\partial  \mathbb{D}.\] In particular, $ \partial U $ is non-locally connected.
\end{thm}
\begin{thm}{\bf (Dynamics on the boundary of Baker domains, {\normalfont \cite{rippon-stallard, bfjk19}})}\label{teo-dynamics-boundary-baker}
Let $ f\colon \mathbb{C}\to\mathbb{C} $ be a transcendental entire function and $ U $ be a Baker domain of $ f $, such that $ f_{|U} $ has finite degree. Then, the following holds.
\begin{enumerate}[label={\em (\alph*)}]
	\item If $ U $ is hyperbolic or simply parabolic, then $ \mathcal{I}(f)\cap\partial U $ (the set of escaping points in  $\partial U$) has full harmonic measure.
		\item If $ U $ is doubly parabolic type, then  the set of points in $ \partial U $ whose orbit is dense in $ \partial U $ has full harmonic measure. In particular, $ \mathcal{I}(f)\cap\partial U $ has zero harmonic measure.
\end{enumerate}
\end{thm}
\subsection*{Indecomposable continua}
Finally, we include the definition of indecomposable continuum and the following result, which gives a sufficient condition for the accumulation set of a curve to be an indecomposable continuum. Here, we shall understand simple curve as the continuous, one-to-one image of the  non-negative real numbers.
\begin{defi}{\bf (Indecomposable continuum)}
	We say that $ X\subset \widehat{\mathbb{C}} $ is a {\em continuum } if it is compact and connected.
	A continuum is {\em indecomposable} if it cannot be expressed as the union of two
proper subcontinua.
\end{defi}
\begin{thm}{\bf (Curry, {\normalfont \cite[Thm. 8]{curry}})}\label{teo-curry}
	Let $ X $ be a one-dimensional non-separating plane continuum which is the closure of a simple curve  that limits upon itself. Then $ X $ is indecomposable.
\end{thm}

\section{Basic properties of the dynamics of $ f $}\label{sec-3-dynamics-f}

In this section we gather some of the properties of the function $ f(z)=z+e^{-z}$, as well as its dynamics, which are used recurrently during the proofs of the main theorems. First, we include a quick description of the general dynamics of $ f $, summarizing the ideas of \cite{baker-dominguez, fh}. From there, it will be deduced that only the study of $ f $ on the strip $ S\coloneqq \left\lbrace z\in\mathbb{C}\colon \left| \textrm{Im }z\right| \leq\pi\right\rbrace  $ is needed.

\subsection*{General dynamics of $ f $}

To give a first approach to the dynamics, one may consider the semiconjugacy $ w=e^{-z}$ between $ f(z)=z+e^{-z} $ and $ h(w)=we^{-w} $. Observe that $ w=0 $ is a fixed point of multiplier 1, and $ h(w)=w-w^2+\mathcal{O}(w^3) $ near 0, implying that 0 is a parabolic fixed point having one attracting and one repelling direction. From the fact that $ \mathbb{R} $ is invariant by $ h $ and from the action of $ h $ in $ \mathbb{R} $, it is deduced that the repelling direction is $ \mathbb{R}_- $, which belongs to the Julia set $ \mathcal{J}(h) $, and the attracting direction is $ \mathbb{R}_+$, which belongs to the immediate parabolic basin of 0, and hence to the Fatou set $ \mathcal{F}(h) $. See Figure \ref{fig-h}. We denote by $ \mathcal{A}_0 $ the immediate basin of 0.

We note that all preimages of $ \mathbb{R}_-$ are in the Julia set and, since 0 is an asymptotic value, they separate the plane into infinitely many components. It follows that the Fatou set $ \mathcal{F}(h) $ has infinitely many connected components.

		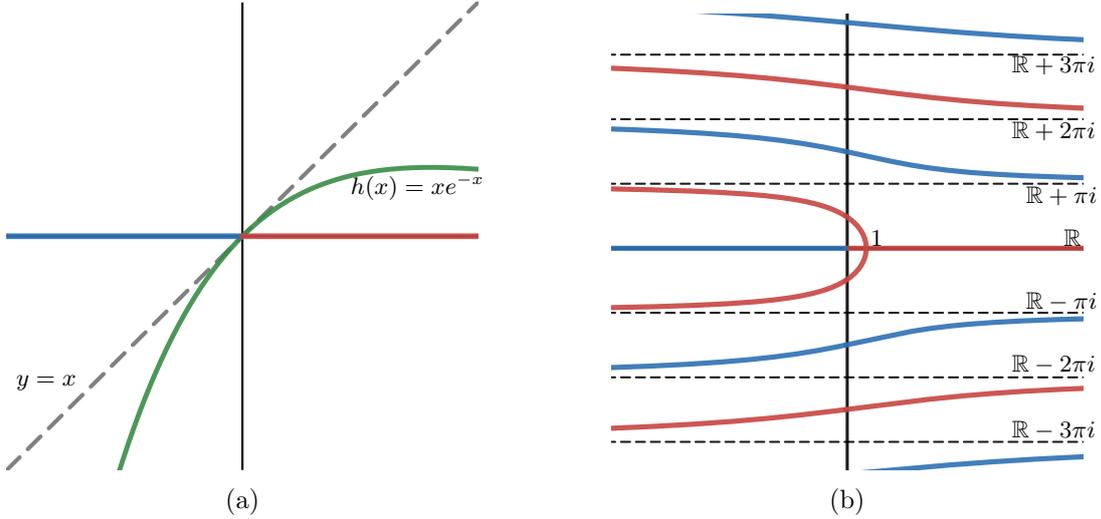
\begin{figure}
		\centering
		\begin{subfigure}[b]{0.4\textwidth}
			\centering
			\includegraphics[width=\textwidth]{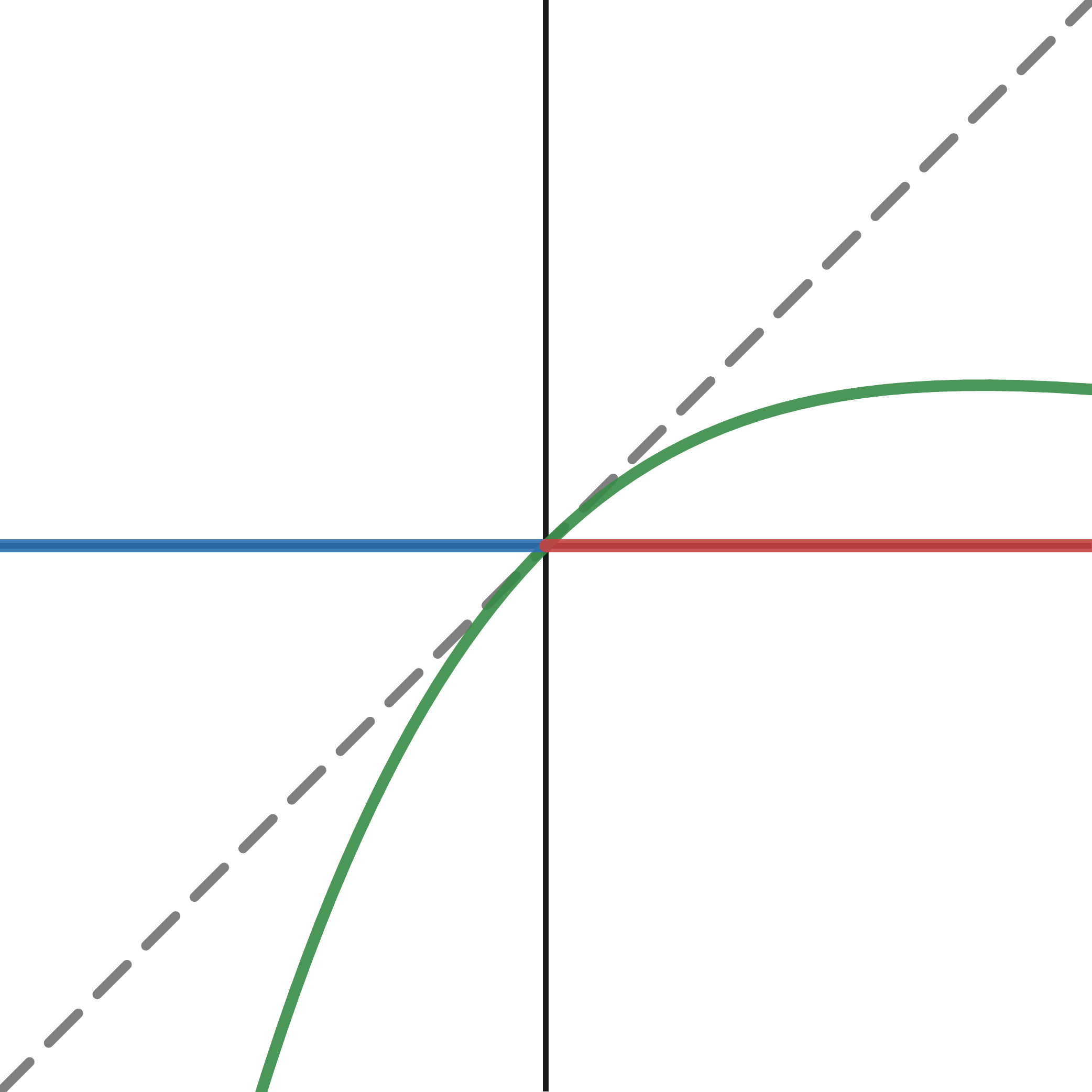}
			\setlength{\unitlength}{\textwidth}
			\put(-0.98, 0.18){\footnotesize$ y=x $}
			\put(-0.27, 0.59){\footnotesize$ h(x)=xe^{-x}$}
			\caption{}
			\label{fig1}
		\end{subfigure}
		\hspace{1.5cm}
		\begin{subfigure}[b]{0.4\textwidth}
			\centering
			\includegraphics[width=\textwidth]{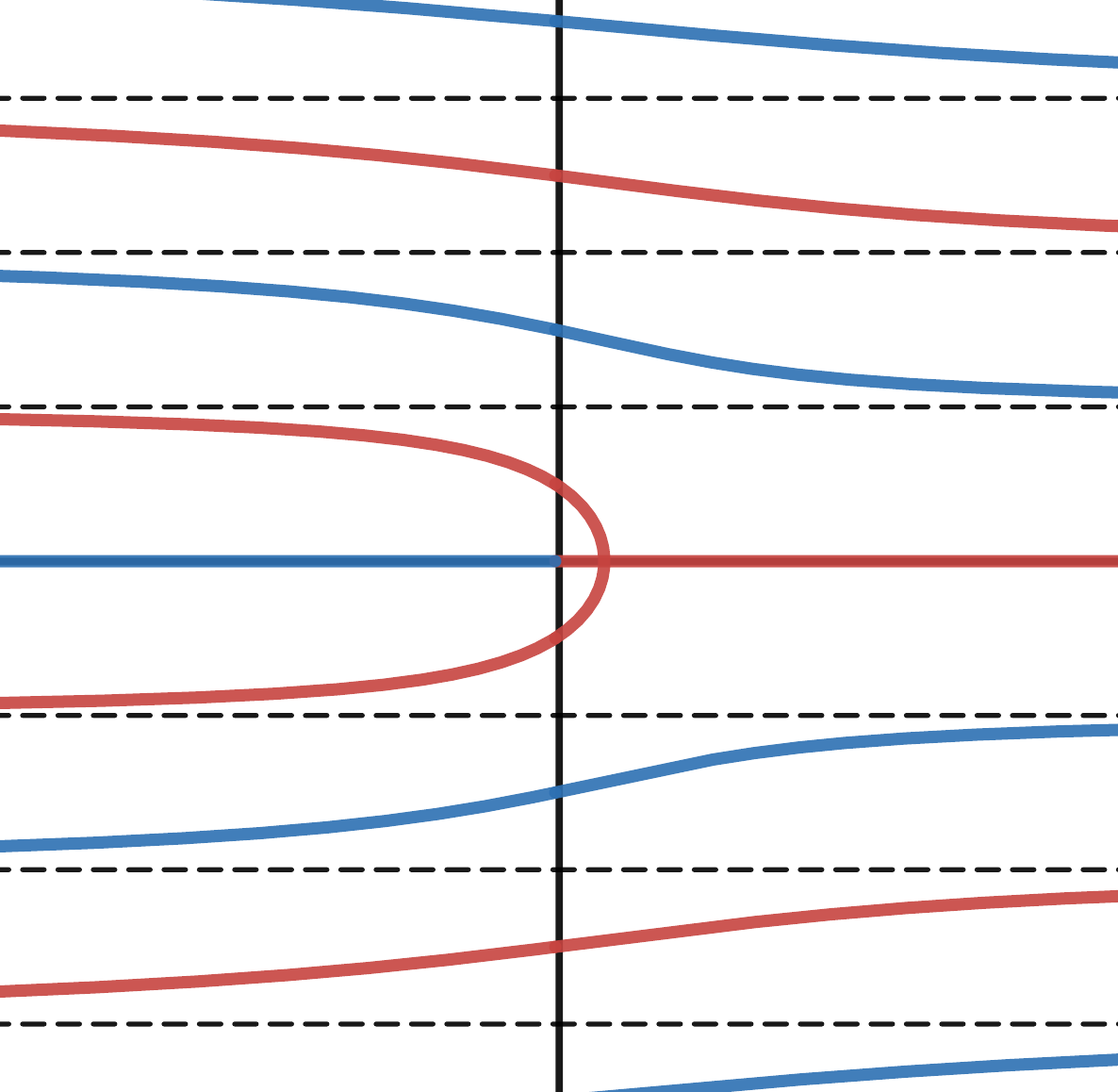}
			\setlength{\unitlength}{\textwidth}
			\put(-0.04, 0.48){\footnotesize$\mathbb{R}$}
			\put(-0.12, 0.57){\footnotesize$\mathbb{R}+\pi i$}
			\put(-0.12, 0.345){\footnotesize$\mathbb{R}-\pi i$}
			\put(-0.15, 0.21){\footnotesize$\mathbb{R}-2\pi i$}
			\put(-0.15, 0.07){\footnotesize$\mathbb{R}-3\pi i$}
			\put(-0.15, 0.71){\footnotesize$\mathbb{R}+2\pi i$}
			\put(-0.15, 0.85){\footnotesize$\mathbb{R}+3\pi i$}
			\put(-0.45, 0.48){\footnotesize$1$}
			\caption{}
			\label{fig2}
		\end{subfigure}

		\caption{\footnotesize In the left, Figure \ref{fig1} shows the plot of the real function $ h(x) = xe^{-x} $
			(green), together with the diagonal $ y = x $ (grey dotted line). The point 
			$ x=0 $ is a parabolic fixed point. Points in $ \mathbb{R}_+ $ (red) are attracted to 0, so $ \mathbb{R}_+\subset \mathcal{A}_0\subset\mathcal{F}(h) $, while points in $ \mathbb{R}_- $ (blue) converge to $ -\infty $ exponentially
			fast and $ \mathbb{R}_-\subset \mathcal{J}(h) $.  In the right, Figure \ref{fig2} shows in red the preimages of the positive real line $ \mathbb{R}_+ $; and, in blue, the preimages of the negative real
			line $ \mathbb{R}_- $. By the invariance of the Fatou and the Julia sets, all red lines are contained in the Fatou set, while
			the blue ones are in the Julia set. One deduces that the immediate parabolic basin $ \mathcal{A}_0 $ is contained in the region
			bounded by the two blue lines lying in the strips $ \left\lbrace \pi < y < 2\pi\right\rbrace  $ and $ \left\lbrace -2\pi < y < -\pi\right\rbrace  $ respectively.}\label{fig-h}
	\end{figure}

 It is not hard to see that the only two singular values of $ h $ are 0 and $ e^{-1} $, the latter being contained in the immediate parabolic basin $ \mathcal{A}_0 $. Therefore, the Fatou set $ \mathcal{F}(h) $ is precisely $ \mathcal{A}_0 $ and its preimages under $ h $. Indeed, since $ h $ has only a finite number of singular values, it cannot have Baker nor wandering domains (\cite[Sect. 5]{eremenko-lyubich}), and the presence of any other invariant Fatou component (either a basin or a Siegel disk) would require an additional singular value (see e.g. \cite[Thm. 7]{bergweiler}). Since there are infinitely many Fatou components for $ h $, $ \mathcal{A}_0 $ has infinitely many preimages, separated by the preimages of $ \mathbb{R}_- $. 

We lift these results to the dynamical plane of $ f $, using Bergweiler's result (\cite{bergweiler95}), which ensures that the Fatou and Julia sets of $ f $ and $ h $ are in correspondence under the projection $ w=e^{-z} $.  Preimages of $ \mathbb{R}_+ $ under $ e^{-z} $, which are precisely the forward invariant horizontal lines $ \left\lbrace \textrm{Im }z=2k\pi i\right\rbrace_{k\in\mathbb{Z}}  $, are in the Fatou set and their points escape to $ \infty $ to the right. Preimages of $ \mathbb{R}_- $ under the exponential projection are the forward invariant horizontal lines $ \left\lbrace \textrm{Im }z=(2k+1)\pi i\right\rbrace_{k\in\mathbb{Z}}  $, which are in the Julia set and whose points escape to $ -\infty $  exponentially fast. The horizontal strips $  S_k\coloneqq \left\lbrace (2k-1)\pi\leq\textrm{Im }z\leq(2k+1)\pi\right\rbrace $ contain a preimage $ U_k $ of $ \mathcal{A}_0 $ which, in turn, contains a preimage of $ \mathbb{R}_+ $ under $ e^{-z} $, that is $ \left\lbrace \textrm{Im }z=2k\pi i\right\rbrace  $. Such horizontal line is forward invariant, so this implies that $ U_k $ is forward invariant and iterates tend to $ \infty $, so $ U_k $ is a Baker domain. 

Moreover, we note that $ \mathcal{F}(f) $ is precisely the union of these Baker domains $ U_k $ and their preimages under $ f $. Indeed, any Fatou component $ V $ of $ f $ must project by $ w=e^{-z} $ to a preimage of $ \mathcal{A}_0 $, implying that $ V $ is mapped to some $ U_k $ in a finite number of steps. Hence, the presence of wandering domains is ruled out.

Finally, we note that the function $ f $ satisfies the relation $ f(z+2k\pi i)= f(z)+2k\pi i $, for all $ z\in\mathbb{C} $ and $ k\in\mathbb{Z} $, so it is enough to study it in the central strip $ S\coloneqq S_0 $ and the corresponding Baker domain $ U\coloneqq U_0 $. To do so, 
we consider the conformal branch of the semiconjugacy $ w=e^{-z} $, defined on $ \textrm{Int} S $, i.e.\[E(z)\coloneqq e^{-z}\colon\textrm{Int} S\longrightarrow \mathbb{C}\smallsetminus\mathbb{R}_-,\]\[E^{-1}(w)\coloneqq -\textrm{Log} (w)\colon\mathbb{C}\smallsetminus\mathbb{R}_-\longrightarrow \textrm{Int} S,\]where $ \textrm{Log}\colon\mathbb{C}\smallsetminus\mathbb{R}_-\to \textrm{Int} S $ denotes the principal branch of the logarithm. Since $ U\subset \textrm{Int } S $ and $ \mathcal{A}_0\subset \mathbb{C}\smallsetminus\mathbb{R}_- $, this gives a conformal conjugacy between $ f_{|U} $ and $ h_{|\mathcal{A}_0} $. Hence, we deduce that the Baker domain $ U $ is of doubly parabolic type and of degree two.

		\begin{figure}[htb!]\centering
	\includegraphics[width=15cm]{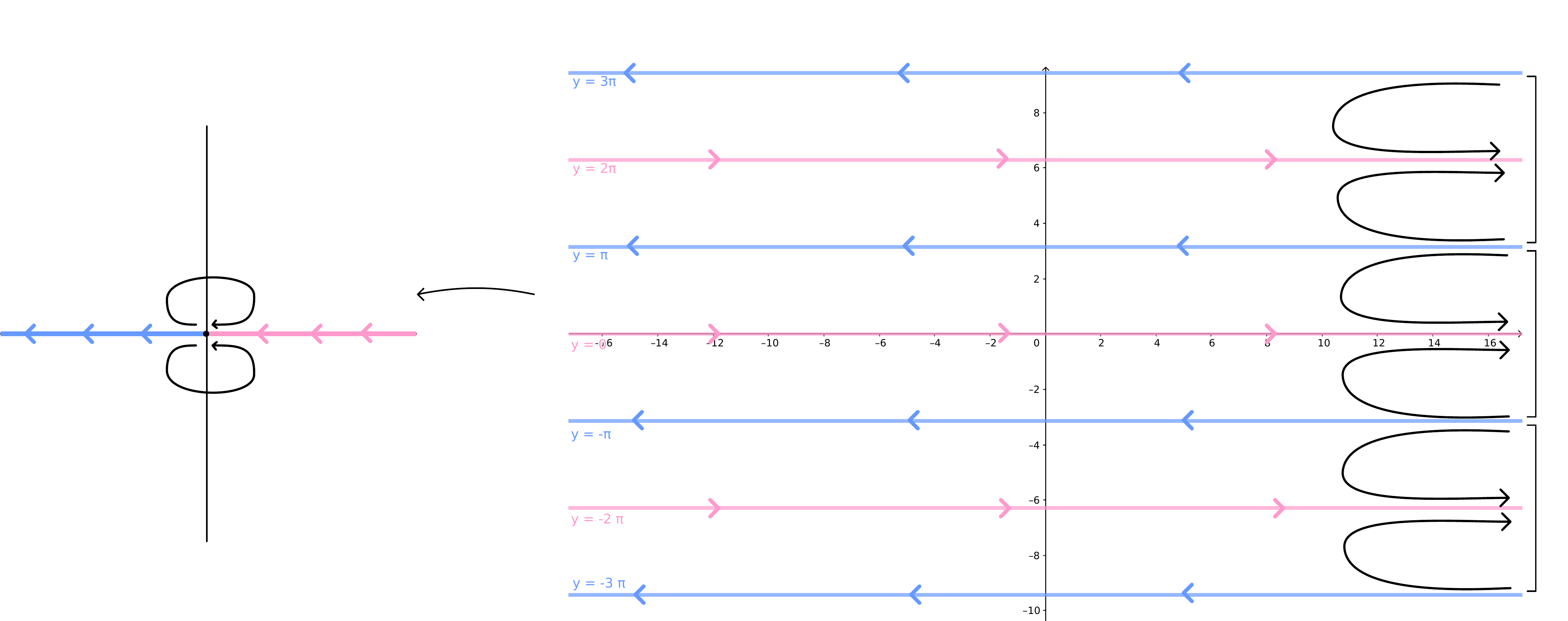}
	\setlength{\unitlength}{15cm}
	\put(-0.73, 0.22){\footnotesize $ w=e^{-z} $}
	\put(-0.94, 0.33){\footnotesize $ h(w)=we^{-w} $}
	\put(-0.4, 0.37){\footnotesize $ f(z)=z+e^{-z} $}
	\put(-0.015, 0.18){\footnotesize $  S_0 $}
	\put(-0.015, 0.29){\footnotesize $ S_1 $}
	\put(-0.015, 0.065){\footnotesize $  S_{-1} $}
	\caption{\footnotesize Schematic representation of the dynamics of $ h $ and $ f $ and how the exponential projection $ w=e^{-z} $ relates both of them. In the left, $ \mathbb{R}_+$ (in pink) is contained in the immediate parabolic basin $ \mathcal{A}_0 $. Its preimages by $ w=e^{-z} $, the lines $ \left\lbrace \textrm{Im }z=2k\pi i\right\rbrace_{k\in\mathbb{Z}}  $ (also in pink), lie each of them in a Baker domain $ U_k $. In blue, in the left there is $ \mathbb{R}_- \subset\mathcal{J}(h)$. Its preimages $ \left\lbrace \textrm{Im }z=(2k+1)\pi i\right\rbrace_{k\in\mathbb{Z}}  $ lie in $ \mathcal{J}(f) $ and separate the plane into the strips $ S_k $. }
\end{figure}
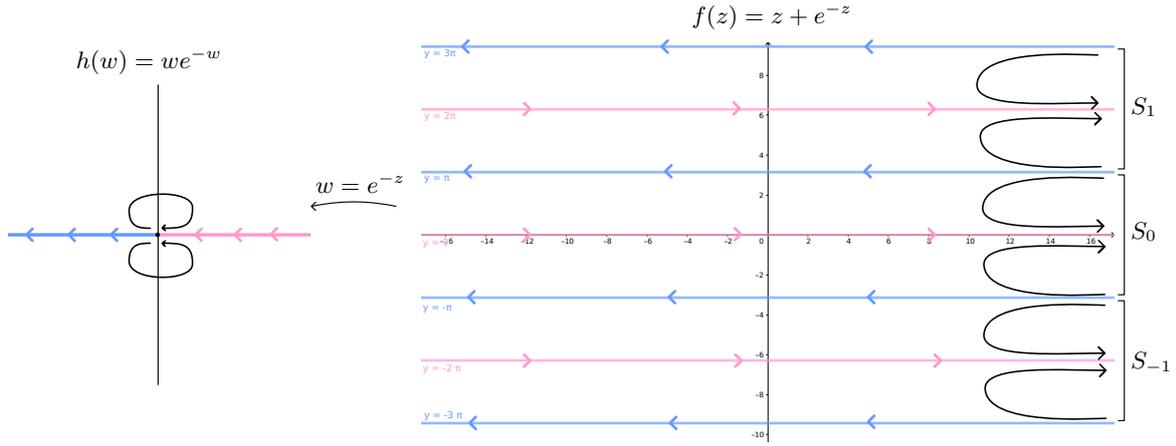
\begin{remark}
	Although working with the function $ h $ may seem easier, for having a finite number of singular values and being  postsingularly bounded, the fact that one asymptotic value lies in the Julia set reduces this advantage. In general, we shall work with $ f $, its logarithmic lift.
\end{remark}

\subsection*{Action of $ f $ in the strip $ S $}
As seen before, it is enough to consider $ f $ in the strip $ S=\left\lbrace z\in\mathbb{C}\colon \left| \textrm{Im }z\right| \leq \pi\right\rbrace  $, delimited by the horizontal lines $L^\pm\coloneqq\left\lbrace z\colon \textrm{Im }z=\pm \pi\right\rbrace $. See Figure \ref{fig-absorbing}.
	 \begin{figure}[htb]\centering
	\includegraphics[width=12cm]{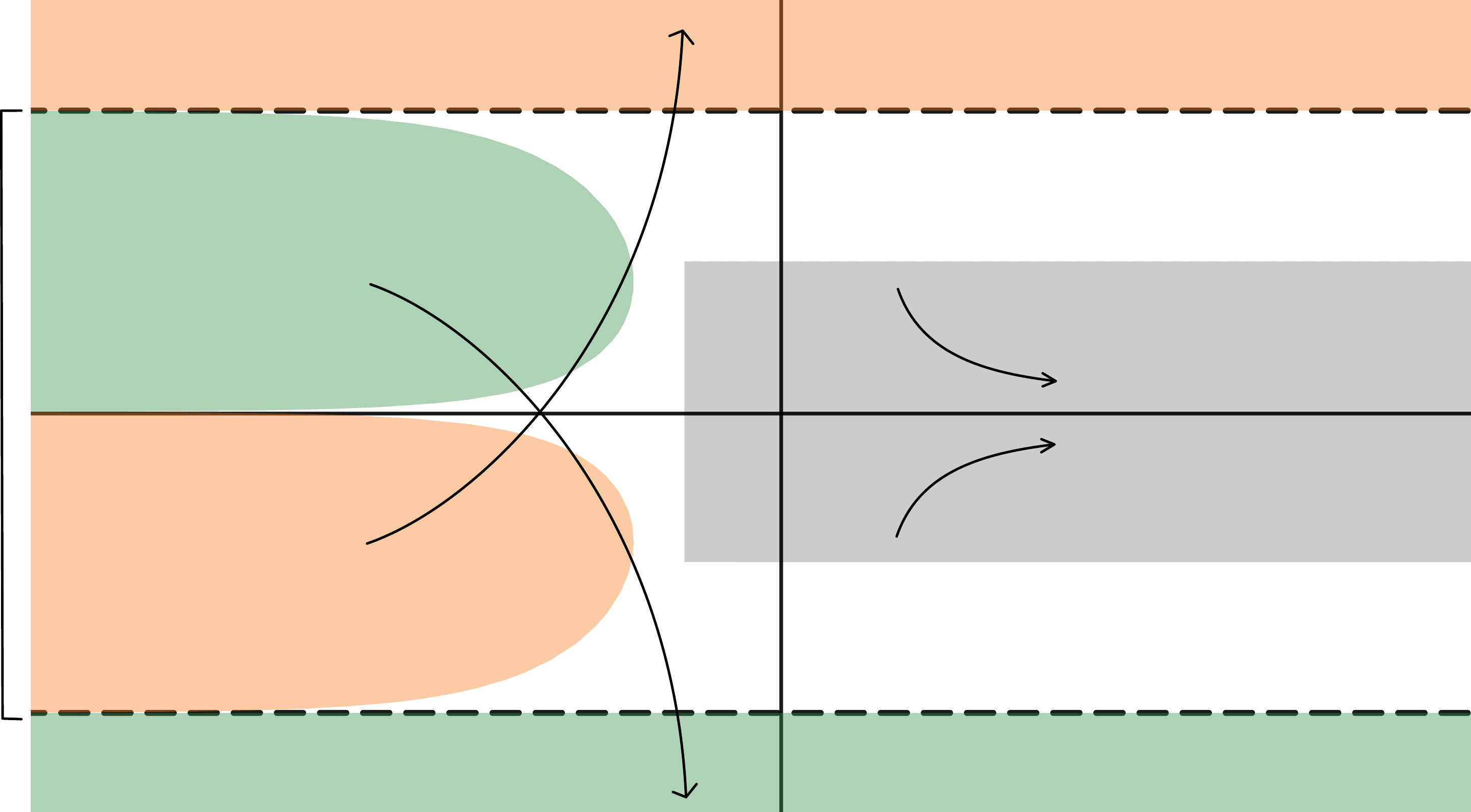}
	\setlength{\unitlength}{12cm}
	\put(-0.1, 0.32){\footnotesize $ V $}
	\put(-1.05, 0.28){\footnotesize $ S $}
	\put(0, 0.46){\footnotesize $ L^+ $}
	\put(-0.01, 0.06){ \footnotesize $ L^- $}
	\caption{\footnotesize Schematic representation of how $ f $ acts on the strip $ S $ and of the absorbing domain $ V $.}	\label{fig-absorbing}
\end{figure}
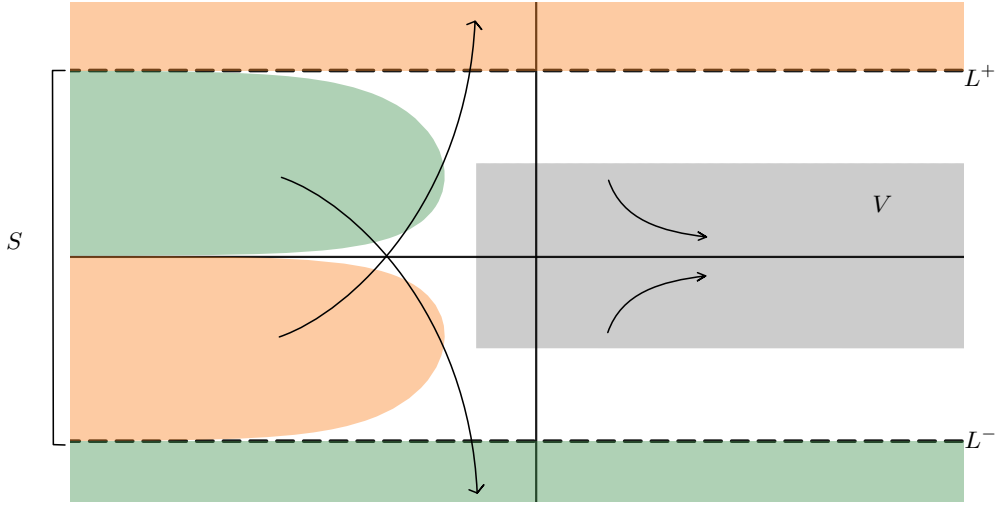

Observe that, to the left, $ f $ behaves like the exponential and, to the right, like the identity. Moreover, if one writes $ f $ as\[f(x,y)=(x+e^{-x}\cos y,y-e^{-x}\sin y),\]preimages of $ L^\pm $ can be computed explicitly as the curves of the form $ \left\lbrace y-e^{-x}\sin y=\pm \pi\right\rbrace  $. In $ S $ they consist precisely of two bent curves converging to $ -\infty $ in both ends, being asymptotic to $ \mathbb{R} $ and to $ L^\mp $ (see Fig. \ref{fig-absorbing}). The region delimited by these curves is mapped outside $ S $ in a one-to-one fashion.  On the other hand, the map $ f\colon f^{-1}(S)\cap S\to S $ is a proper map of degree two, which can be deduced for instance by computing the preimages of $ \mathbb{R} $ in $ S $.

Next, we define the set\[\widehat{S}\coloneqq \left\lbrace z\in S\colon f^n(z)\in S\textrm{, for all }n\right\rbrace .\]
Clearly, $ U\subset \widehat{S} $, since $ U $ is forward invariant under $ f $. Moreover, since both $ f\colon f^{-1}(S)\cap S\to S $ and $ f_{|U} $ have degree 2, there cannot be preimages  of  $ U $ in $ S $ other than itself.  Therefore, $ \mathcal{F}(f)\cap\widehat{S}=U $. On the other hand,  $\partial U\subset \mathcal{J}(f)\cap\widehat{S} $. The other inclusion, which is going to be proved in Proposition \ref{prop-caracteritzacio-frontera}, cannot be  claimed directly  to be true, for the possible existence of buried points in $ \widehat{S} $, i.e. points in $ \mathcal{J}(f) $ which are not eventually mapped to the boundary of any Baker domain $ U_k $.
\subsection*{Absorbing domains and expansion of $ f $}

Let us define the following set 
\[ V\coloneqq \left\lbrace z\in S\colon \textrm{Re }z>-1, \left| \textrm{Im }z\right| <\frac{\pi}{2} \right\rbrace.\]
\begin{lemma}
	The set $ V $ is an absorbing domain for $ f $ in $ U $.
\end{lemma}
\begin{proof}
	Clearly, $ V $ is open and connected. For the forward invariance, consider $ z=x+iy\in V $, so $ x>-1 $ and $ \left| y\right| <\frac{\pi}{2} $, then \[ \textrm{Re }f(x+iy)= x+e^{-x}\cos y>x>-1, \]
	\[ \left| \textrm{Im }f(x+iy)\right| <\left| \frac{\pi}{2}-e^{-x}\right| <\frac{\pi}{2}. \]
	Finally, the fact that $ V $ is absorbing, i.e. that all compact sets in $ U $ must eventually enter in $ V $, can be deduced from the dynamics on the conjugate parabolic basin $ \mathcal{A}_0 $.
	Indeed, $ E(V) $ is the following forward invariant set, \[E(V)=\left\lbrace w\in\mathbb{C}\colon \left| w\right| <\frac{1}{e}, \left| \textrm{arg }w\right| <\frac{\pi}{2}\right\rbrace .\] Observe that $ E(V) $ is an circular sector of angle $ \frac{\pi}{2} $ containing the real interval $ (0,e) $, which is in the attracting direction of the parabolic point $ w=0 $. Hence, $ E(V) $ is a parabolic petal (see e.g. \cite[p. 74]{Steinmetz+2011}), so all compact sets in $ \mathcal{A}_0 $ must eventually enter in $ E(V) $. Hence, applying back the conjugacy, we get that $ V $ is an absorbing domain for $ f $ in $ U $.
\end{proof}
\begin{remark} Since it contains the critical point 0, $ V $ is not a fundamental set. It can be turned into one making it smaller, for instance taking $ \left\lbrace z\in S\colon \textrm{Re }z>0, \left| \textrm{Im }z\right| <\frac{\pi}{2} \right\rbrace $. On the other hand, fundamental sets, and absorbing domains, can be chosen bigger, although we have no need to do that. In fact, using local theory around parabolic fixed points (see e.g. \cite[p. 74]{Steinmetz+2011}), there exist fundamental sets which approach tangentially $ L^\pm $. 
\end{remark}

One of the advantages of choosing $ V $ as we have done is that the map is expanding outside it (although not uniformly expanding). Indeed, a simple computation yields: \[f'(x+iy)=1-e^{-x}\cos y+i e^{-x}\sin y, \]\[\left| f'(x+iy)\right| =\sqrt{1+e^{-2x}-2e^{-x}\cos y}.\] Therefore, $ \left| f'(x+iy)\right| >1 $ if and only if $ e^{-x}-2\cos y>0 $. This last inequality is satisfied if $ \frac{\pi}{2}<\left| y\right| <\pi $ or if $ x<-1 $. Therefore, $ \left| f'(z)\right| >1 $ for all $ z\in S\smallsetminus \overline{V} $. 

Since $ S\smallsetminus \overline{V}$ is not convex, in order to apply the expansion of $ f $ as an augmenter of the distance between points, we need to consider a more appropriate distance than the Euclidean one. To this aim, we define the following metric. 

\begin{defi}{\bf($ \rho $-distance in $ S\smallsetminus\overline{V} $)}\label{def-rho-distance}
	Given $ z,w\in S\smallsetminus\overline{V} $, let us define its $ \rho $-distance as: \[\rho (z,w)\coloneqq\inf l (\gamma),\] where the infimum is taken over all paths $ \gamma\subset S\smallsetminus\overline{V} $ with endpoints $ z $ and $ w $, and $ l $ denotes the length of the path with respect to the Euclidean metric.
	
	Given a set $ K \subset S\smallsetminus\overline{V}$, we denote by $ \textrm{diam}_\rho (K) $ the diameter of $ K $  with respect to the $ \rho $-distance, i.e.\[\textrm{diam}_\rho (K) =\sup\limits_{x,y\in K} \rho(x,y).\]
\end{defi}

Observe that the Euclidean distance is always smaller than the $ \rho $-distance, i.e. 
\[\left| z-w\right| \leq \rho(z,w), \hspace{0.3cm}\textrm{for all }z,w\in S\smallsetminus\overline{V} , \] with equality if both $ z $ and $ w $ are contained in a convex subset of $ S\smallsetminus\overline{V} $.

Notice also that the $ \rho $-distance between two points can be arbitrarily large, although the Euclidean distance between them remains bounded. However, we are going to restrict the use of the $ \rho $-distance to particular subsets of $ S\smallsetminus\overline{V} $, where we do have an upper bound for the $ \rho $-distance in terms of the Euclidean one (see Lemma \ref{lemma-properties-Sij}).

\begin{remark} Let us observe that, instead of considering the dynamical system defined by $ f $ in $ \mathbb{C} $, we can restrict to the one given by $ f $ in $ \widehat{S} $. 
	For it we have a similar situation that the one for $ \lambda e^{z} $, $ 0<\lambda<\frac{1}{e} $, in \cite{goldberg-devaney}, and the corresponding generalization in \cite{baranski, baranski-karpinska}: a unique Fatou component which contains the postsingular set. Mainly, two things distinguish our situation from theirs. First, $ f $ in $ \widehat{S} $ has degree two, and the functions they are dealing with have infinite degree. Second, they have uniform expansion (at least in the logarithmic tracts), while our expansion is not uniform (compare with Prop. \ref{remark_expansion}). Hence, the results on next sections are meant to overcome this difficulty.
\end{remark}
\subsection*{Itineraries in $ \widehat{S} $ and symbolic dynamics}
Recall that $ f\colon f^{-1}(S)\cap S\to S $ has degree two and the critical value is 1. Therefore, the two branches of the inverse of $ f $ in $ S $, say $ \phi_0 $ and $ \phi_1 $, are well-defined in $ S\smallsetminus\left[ 1,+\infty\right)  $.
More precisely \[
\phi_0\colon S\smallsetminus\left[ 1,+\infty\right)\to \Omega_0\coloneqq S\cap\mathbb{H}_+, \]
\[
\phi_1\colon S\smallsetminus\left[ 1,+\infty\right)\to \Omega_1\coloneqq S\cap\mathbb{H}_- , \]
where $ \mathbb{H}_+ $ and $ \mathbb{H}_- $ denote the upper and the lower half plane, respectively (see Fig. \ref{fig-inverses}). 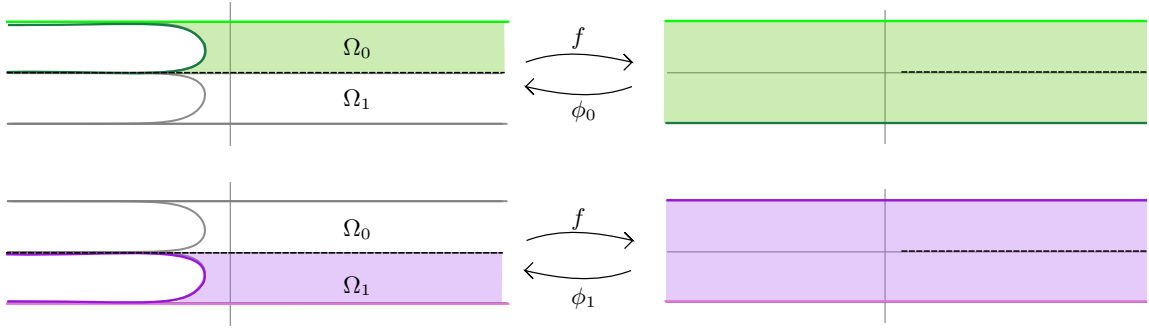
\begin{figure}[htb!]\centering
	\includegraphics[width=15cm]{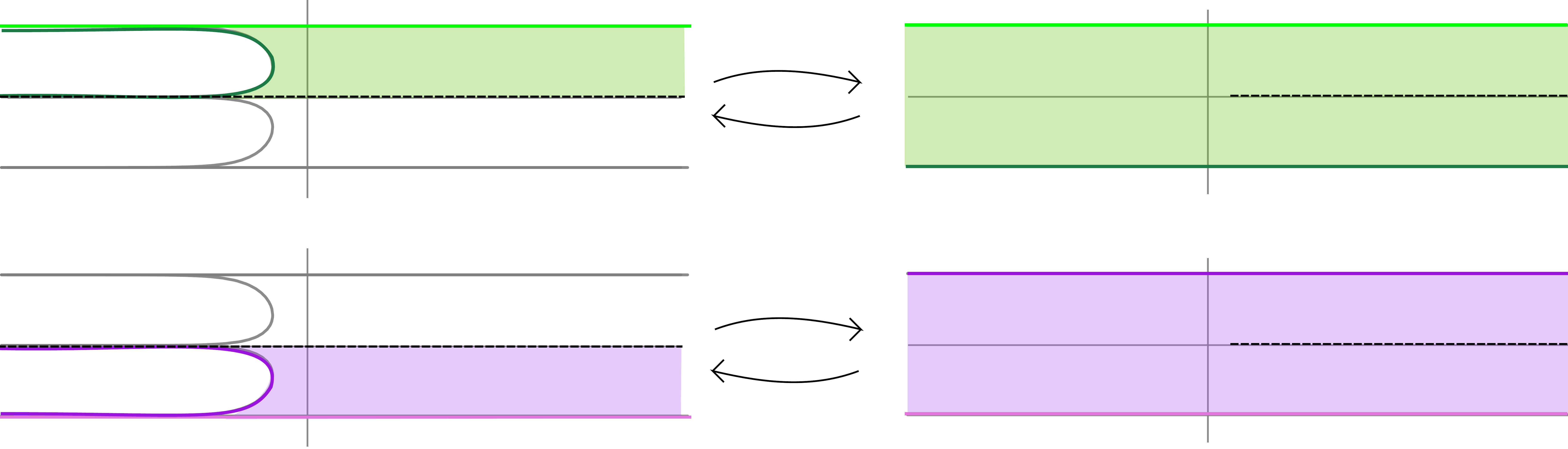}
	\setlength{\unitlength}{15cm}
	\put(-0.505, 0.25){\footnotesize $ f $}
	\put(-0.505, 0.185){\footnotesize $ \phi_0 $}
	\put(-0.505, 0.09){\footnotesize$ f $}
	\put(-0.505, 0.02){\footnotesize $ \phi_1 $}
		\put(-0.705, 0.24){\footnotesize $ \Omega_0 $}
	\put(-0.705, 0.195){\footnotesize $\Omega_1$}
	\put(-0.705, 0.08){\footnotesize$ \Omega_0 $}
	\put(-0.705, 0.03){\footnotesize $ \Omega_1 $}
	\caption{\footnotesize Action of the inverses $ \phi_0 $ and $ \phi_1 $ on the strip S.}	\label{fig-inverses}
\end{figure}

We claim that $ \phi_0 $ and $ \phi_1 $ do not increase the $ \rho $-distance between points, as shown in the following proposition.

\begin{prop}
	{\bf (Contraction and uniform contraction in $ S\smallsetminus \overline{V} $)}\label{remark_contract} 
		The following properties hold true. \begin{enumerate}[label={\em (\alph*)}]
		\item  Let $ z,w\in S\smallsetminus \overline{V} $. Then, for $ i\in\left\lbrace 0,1 \right\rbrace   $, \[\rho(\phi_i(z), \phi_i(w))\leq \rho(z,w).\]
		\item Let $ k\in \mathbb{R} $ and let $ S_k\coloneqq \left\lbrace z=x+iy\in S\smallsetminus \overline{V}\colon x\leq k\right\rbrace $. Then, there exists $ \lambda\coloneqq \lambda(k)<1 $ such that, if $ z,w\in S\smallsetminus \overline{V} $, then, for $ i\in\left\lbrace 0,1 \right\rbrace   $, \[\rho(\phi_i(z), \phi_i(w))\leq\lambda \rho(z,w).\]
		Moreover, if $ K\subset S_k $ is a compact set, then \[\textrm{\em diam}_\rho (\phi_i (K))\leq \lambda \textrm{\em diam}_\rho (K).\]
	\end{enumerate}
\end{prop}
\begin{proof}\begin{enumerate}[label={(\alph*)}]
		\item As observed above, it holds $ \left| f'(z)\right| >1 $ for all $ z\in S\smallsetminus \overline{V} $. Therefore, if $ \gamma \subset S\smallsetminus\overline{V} $ is a geodesic (in $ S\smallsetminus\overline{V}  $) joining $ z $ and $ w $, then $ \phi_i(\gamma) $ is a curve joining $ \phi_i(z) $ and $ \phi_i(w)$, and \[\rho(\phi_i(z), \phi_i(w))\leq \int_{\phi_i(\gamma)} ds=\int_\gamma\left|\phi_i'(s) \right| ds<\int_\gamma ds=\rho(z,w),\] as desired.
		
		\item We start by noticing that $ \left| f'\right|  $ is uniformly bounded in $  S\smallsetminus \overline{V}$. Indeed, on the one hand, for all $ z=x+iy $ with $ x\leq-1 $, it holds \[\left| f'(x+iy)\right| =\sqrt{1+e^{-2x}-2e^{-x}\cos y}\geq \sqrt{1+e^{2}-2e} >1.\]
		On the other hand, assuming $ k>-1 $ and $ -1<x<k $, necessarily $ \frac{\pi}{2}\leq \left| y\right| \leq \pi $, so \[\left| f'(x+iy)\right| =\sqrt{1+e^{-2x}-2e^{-x}\cos y}\geq \sqrt{1+e^{-2x}}\geq \sqrt{1+e^{-2k}}>1.\]Hence,  there exists a constant $ \lambda $, depending only on $ k $, such that $ \left| f'(z)\right| \geq \lambda $, for all $ z\in \left\lbrace z=x+iy\in S\smallsetminus \overline{V}\colon x\leq k\right\rbrace $. 
		Hence, the first statement follows applying the same reasoning as in (a).
		
		Finally, let $ K\subset S_k $ and denote by $ \lambda $ the constant of contraction in $ S_k $. Then, for all $ z,w\in \phi_i(K) $, we have $ f(z), f(w)\in K $, and \[\rho(z,w)\leq \lambda\rho(f(z), f(w))\leq \lambda \textrm{diam}_\rho (K).\]
		Hence, $  \textrm{diam}_\rho (\phi_i(K))\leq \lambda \textrm{diam}_\rho (K)$, as desired.
	\end{enumerate}
\end{proof}

\begin{remark}	{\bf (Expansion and uniform expansion in $ S\smallsetminus \overline{V} $)}\label{remark_expansion}
	We note that, as a direct consequence of Proposition \ref{remark_contract} (a), if $ z,w\in \Omega_i $ and $ f(z), f(w)\in S\smallsetminus\overline{V} $, then \[\rho(z, w)\leq \rho(f(z), f(w)).\]
	
	Likewise, the expansion is uniform in any half-strip $ S_k $. In particular, if $ K $ is a compact set such that $ \textrm{diam}_\rho (K)>0 $ and $ f^n(K)\subset S_k\cap \Omega_{i_n} $, $ i_n\in \left\lbrace 0,1\right\rbrace  $, then $ \textrm{diam}_\rho (f^n(K))\to\infty $, as $ n\to\infty $. 
\end{remark}

Next, we use this subdivision of the strip in $ \Omega_0 $ and $ \Omega_1 $  to define the itinerary for points in $ \widehat{S} $, where
 $ \Sigma_2 $ denotes the space of infinite sequences of two symbols, taken to be 0's and 1's.
\begin{defi}{\bf (Itineraries in $ \widehat{S} $)}\label{def-itinerary}
	Let $ z\in \widehat{S} $ be such that $ f^n(z)\notin \mathbb{R} $, for all $ n\geq 0 $. The sequence $ I(z)={\underline{s}}=\left\lbrace s_n\right\rbrace _n\in \Sigma_2 $ satisfying $ f^n(z)\in \Omega_{s_n} $ is called the \textit{itinerary} of $ z $.
\end{defi}
\begin{remark}
	For points in $ \widehat{S} $ which are eventually mapped to $ \mathbb{R} $, the itinerary is not defined. However, this can be neglected because they are in the Baker domain and their dynamics are already understood.
\end{remark}

We will need a further subdivision of the strip. Let us define the regions\[ \Omega_{ij}\coloneqq \phi_i( \phi_j(S))\smallsetminus \overline{V}.\] For instance, the region $ \Omega_{00} $ has to be seen as the set of points in $ \Omega_0 $ which remain in $ \Omega_0 $ after one iteration of the function, while points in $ \Omega_{01}$ are the points which change to $ \Omega_1 $. Clearly, if $ z\in \widehat{S} $ belongs to $ \Omega_{00} $, its itinerary starts with $ 00 $; while if  $z\in \Omega_{01} $, then $ I(z) $ begins with $ 01 $. The absorbing domain $ V $ is removed from the regions for practical use: this has no effect on the study of $ \partial U $, since its points are never in $ V $, but it allows us to give better estimates on the function. See Figure \ref{fig-regionsSij}.

\begin{figure}[htb!]\centering
	\includegraphics[width=15cm]{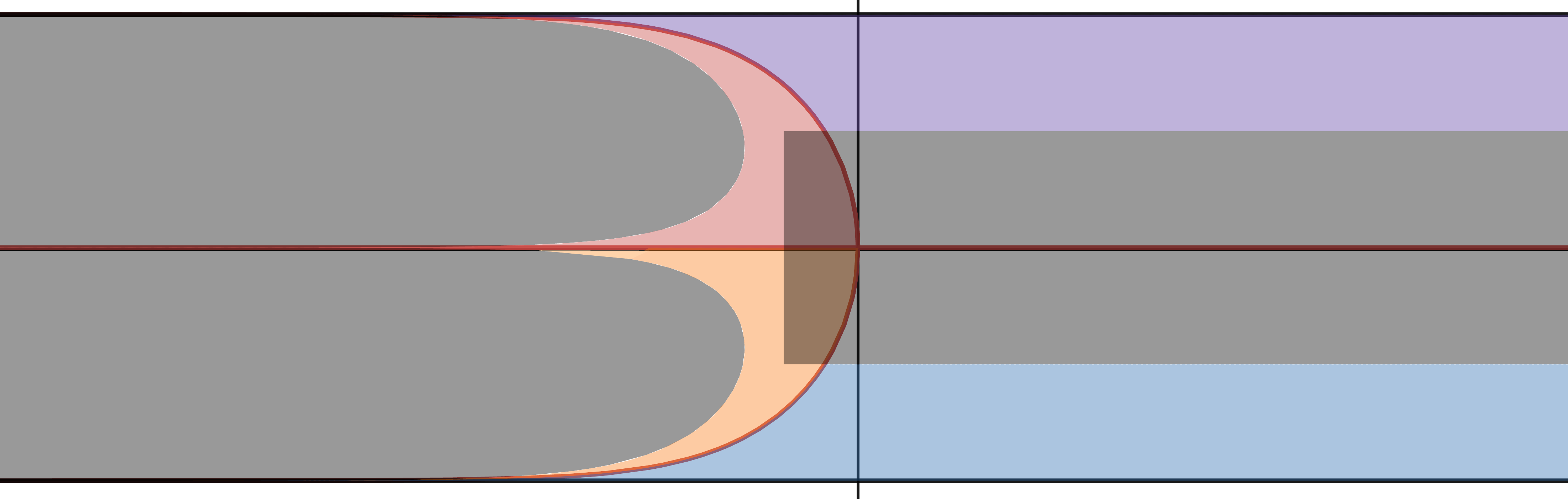}
	\setlength{\unitlength}{15cm}
	\put(-0.2, 0.26){$ \Omega_{00} $}
	\put(-0.2, 0.04){$ \Omega_{11} $}
	\put(-0.2, 0.17){$ V $}
	\put(-0.535, 0.19){$ \Omega_{01}$}
	\put(-0.535, 0.1){$ \Omega_{10}$}
	\caption{\footnotesize Graphic representation of the regions $ \Omega_{ij} $, $ i,j\in \left\lbrace 0,1\right\rbrace  $.}
	\label{fig-regionsSij}
\end{figure}
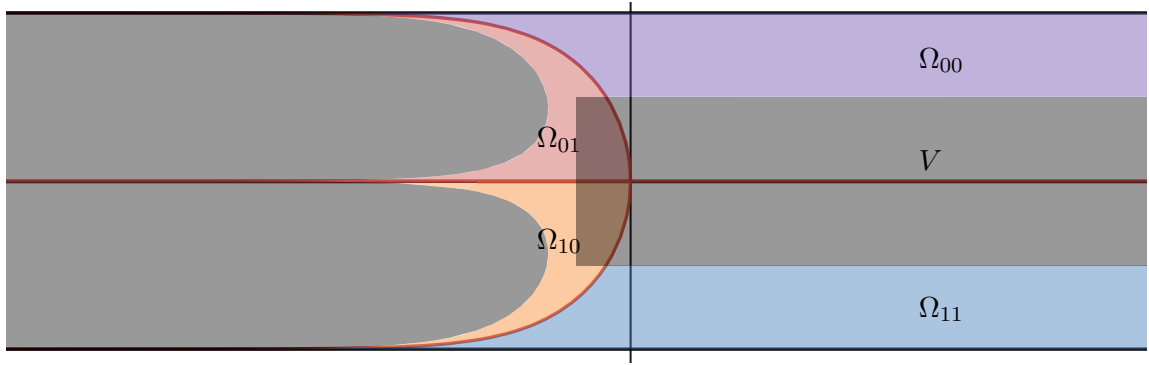

\begin{lemma}{\bf (Properties of the regions $ \Omega_{ij} $)}\label{lemma-properties-Sij}
	The following properties hold true. \begin{enumerate}[label={\em (\alph*)}]
		\item  $ \Omega_{01}, \Omega_{10}\subset \left\lbrace z\in S\colon \textrm{\em Re z}<0 \right\rbrace $. Therefore, if  $ z\in S\smallsetminus \overline{V} $ with $ \textrm{\em Re }z>0 $, either $ z\in \Omega_{00} $ or $ z\in \Omega_{11} $.
		\item If  $ z\in S\smallsetminus \overline{V} $ with $ -\frac{\pi}{2}<\textrm{\em Im }z<\frac{\pi}{2} $ and $ f(z)\in S $, then either $ z\in \Omega_{01} $ or $ z\in \Omega_{10} $.
		\item For $ z\in \Omega_{ii}$, $ i\in \left\lbrace 0,1\right\rbrace  $, we  have $ \left| \textrm{\em Im }z\right|>\frac{\pi}{2} $.  In particular, $  \textrm{\em Re } f(z)<\textrm{\em Re }z$ and, if $ z\notin L^\pm $, $ \left|  \textrm{\em Im } f(z)\right| <\left| \textrm{\em Im }z\right| $. 
		\item For $ z,w\in \Omega_{ij} $, $ i, j\in \left\lbrace 0,1\right\rbrace  $, it holds $\left| z-w\right| \leq \rho(z,w)\leq \left| z-w\right|+\pi$.
	\end{enumerate}
\end{lemma}
\begin{proof}
	The proof is direct from the definition of the regions. See also Figure \ref{fig-regionsSij}.
\end{proof}

\section{The escaping set in $\partial U $: Proof of \ref{teo:A}}\label{sect-4-escaping}
This section is devoted to the proof of \ref{teo:A}, which asserts that escaping points in $ \partial U $ are organized in curves, and $ \partial U $ is precisely the closure of these curves. To do so, a detailed study of the escaping set is required, which is carried out in a several number of steps. First, it is proven that all escaping points in $ \partial U $ are left-escaping (Lemma \ref{lemma-non-ocsillating-escaping}), and sufficiently to the left, curves of escaping points with the same itinerary are constructed (Prop. \ref{teo-curves-of-escaping-points}). Afterwards, these curves are enlarged by the dynamics to collect all points in $ S $ with the same itinerary (Thm. \ref{teo-dynamic-ray}); and, finally, all this construction is used to prove a characterization of $ \partial U $ (Prop. \ref{prop-caracteritzacio-frontera}), which is of independent interest. As indicated in the end of the section, \ref{teo:A} will follow from Theorem \ref{teo-dynamic-ray} (a) and Proposition \ref{prop-caracteritzacio-frontera} (b).

First, recall that $ \partial U\subset \widehat{S} $, where $ \widehat{S} $ consists of all the points in $ S $ which never leave $ S $ under iteration; and  
 observe that in $ \widehat{S} $ there are three distinguished ways to escape to infinity. Indeed, points can escape to infinity to the left, to the right, or oscillating from left to right. This leads us to define the following sets:
\[\mathcal{I}_S^+\coloneqq\left\lbrace z\in\mathcal{I}(f)\cap \widehat{S}\colon  \textrm{Re } f^{n}(z)\to +\infty \right\rbrace, \]
\[\mathcal{I}_S^-\coloneqq\left\lbrace z\in\mathcal{I}(f)\cap \widehat{S}\colon  \textrm{Re } f^{n}(z)\to -\infty \right\rbrace. \]
By construction, these two sets are disjoint, but they may not  contain all the escaping points: points which escape to $ \infty $ oscillating from left to right  belong  neither to $ \mathcal{I}_S^- $  nor to $ \mathcal{I}_S^- $. However, this possibility is excluded, as it is shown in the following lemma. Intuitively, oscillations are not possible because, on the right, the map is close to the identity.
\begin{lemma}{\bf (No oscillating escaping points)}\label{lemma-non-ocsillating-escaping}
	There are no oscillating escaping points, i.e.\[\mathcal{I}(f)\cap\widehat{S}=\mathcal{I}_S^+ \cup \mathcal{I}_S^- .\]
	Moreover, $ \mathcal{I}_S^+=U $.
\end{lemma}
\begin{proof}
	Assume $ z\in \mathcal{I}(f)\cap \widehat{S} $. For any $ r>0 $, there exists $ n_0 $ such that, for all $ n\geq n_0 $, $ f^n(z)\in S $ and $ \left| f^n(z)\right|>r $. In particular, taking $ r>\sqrt{\pi^2+1} $, there exists $ R>1 $ such that $\textrm{Re } f^n(z)>R $ or $\textrm{Re } f^n(z)<-R $, for all $ n\geq n_0 $. Assuming that $ \textrm{Re }z>R $, we are going to see that it is not possible to have $ \textrm{Re }f(z)<-R $, so oscillating escaping orbits are not possible. Indeed, \[\textrm{Re } f(x+iy)=x+e^{-x}\cos y\geq x-e^{-x}\geq R-e^{-R}.\]Since $ R>1 $, the right-hand side of the inequality is greater than 0, so it does not hold $ \textrm{Re }f(z)<-R $, proving the first statement.

To prove the second statement, first observe that $ U\subset \mathcal{I}^+_S $. 
	It is left to show that, for $ z \in \widehat{S}\smallsetminus U $, it cannot hold $ \textrm{Re } f^n(z)\to +\infty $. Indeed, such a point never enters the absorbing domain, so, when $ \textrm{Re } f^n(z)>0 $, either $ \textrm{Im } f^n(z)>\frac{\pi}{2} $ or $ \textrm{Im } f^n(z)<-\frac{\pi}{2} $. In both cases, $ \textrm{Re } f^{n+1}(z)<\textrm{Re } f^n(z) $, so it is impossible for a point which is not in $ U $ to belong to $ \mathcal{I}_S^+ $.
\end{proof}

 Next we show that these left-escaping points are organized in curves, which eventually contain all left-escaping points with the same itinerary. To do so, we adapt the proof of \cite[Prop. 3.2]{goldberg-devaney} for the exponential maps $ \lambda e^{z} $, $ 0<\lambda<\frac{1}{e} $, to our setting. Moreover, the construction is made in such a way that a parametrization of the curves appears implicitly, as the one introduced in \cite{bodelon} for the exponential family (see also \cite{schleicher-zimmer, tesi-lasse-rempe,rempe2007}). The main attribute of this parametrization is to conjugate the dynamics on the curve with the model of growth given by $ F(t)=t-e^{-t} $, $ t\in\mathbb{R} $. Observe that $ F\colon\mathbb{R}\to\mathbb{R} $ is an increasing homeomorphism of $ \mathbb{R} $ without fixed points, where all iterates converge to $ -\infty $ under iteration. 

\begin{prop}{\bf (Escaping tails)} \label{teo-curves-of-escaping-points}
	For every sequence $ {\underline{s}}=\left\lbrace s_n\right\rbrace _n \in\Sigma_2$ there exists a curve of left-escaping points $ \gamma_{\underline{s}}\colon \left( -\infty, -2\right] \to \mathcal{I}^-_S $, whose points have itinerary $ {\underline{s}} $ and $ \gamma_{\underline{s}}\subset\partial U $. Such curve is called {\em escaping tail}.
	The following properties hold.
		\begin{enumerate}[label={\em (\alph*)}]
		\item {\em (Asymptotics and dynamics)} It holds that  $ \textrm{Re }\gamma_{{\underline{s}}}(t)\to -\infty $ , as $ t\to-\infty $,  and $\textrm{Re } f^n(\gamma_{{\underline{s}}}(t)) \to -\infty$, as $ n\to \infty $. Moreover, $ \textrm{Re }f^n(\gamma_{{\underline{s}}}(t))\leq -2 $ for all $ n\geq 0 $.
		\item {\em (Uniqueness)} Escaping tails are unique, in the sense that if $ z\in \mathcal{I}^-_S $, with $ I(z)=\underline{s} $, and $ \textrm{Re }f^n(z)\leq -2 -\pi$ for all $ n\geq 0 $, then $ z\in\gamma_{\underline{s}} $.
		\item {\em (Internal dynamics)} For $ t\leq -2 $, it is satisfied \[ f(\gamma_{\underline{s}}(t))=\gamma_{\sigma({\underline{s}})}(F(t)),\] where $ \sigma  $ denotes the shift map and $ F(t)=t-e^{-t} $.
	\end{enumerate}
\end{prop}

It is worth mentioning that the existence of such curves of escaping points can be deduced directly from \cite[Thm. 1.2]{r3s} for functions in class $ \mathcal{B} $ of finite order, applied to $ h(w)=we^{-w} $. Indeed, both functions $ f $ and $ h $ are semiconjugate by $ w=e^{-z} $, so left-escaping points for $ f $ correspond to the escaping set of $ h $.  Then, if $ z\in \mathcal{I}^-_S $, then $ w=e^{-z}\in\mathcal{I}(h) $ and, by \cite[Thm. 1.2]{r3s}, it is connected to $ \infty $ by a curve $ \Gamma $ of escaping points. An appropriate lift $ \gamma $ of $ \Gamma $ is a curve of left-escaping points connecting $ z $ to infinity.
It is easy to see that points in $ \gamma $ must have the same itinerary. Indeed, $ \gamma $ must be contained in either $ \Omega_0 $ or in $ \Omega_1 $, since it cannot intersect $ \mathbb{R} $ (because it is in the Fatou set) nor $ L^\pm $ (since $ L^\pm $ separate distinct preimages of the $ w $-plane under $ w=e^{-z} $). Moreover, this is also true for any iterated image of $ \gamma $, implying that all points in $ \gamma $ must have the same itinerary.

However, from this general result, it cannot be deduced which of these curves are in $ \partial U $ and it does not give a parametrization for the curves, which will be important in the following sections. This is why we choose an alternative proof for Proposition \ref{teo-curves-of-escaping-points}, based on the more constructive approach of \cite{goldberg-devaney}. On the other hand, we do apply \cite[Thm. 1.2]{r3s} to deduce the uniqueness of the escaping tails.

\begin{proof} [Proof of Proposition \ref{teo-curves-of-escaping-points}]
	First, let us show that, to every $  t\leq- 2  $ and $ \underline{s}\in\Sigma_2 $, we can find a left-escaping point $ z^{t, \underline{s}} $, with itinerary $ \underline{s} $, associated to $ t $. To do so, fix 
 $  t\leq- 2  $ and $ \underline{s}\in\Sigma_2 $, and let $ D^{t,\underline{s}}_0 $ be the square of side length $ \pi $ located in $ \Omega_{s_0} $ and right-hand side at $t_0\coloneqq t $. We construct  a sequence of squares $ \left\lbrace D^{t,\underline{s}}_n\right\rbrace _n $, where $ D^{t,\underline{s}}_n $ is a square of side length $ \pi $, located in $ \Omega_{s_n} $ and right-hand side  $t_n\coloneqq F^n(t)  $, where $ F(t)=t-e^{-t} $. Observe that $t_n\to-\infty $, as $ n\to\infty $. Compare with Figure \ref{fig-quadrats}.
 		\begin{figure}[htb!]\centering
 	\includegraphics[width=15cm]{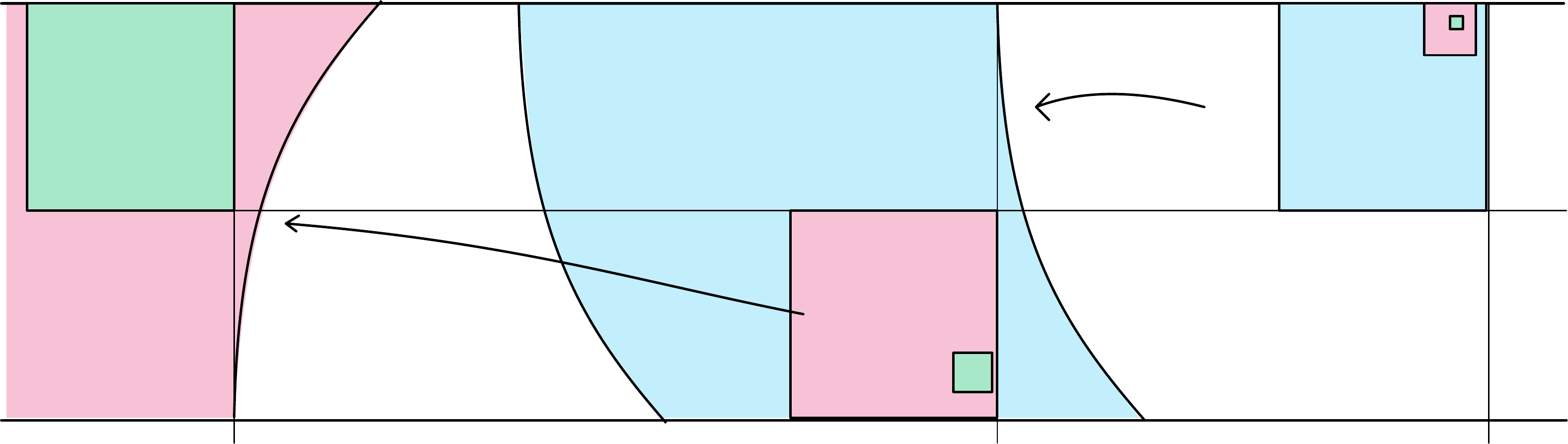}
 	\setlength{\unitlength}{15cm}
 	\put(-0.7, 0.105){\footnotesize $f$}
 	\put(-0.29, 0.23){\footnotesize $f$}
 	\put(-0.13, 0.22){\footnotesize $D_0^{t, \underline{s}}$}
 	\put(-0.45, 0.08){\footnotesize $D_1^{t, \underline{s}}$}
 	\put(-0.94, 0.22){\footnotesize $D_2^{t, \underline{s}}$}
 		\put(-0.11, -0.02){\footnotesize $\left\lbrace \textrm{Re }z=t\right\rbrace $}
 			\put(-0.425, -0.02){\footnotesize $\left\lbrace \textrm{Re }z=F(t)\right\rbrace $}
 				\put(-0.91, -0.02){\footnotesize $\left\lbrace \textrm{Re }z=F^2(t)\right\rbrace $}
 	\caption{\footnotesize Schematic representation of the first three squares $ \left\lbrace D^{t,\underline{s}}_n\right\rbrace _n $, for a given $ t\leq -2 $, showing how they satisfy $ D^{t,\underline{s}}_n\subset f(D^{t,\underline{s}}_{n-1}) $.}\label{fig-quadrats}
 \end{figure}
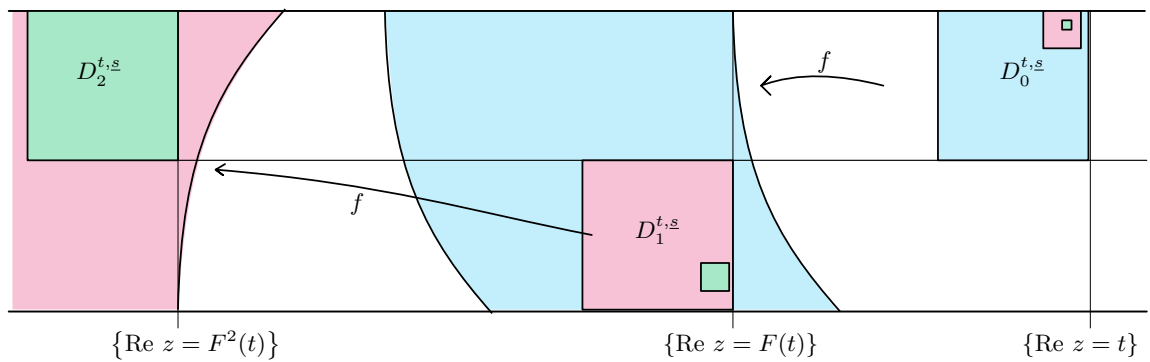
	
	\begin{claim*}
		The squares $ \left\lbrace D^{t,\underline{s}}_n\right\rbrace _n $ satisfy $ D^{t,\underline{s}}_n\subset f(D^{t,\underline{s}}_{n-1}) $, for all $ n\geq 1 $.
	\end{claim*}
\begin{proof}[Proof of the claim]
It is enough to show that $ D^{t,\underline{s}}_1\subset f(D^{t,\underline{s}}_{0}) $. Let us denote by $ \partial^- D$ and $ \partial^+ D$, the left and the right-hand sides of a square $ D $, respectively. 

First let us observe that, on the left, the map $ f $ acts on a similar way than the exponential, sending vertical segments to circular curves, which start at $ L^+ $, ends at $ L^- $ and have an auto-intersection in the negative real line. Compare with Figure \ref{fig-din-S}.

\begin{figure}[htb!]\centering
	\includegraphics[width=15cm]{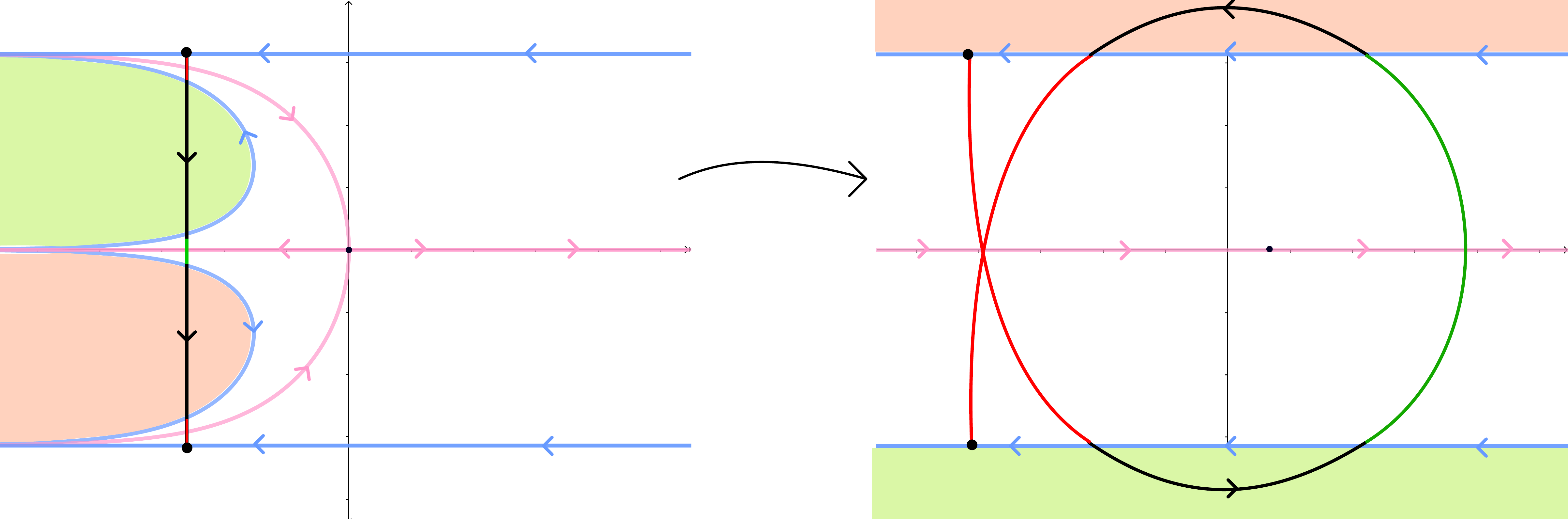}
	\setlength{\unitlength}{15cm}
	\put(-0.52,0.23){$ f $}
	\put(-0.9,0.305){\footnotesize $ t+\pi i $}
	\put(-0.9,0.025){\footnotesize $ t-\pi i $}
	\put(-0.42,0.305){\footnotesize $ F(t)+\pi i $}
	\put(-0.42,0.025){\footnotesize $ F(t)-\pi i $}
	\caption{\footnotesize Schematic representation of how $ f $ acts on the left-side on the strip.}
	\label{fig-din-S}
\end{figure}
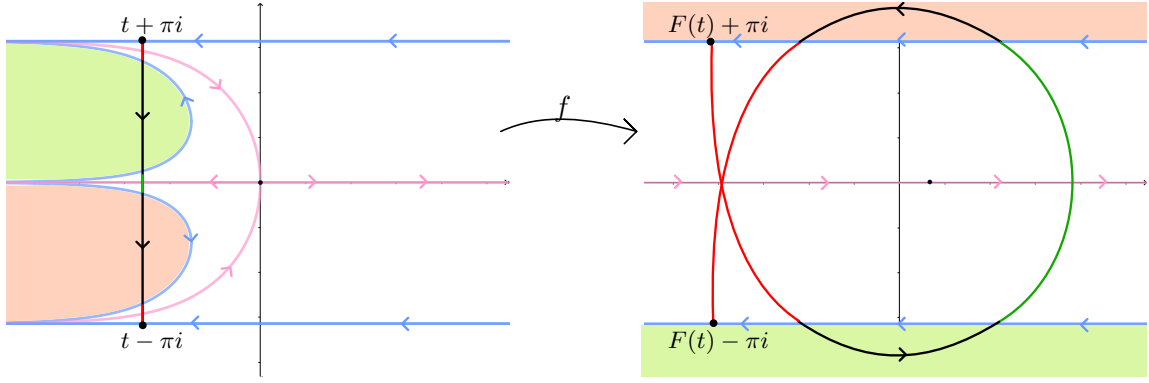

 Moreover, if $ \textrm{Re }z=t\leq -2 $, we have the following inequality controlling the modulus of the image: \[ \left| f(z)\right| =\left| z+e^{-z}\right| \geq \left| e^{-z}\right| -\left| z\right| = e^{-\textrm{Re }z} -\left| z\right|>\frac{1}{2}e^{-\textrm{Re } z}=\frac{1}{2}e^{-t}>-t.\]

To prove that $ D^{t,\underline{s}}_1\subset f(D^{t,\underline{s}}_{0}) $, it is enough to show that $ \partial^- D^{t,\underline{s}}_1$ and  $ \partial^+ D^{t,\underline{s}}_1$ are contained in $ f(D^{t,\underline{s}}_{0}) $. In fact, we shall see that $ \partial^- D^{t,\underline{s}}_1$ and  $ \partial^+ D^{t,\underline{s}}_1$ are contained in $ f(D^{t,\underline{s}}_{0}) \cap S\cap \left\lbrace \textrm{Re }z<0\right\rbrace $. Compare with Figure \ref{fig-din-S}.

First we see that $ \partial^+ D^{t,\underline{s}}_1$ is located more to the left than $ f(\partial ^+D^{t,\underline{s}}_{0}) $. Indeed, points in $ \partial^+ D^{t,\underline{s}}_1$ have real part $ t-e^{-t} $, while for $z\in \partial^+ D^{t,\underline{s}}_{0} $ it is satisfied that $\textrm{Re }f(z)\geq t-e^{-t}$.

Finally, to see that $ \partial^- D^{t,\underline{s}}_1$ is contained in $ f(D^{t,\underline{s}}_{0}) \cap S\cap \left\lbrace \textrm{Re }z<0\right\rbrace $,  we shall see that $ \partial^- D^{t,\underline{s}}_1$ is located more to the right than $ f(\partial ^-D^{t,\underline{s}}_{0})\cap S\cap  \left\lbrace \textrm{Re }z<0\right\rbrace $. For 
 $ z \in \partial ^-D^{t,\underline{s}}_{0} $ and such that $ f(z)\in S\cap \left\lbrace \textrm{Re }z<0\right\rbrace  $, we have:
\[ \textrm{Re }f(z)\leq- \left| f(z)\right| +\pi <-\frac{1}{2}e^{-\textrm{Re } z}+\pi=-\frac{1}{2}e^{-(t-\pi)}+\pi.\]

A point $ z\in \partial^- D^{t,\underline{s}}_1 $ has real part $ t-e^{-t}-\pi $, which is easy to see that it is bigger than our previous bound. Indeed,  the real function $ h(x)= x-e^{-x} +\frac{1}{2}e^{-(x-\pi)}-2\pi$ is positive, when $ x<0 $.  Therefore, the claim is proved. 
\end{proof}

	Now, let us define\[ Q^{t, \underline{s}}_n\coloneqq\phi_{s_0}\circ\dots\circ\phi_{s_n}(\overline{D^{t, \underline{s}}_{n+1}}),\]\[z^{t, \underline{s}}\coloneqq\bigcap\limits_{n\geq0} Q^{t, \underline{s}}_n.\]
	Notice that $ z^{t, \underline{s}} $ is a unique point. Indeed, $ \left\lbrace Q^{t, \underline{s}}_n\right\rbrace _n $ is a sequence of nested compact sets contained in $ D^{t, \underline{s}}_0 $. Its intersection is a connected compact set, and to prove that it consists precisely of a unique point, we shall see that the diameter of $ Q^{t, \underline{s}}_n $ tends to 0, as $ n\to\infty $. Indeed, since $  \phi_{s_k}\circ\dots\circ\phi_{s_n}(\overline{D^{t, \underline{s}}_{n+1}}) \subset \left\lbrace \textrm{Re }z<0\right\rbrace  $ for all $ n\geq 0 $ and $ k\leq n $, each time we apply either $ \phi_0 $ or $ \phi_1 $ we are applying a contraction of constant $ \frac{1}{\lambda} <1$ with respect to the $ \rho $-distance (see Prop. \ref{remark_contract}).  Recall that, in the half-plane $ \left\lbrace \textrm{Re }z<-2\right\rbrace $, the $ \rho $-distance and the Euclidean distance coincide. Hence,
	\[\textrm{diam }Q^{t, \underline{s}}_n= \textrm{diam}_\rho Q^{t, \underline{s}}_n\leq\frac{1}{\lambda^{n+1}}\sqrt{2}\pi\to 0,\textrm{ as } n\to \infty.\]
	
	The point $ z^{t, \underline{s}} $ satisfies the required conditions. Indeed, $ z^{t, \underline{s}} $
	 follows the itinerary prescribed by $ {\underline{s}} $ and converges to $ -\infty  $ under iteration. Moreover,  we claim that $ z^{t, \underline{s}} \in\partial U$. Indeed, since $ f( \overline{D^{t, \underline{s}}_{n+1}} ) $ intersects $ U $, then $ \overline{D^{t, \underline{s}}_{n+1}} $ contains points of $ U $, and so does $ Q^{t, \underline{s}}_n $.  Since the sets $ \left\lbrace Q^{t, \underline{s}}_n\right\rbrace_n $ shrink to $ z^{t, \underline{s}}  $, this gives a sequence of points in $ U $ approximating $ z^{t, \underline{s}} $. 
	 
	Therefore, we associate to any $ t\leq -2$ and $ \underline{s} \in \Sigma_2$ the point $ z^{t, \underline{s}} $. 
	 Observe that the resulting point $ z^{t, \underline{s}} $ depends continuously on $ t $, since the entire construction depends continuously on $ t $. Hence, letting $ t\to-\infty $, the points $ z_{{\underline{s}}, t} $ describing the required curve $ \gamma_{\underline{s}} $ of left-escaping points with itinerary $ {\underline{s}} $. This induces naturally a parametrization on $ \gamma_{\underline{s}} $: we define 
	$\gamma_{\underline{s}} \colon \left( -\infty, -2\right]  \to \mathbb{C}$ such that
	$ \gamma_{\underline{s}}(t)\coloneqq z^{t,\underline{s}}$.
	
	Finally, let us prove that, with this parametrization, the announced properties actually hold.
		\begin{enumerate}[label={ (\alph*)}]
			\item { (Asymptotics and dynamics)}
			
			It is clear by the construction of the squares that  $ \gamma_{{\underline{s}}}(t)\to -\infty $ , as $ t\to-\infty $,  and, for every $ t\leq -2 $, $ f^n(\gamma_{{\underline{s}}}(t)) \to -\infty$, as $ n\to \infty $. Moreover, since the orbit of a point is contained in the corresponding squares, we have  $ \textrm{Re }f^n(\gamma_{{\underline{s}}}(t))\leq -2 $ for all $ n\geq 0 $.
			 
		\item { (Uniqueness)} 
		
		Uniqueness follows from the results in \cite{r3s}, which imply that every point in $ \mathcal{I}_S^- $ can be connected to infinity by a curve of left-escaping points with the same itinerary. 
		
		Assume, on the contrary, that there exists $ z_0\in \mathcal{I}^-_S $, with $ I(z_0)=\underline{s} $, and $ \textrm{Re }f^n(z_0)\leq -2 -\pi$ for all $ n\geq 0 $, but $ z_0\notin\gamma_{\underline{s}} $. Then, there would exist another curve $ \widetilde{\gamma}_{\underline{s}} $ of left-escaping points with itinerary $ \underline{s} $ connecting $ z_0 $ to $ \infty $. Consider an open set $ W $ placed in the left-unbounded region delimited by $ \gamma_{\underline{s}} $, $ \widetilde{\gamma}_{\underline{s}} $ and $ \left\lbrace z\in S\colon \textrm{Re }z= \textrm{Re }z_0\right\rbrace  $. 
		
		We claim that $ f^n (W)\subset S\cap \left\lbrace \textrm{Re }z<-2\right\rbrace  $, for all $ n\geq 0 $.
		Indeed, note that   $\gamma_{\underline{s}} , \widetilde{\gamma}_{\underline{s}}\subset  \left\lbrace \left| \textrm{Im }z\right| >\frac{\pi}{2}\right\rbrace$.
		Then, $ W\subset S\cap \left\lbrace\left| \textrm{Im }z\right| >\frac{\pi}{2}\right\rbrace   $. Recall that, for $ z\in S\cap \left\lbrace\left| \textrm{Im }z\right| >\frac{\pi}{2}\right\rbrace  $, $  \textrm{Re }f(z)<\textrm{Re }z$.  Hence, $ f (W)\subset S\cap \left\lbrace \textrm{Re }z<-2\right\rbrace  $, and, by continuity, $ f (W) $ is the  left-unbounded region delimited by $ f(\gamma_{\underline{s}} )$, $ f(\widetilde{\gamma}_{\underline{s}} )$ and $ f(\left\lbrace z\in S\colon \textrm{Re }z= \textrm{Re }z_0\right\rbrace ) \subset S\cap \left\lbrace \textrm{Re }z<-2\right\rbrace$. 
		We can apply the same argument inductively to see that $ f^n (W)\subset S\cap \left\lbrace \textrm{Re }z<-2\right\rbrace  $, for all $ n\geq 0 $, as claimed.
		
		Therefore, $ W $ is an open set which never enters the Baker domain, so $ W\subset \mathcal{J}(f) $, leading to a contradiction.

		\item {  (Internal dynamics)} 
		
		We have to prove that, for $ t\leq -2 $, \[ f(\gamma_{\underline{s}}(t))=\gamma_{\sigma({\underline{s}})}(F(t)).\]

			First observe that, since $ F $ is an increasing map, $ F(t)<-2 $ for $ t\leq-2 $, so $ \gamma_{\sigma({\underline{s}})}(F(t)) $ is defined. 
		
		To construct the point $ \gamma_{\sigma({\underline{s}})}(F(t)) $ we use the sequence of squares $ \left\lbrace D^{F(t), \sigma({\underline{s}})}_n\right\rbrace _n $. Therefore, the $ n $-th square has right-hand side located at $ \left\lbrace x= F^n(F(t))= F^{n+1}(t)\right\rbrace  $ and it is in the half-strip $\Omega_{s_{n+1}} $. Hence, $ D^{F(t), \sigma({\underline{s}})}_n= D^{t,{\underline{s}}}_{n+1} $. Moreover,\[ Q^{F(t), \sigma({\underline{s}})}_n=\phi_{s_1}\circ\dots\circ \phi_{s_{n+1}}\left( \overline{D^{F(t), \sigma({\underline{s}})}_{n+1}}\right) = \phi_{s_1}\circ\dots\circ \phi_{s_{n+1}}\left( \overline{D^{t, {\underline{s}}}_{n+2}}\right) =f(Q^{t,{\underline{s}}}_{n+1}).\]
		Then, \[\gamma_{\sigma({\underline{s}})}(F(t))=\bigcap\limits_{n\geq0} Q^{F(t), \sigma({\underline{s}})}_n =\bigcap\limits_{n\geq0} f(Q^{t,{\underline{s}}}_{n+1})=f( \bigcap\limits_{n\geq0} Q^{t, {\underline{s}}}_{n+1}) =f(\gamma_{\underline{s}}(t)),\] as desired.
	\end{enumerate}
\end{proof}

Escaping tails are mapped among them following the symbolic dynamics given by its itinerary: if $ \sigma $ denotes the shift map in $ \Sigma_2 $ and $ {\underline{s}}\in\Sigma_2 $, we have $ f(\gamma_{\underline{s}})\subset \gamma_{\sigma({\underline{s}})} $. 
Moreover, we claim that, as a consequence of Proposition \ref{teo-curves-of-escaping-points} (c), this last inclusion is strict. Indeed, recall that, for all $ t_0\leq -2 $, it holds $ F(t_0)<t_0 $. Hence,  Proposition \ref{teo-curves-of-escaping-points} (c) implies \[ f(\gamma_{{\underline{s}}}(\left\lbrace t\colon t\leq t_0 \right\rbrace ))=\gamma_{\sigma({\underline{s}})}(\left\lbrace t\colon t\leq F(t_0) \right\rbrace )\subset \gamma_{\sigma({\underline{s}})}(\left\lbrace t\colon t\leq t_0 \right\rbrace ),\] where the last inclusion is strict.

Next, we define the dynamic rays as the natural extension of the escaping tails: we enlarge a given escaping tail $ \gamma_{\underline{s}} $ by adding to it all points in $ \widehat{S} $ which  are eventually mapped to $ \gamma_{\sigma^n({\underline{s}})} $, for some $ n\geq 0 $ (see Fig. \ref{fig-hairs}).  Next theorem includes the formal definition as well as the corresponding extension of the dynamical properties of the escaping tails. Moreover, a new property is proven, showing the continuity of the parametrization the hairs with respect to the itinerary, analogously to \cite[Lemma 3.2]{rempe2007}.

\begin{thm}{\bf (Dynamic rays)}\label{teo-dynamic-ray}
		Let $ {\underline{s}}\in \Sigma_2 $. Let us define the {\em dynamic ray} (or {\em hair}) of sequence $ \underline{s} $ as 
	$\gamma^\infty_{\underline{s}} \colon \left( -\infty, +\infty\right)   \to \mathcal{I}^-_S$ such that, if $ n\geq 0 $ with $ F^n(t)<-2 $, then \[\gamma^\infty_{\underline{s}}(t)\coloneqq \phi_{s_0}\circ\dots\circ\phi_{s_{n-1}}(\gamma_{\sigma^n({\underline{s}})}(F^n(t))).\]
	The following properties hold.
	\begin{enumerate}[label={\em (\alph*)}]
		\item{\em (Well-defined)} Dynamic rays are well-defined, in the sense that the definition does not depend on $ n $. Moreover,  $ \gamma^\infty_{\underline{s}} $ is actually a curve and contains all left-escaping points with itinerary $ \underline{s} $.
		\item {\em (Internal dynamics)} 	For $ t\in\mathbb{R} $, it holds \[ f(\gamma^\infty_{\underline{s}}(t))=\gamma^\infty_{\sigma({\underline{s}})}(F(t)),\] where $ \sigma  $ denotes the shift map and $ F(t)=t-e^{-t} $.
		\item {\em (Continuity between rays)}\label{lemma-continuity-rays} Let $ n_0 \in\mathbb{N}$ and $ \underline{s}\in\Sigma_2$. Let us denote by $ \Sigma_2({\underline{s}}, n_0) $ the set of all sequences $ \widetilde{{\underline{s}}}\in\Sigma_2 $ which agree with $ {\underline{s}}$ in the first $ n_0 +1$ entries. Then, for all $ t_0\in\mathbb{R} $ and $ \varepsilon >0 $, there exists $ n_0 $ such that \[\left| \gamma^\infty_{\underline{s}}(t)-\gamma^\infty_{\widetilde{{\underline{s}}}}(t)\right| <\varepsilon,\] for all $ t\leq t_0 $ and $ \widetilde{{\underline{s}}}\in \Sigma_2({\underline{s}}, n_0)$.
	\end{enumerate}
\end{thm}
		\begin{figure}[htb!]\centering
	\includegraphics[width=12cm]{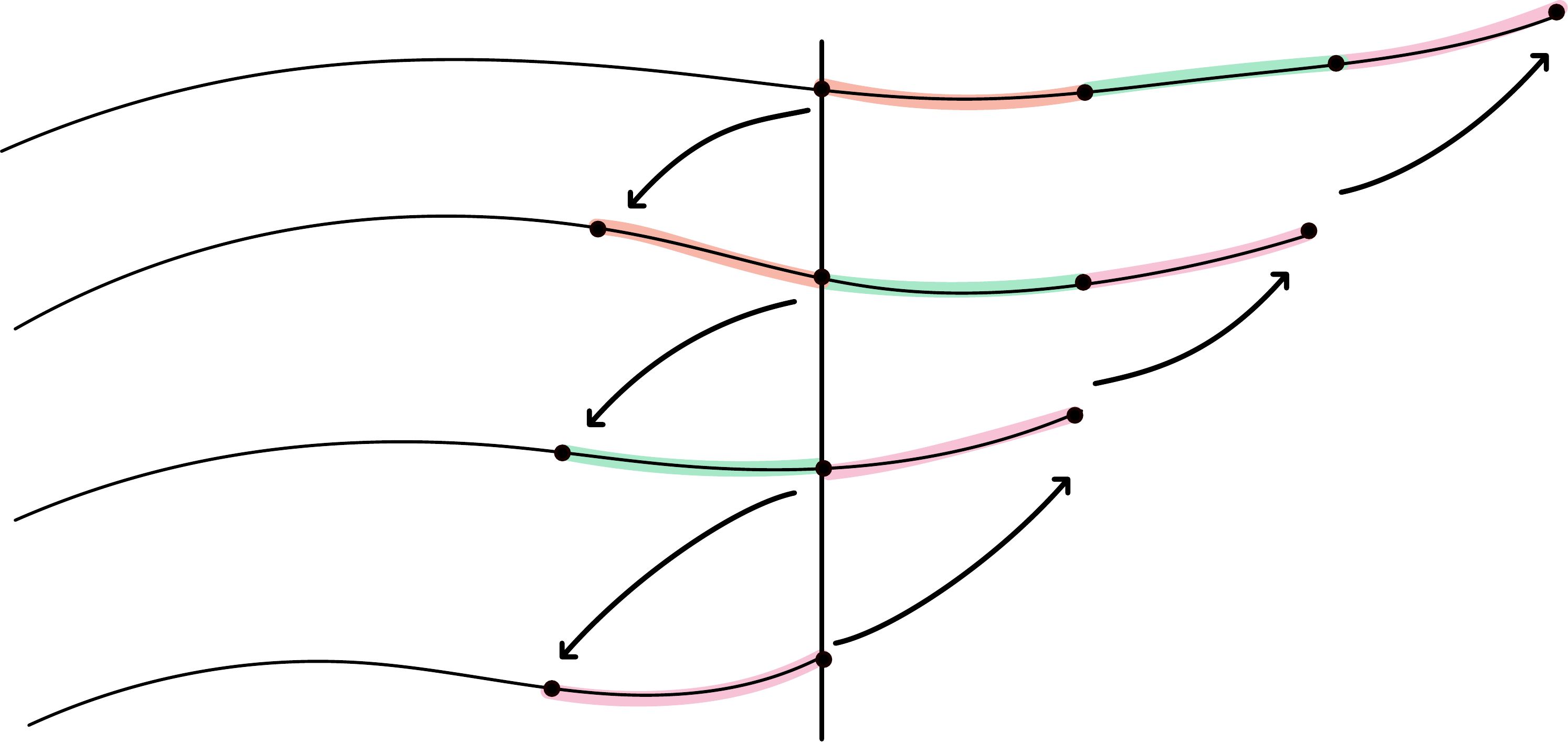}
	\setlength{\unitlength}{12cm}
	\put(-0.62, 0.09){\footnotesize $f$}
	\put(-0.59, 0.245){\footnotesize $f$}
	\put(-0.57, 0.38){\footnotesize $f$}
	\put(-0.38, 0.09){\footnotesize $\phi_{s_2}$}
	\put(-0.21, 0.245){\footnotesize $\phi_{s_1}$}
	\put(-0.05, 0.38){\footnotesize $\phi_{s_0}$}
	\put(-1, 0.4){\footnotesize $  \gamma_{\underline{s}} $}
	\put(-1.02, 0.26){\footnotesize $  \gamma_{\sigma({\underline{s}})} $}
	\put(-1.02, 0.13){\footnotesize $  \gamma_{\sigma^2({\underline{s}})} $}
	\put(-1.01, 0){\footnotesize $  \gamma_{\sigma^3({\underline{s}})} $}
	\caption{\footnotesize Construction of the hair $ \gamma^\infty_{\underline{s}} $ from the escaping tail $ \gamma_{\underline{s}} $. Intuitively, the process is clear: since the endpoint of the escaping tail is not mapped to the endpoint of the next escaping tail but to a point further to the left, the remaining piece of escaping tail can be added to the previous one by pulling back by the inverse. Repeating the process we get all the points in the ray. }\label{fig-hairs}
\end{figure}
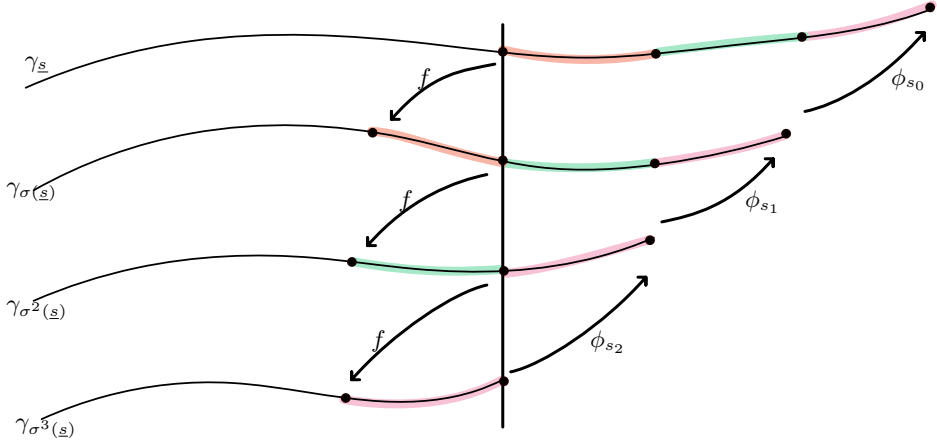

\begin{proof}
		\begin{enumerate}[label={ (\alph*)}]
		\item {(Well-defined)} 
		
Fix $ {\underline{s}}\in\Sigma_2 $ and $ t>-2 $, and let $ m>n $ be such that $ F^m(t)<-2 $ and $ F^n(t)<-2 $. Put $ m=n+l $, with $ l>0 $. We have to see that \[
		\phi_{s_0}\circ \dots\circ\phi_{s_{n-1}}\circ\phi_{s_n}\circ\dots\circ \phi_{s_{n+l-1}} \left( \gamma_{\sigma^{n+l}({\underline{s}})}(F^l(F^n(t)))\right) = \phi_{s_0}\circ \dots\circ\phi_{s_{n-1}}\left( \gamma_{\sigma^{n}({\underline{s}})}(F^n(t))\right).
		\]Since $ \phi_i $, $ i\in \left\lbrace 0,1\right\rbrace  $, are univalent, this is equivalent to \[
		\phi_n\circ\dots\circ \phi_{s_{n+l-1}} \left( \gamma_{\sigma^{n+l}({\underline{s}})}(F^l(F^n(t)))\right) =  \gamma_{\sigma^{n}({\underline{s}})}(F^n(t)),
		\] and this last equality holds true by the internal dynamics of the escaping tail (Prop. \ref{teo-curves-of-escaping-points}(c)). 
		
		 Finally, in view of Proposition \ref{teo-curves-of-escaping-points}, it is clear that dynamic rays are actually curves and contain all left-escaping points with the same itinerary, proving statement (a).

		\item {(Internal dynamics)}

			We shall assume that $ t>-2 $, otherwise the point $ \gamma^\infty_{\underline{s}}(t) $ is in the escaping tail, where we have already proven the statement. Let $ n $ be such that $ F^n(t)\leq -2 $. Then, applying the known equality for the escaping tails, we have\[ f(\gamma^\infty_{\underline{s}}(t))= f(\phi_{s_0}\circ\dots\circ \phi_{s_{n-1}}(\gamma_{\sigma^n({\underline{s}})}(F^n(t))))=\]\[= \phi_{s_1}\circ\dots\circ \phi_{s_{n-1}}(\gamma_{\sigma^{n-1}(\sigma ({\underline{s}}))}(F^{n-1}(F(t))))=\gamma^\infty_{\sigma({\underline{s}})}(F(t)), \] proving statement (b).
			
		\item {(Continuity between rays)}
		
		Fix $ \underline{s}\in\Sigma_2$ and $ t_0\in\mathbb{R} $. The goal is to determine $ n_0 \in\mathbb{N}$ such that if $ \widetilde{{\underline{s}}}\in\Sigma_2 $ which agree with $ {\underline{s}}$ in the first $ n_0 +1$ entries and $ t\leq t_0 $, then \[\left| \gamma^\infty_{\underline{s}}(t)-\gamma^\infty_{\widetilde{{\underline{s}}}}(t)\right| <\varepsilon.\]
		
		To do so,	first assume $ t_0\leq -2 $ and fix $ \varepsilon>0 $. Let $ \lambda>1 $ be the factor of expansion of $ f $ in $ S\cap \left\lbrace \textrm{Re }z<0\right\rbrace  $ (see Rmk. \ref{remark_expansion}). Let $ n_0 $ be such that $ \frac{1}{\lambda^{n_0}}\sqrt{2}\pi<\varepsilon $. We claim that for $ \widetilde{{\underline{s}}}\in\Sigma_2({\underline{s}},n_0)$ and $ t\leq t_0 $ it holds \[\left| \gamma^\infty_{\underline{s}}(t)-\gamma^\infty_{\widetilde{{\underline{s}}}}(t)\right| <\varepsilon.\]Indeed, by construction we have\[\gamma^\infty_{\underline{s}}(t),\gamma^\infty_{\widetilde{{\underline{s}}}}(t)\in\bigcap\limits_{n=0}^{n_0-1} Q^{t,{\underline{s}}}_n=Q^{t,{\underline{s}}}_{n_0-1}=\phi_{s_0}\circ\dots\circ \phi_{s_{n_0-1}}(\overline{D^{t, {\underline{s}}}_{n_0}}).\]Therefore,\[\textrm{diam } Q^{t,{\underline{s}}}_{n_0}\leq\frac{1}{\lambda^{n_0}}\textrm{diam } \overline{D^{t, {\underline{s}}}_{n_0}}=\frac{1}{\lambda^{n_0}}\sqrt{2}\pi<\varepsilon,\]implying that $ \left| \gamma^\infty_{\underline{s}}(t)-\gamma^\infty_{\widetilde{{\underline{s}}}}(t)\right| <\varepsilon $, as desired.
		
		Now assume $ t_0>-2 $. 
		Choose $ n_1 $ such that $ F^{n_1}(t_0)<-2 $ (and, hence, $ F^{n_1}(t)<-2 $, for all $ t\leq t_0 $). By the previous reasoning, we can find $ n_0 $ such that $ \left| \gamma^\infty_{\sigma^{n_1}({\underline{s}})}(t)-\gamma^\infty_{\widetilde{{\underline{s}}}}(t)\right| <\varepsilon  $, for $ \widetilde{{\underline{s}}}\in \Sigma_2(\sigma^{n_1}({\underline{s}}),n_0) $ and $ t\leq -2 $. Take $ n\coloneqq n_0+n_1 $ and let us check that the property of the lemma is satisfied.
		
		Indeed, take $ \widetilde{{\underline{s}}}\in \Sigma_2({\underline{s}}, n) $. Then, $ \sigma^{n_1}(\widetilde{{\underline{s}}})\in \Sigma_2(\sigma^{n_1}({\underline{s}}), n_0) $ and $ F^{n_1}(t)<-2 $, so
		\[\left| \gamma^\infty_{\sigma^{n_1}({\underline{s}})}(F^{n_1}(t))-\gamma^\infty_{\sigma^{n_1}(\widetilde{{\underline{s}}})}(F^{n_1}(t))\right| <\varepsilon .\] 
		Since applying the inverses $ \phi_i $, $ i\in\left\lbrace 0,1\right\rbrace  $ does not increase the distance between points, we get 	\[\left| \gamma^\infty_{{\underline{s}}}(t)-\gamma^\infty_{\widetilde{{\underline{s}}}}(t)\right|= \left| \phi_{s_0}\circ\dots\circ\phi_{n_1-1} (\gamma^\infty_{\sigma^{n_1}({\underline{s}})}(F^{n_1}(t)))- \phi_{s_0}\circ\dots\circ\phi_{n_1-1}(\gamma^\infty_{\sigma^{n_1}(\widetilde{{\underline{s}}})}(F^{n_1}(t)))\right|\leq\]\[\leq \left| \gamma^\infty_{\sigma^{n_1}({\underline{s}})}(F^{n_1}(t))-\gamma^\infty_{\sigma^{n_1}(\widetilde{{\underline{s}}})}(F^{n_1}(t))\right|<\varepsilon ,\] ending the proof of statement (c).
	\end{enumerate}
\end{proof}
Observe that, by uniqueness, we have $ L^+=\gamma^\infty_{\overline{0}} $ and $ L^+=\gamma^\infty_{\overline{1}} $, implying, in particular, that $ L^\pm\subset \partial U $. Next, we use it to prove new characterization of $ \partial U $, which will be useful in the sequel. 

\begin{prop}{\bf (Characterizations of $ \partial U $)}\label{prop-caracteritzacio-frontera}	\begin{enumerate}[label={\em (\alph*)}]
		\item The boundary of $ U $ consists precisely of the points in $ \mathcal{J}(f) $ which never escape from $ S $, i.e. \[\partial U=\widehat{S}\cap\mathcal{J}(f).\]
		\item 	Every point in $ \partial U $ is in the closure of a dynamic ray, i.e. \[\partial U=\bigcup\limits_{{\underline{s}}\in\Sigma_2} \overline{\gamma^\infty_{\underline{s}}}.\]
	\end{enumerate}
\end{prop}
\begin{proof}
		\begin{enumerate}[label={(\alph*)}]
			\item Let us start by proving statement (a). To do so,	we show the following chain of inclusions:\[\partial U\subset \widehat{S}\cap \mathcal{J}(f)\subset\overline{\bigcup\limits_{n\geq0} \bigcup\limits_{\underline{s}\in\Sigma^n_2} \Phi_{\underline{s}}(L^\pm)} \subset \partial U,\]where $ \Sigma_2^n $ denotes the space of finite sequences of two symbols, $ \left\lbrace 0,1 \right\rbrace  $, of length $ n +1$; and if $ \underline{s}\in \Sigma_2^n $, $ \underline{s}=s_0\dots s_{n} $, then \[\Phi_{\underline{s}}\coloneqq\phi_{s_0}\circ\dots\circ \phi_{s_n}.\]
			
			The first inclusion comes straightforward from the definitions. To prove the second inclusion, consider $ z\in\widehat{S}\cap\mathcal{J}(f) $ and let $ W$ be a neighborhood of $ z $. Without loss of generality, we can assume $ z \notin L^\pm $ and $ W\subset S $. Since $ z\in\mathcal{J}(f) $, by the blow-up property, there exists $ n>0 $ such that $ f^n(W)\not\subset S $. But $ z\in\widehat{S} $, so $ f^n(z)\in S $. Therefore, $ f^n (W) $ intersects $ L^\pm $, and the result follows.
			
			Finally, regarding the third inclusion, it is enough to prove that $ \Phi_{\underline{s}}(L^\pm)\subset \partial U $, for all $ \underline{s}\in\Sigma_2^n $ and $ n\geq 0 $. Hence, fix $ n\geq 0 $ and $ \underline{s}\in\Sigma_2^n $, and consider $ z\in \Phi_{\underline{s}}(L^\pm) $. Since $ f^n(z)\in L^\pm \subset \partial U$,  there exists a sequence of points $\left\lbrace  w_n \right\rbrace _n\subset U $ such that $ w_n\to f^n(z) $. Applying $ \Phi_{\underline{s}} $ to the sequence $\left\lbrace  w_n \right\rbrace _n $,  we have $ \Phi_{\underline{s}}(w_n)\to z  $ with  $ \Phi_{\underline{s}}(w_n)\in U $, since  $ f^{-1}(U)\cap S=U $. Therefore, $ \Phi_{\underline{s}}(L^\pm)\subset \partial U $, as desired. See Figure \ref{fig-frontera-U}.
					
			\item	To prove statement (b), it is enough to show that, given an itinerary $ {\underline{s}}\in\Sigma_2 $, all points in $ \widehat{S}\smallsetminus U $ having this itinerary are precisely the ones in $ \overline{\gamma_{\underline{s}}^\infty} $.
	
	Let us assume first, that $ \underline{s}=\overline{0} $ and there is  $ z\in\partial U $ with this itinerary and $ z\notin L^+ $. Then, $ \textrm{Im }z<\pi $ and, since \[\textrm{Im }f(x+iy)=y-e^{-x}\sin y,\] it follows that there exists $ n\geq 0$ such that $ 0<\textrm{Im }f^n(z)<\frac{\pi}{2} $. Therefore,  $ f^n(z)\in \Omega_{01} $, so $ I(z) $ cannot be constant.
	The analogous argument works for $ \underline{s}=\overline{1} $ and, taking preimages, it also proves the statement for eventually constant sequences. 
	
	Now assume $ \underline{s} $ is a non-eventually constant sequence and there is $ z\in\widehat{S} $, with  $I(z)= \underline{s} $ and $ z\notin \overline{\gamma_{\underline{s}}^\infty } $. Since $  \overline{\gamma_{\underline{s}}^\infty }$ is closed in $ \mathbb{C} $, we have \[\rho (z, \overline{\gamma_{\underline{s}}^\infty })\coloneqq\inf\limits_{w\in\overline{\gamma_{{\underline{s}}}^\infty}} \rho (z,w)>0,\] where $ \rho $ is the distance in $ S\smallsetminus\overline{V} $ defined in \ref{def-rho-distance}.
	
	We note that, since $ f $ is expanding in $ S\smallsetminus\overline{V} $ with respect to $ \rho $, and $ f^n(z) \in S\smallsetminus\overline{V}$, $ f^n(\gamma_{{\underline{s}}}^\infty) \subset S\smallsetminus\overline{V}$, for all $ n\geq 0 $, it holds \[\rho (f^{n+1}(z), f^{n+1}(\overline{\gamma_{\underline{s}}^\infty }))>\rho (f^n(z), f^n(\overline{\gamma_{\underline{s}}^\infty })).\]
	
	Moreover, if both $ f^n(z) $ and $ f^n(\overline{\gamma_{\underline{s}}^\infty }) $ lie in $ \left\lbrace \textrm{Re }z<0 \right\rbrace  $, we have uniform expansion by constant $ \lambda>1  $ (see Rmk. \ref{remark_expansion}), i.e. \[\rho (f^{n+1}(z), f^{n+1}(\overline{\gamma_{\underline{s}}^\infty }))\geq \lambda \rho (f^n(z), f^n(\overline{\gamma_{\underline{s}}^\infty })).\]
	
	Since $ \underline{s }$ is non-eventually constant, there exists an infite increasing sequence $ \left\lbrace n_k\right\rbrace _k $ such that $ f^{n_k}(z), f^{n_k}(\overline{\gamma_{\underline{s}}^\infty })$ lie in $ \Omega_{01}$, so in particular they lie in the left half-plane $ \left\lbrace \textrm{Re }z<0 \right\rbrace  $, where $ f $ expands uniformly by factor $ \lambda >1$. Hence, since $ f $ is always expanding and expands infinitely many times uniformly by factor $ \lambda >1$, we get that \[ \rho (f^n(z), f^n(\overline{\gamma_{\underline{s}}^\infty }))\to\infty, \textrm{ as } n\to\infty.\]
	Hence, we can choose $ N>0 $ such that $ \rho (f^n(z), f^n(\overline{\gamma_{\underline{s}}^\infty }))>2+\pi$ and $ f^n(z)\in\Omega_{01} $, $ f^n(\overline{\gamma_{\underline{s}}^\infty })\subset\Omega_{01} $.
	
	 By construction, $ f^N({\gamma^\infty_{\underline{s}}}) $ contains the escaping tail $ \gamma_{\sigma^N (s)} $, which intersects the vertical segment  $ \left\lbrace z\in S\colon\textrm{Re } z=M\right\rbrace $. Observe that there are no points in $ \Omega_{01} $ at a distance greater than $ 2+\pi $ of $  \gamma_{\sigma^N({\underline{s}})}$, so this leads to a contradiction. 	\end{enumerate}
\end{proof}
			\begin{figure}[htb!]\centering
	\includegraphics[width=10cm]{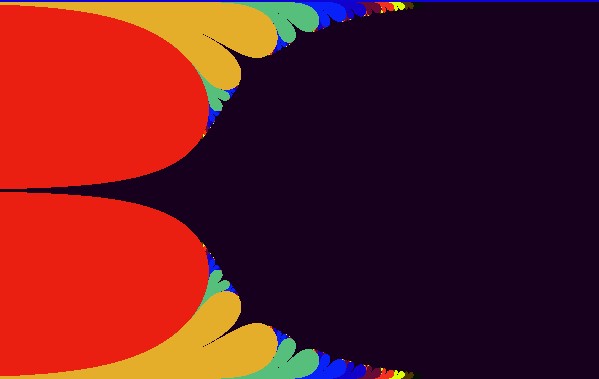}
	\caption{\footnotesize This picture shows the Baker domain (in black) and the regions (in different colors) which are eventually mapped outside $ S $. The boundaries of these regions are precisely $ \Phi_{\underline{s}}(L^\pm)$, $ \underline{s}\in\Sigma_2^n $, for some $ n\geq0 $.  Proposition \ref{prop-caracteritzacio-frontera} tells that $ \partial U $ is precisely the accumulation of those curves. }
	\label{fig-frontera-U}
\end{figure}

We note that the previous proposition allow us to characterize the points in $ \widehat{S} $. Indeed, as noted in Section \ref{sec-3-dynamics-f},  $U\cup \partial U\subset\widehat{S}$; and, from the fact that $ U $ has no more preimages in $ S $ apart from itself, $U\cap \mathcal{F}(f)=\widehat{S}$. The previous proposition characterizes $ \mathcal{J}(f)\cap\widehat{S} $,  implying the following corollary.
\begin{corol}{\bf (Characterization of $ \widehat{S} $)}\label{corol-charctS}
	It holds:\[\widehat{S}=\overline{U}=U\cup\partial U.\]
\end{corol}

From the results of this section, we shall deduce \ref{teo:A}.
\begin{proof}[Proof of \ref{teo:A}]
The first statement of \ref{teo:A} is deduced from statement (a) of Theorem \ref{teo-dynamic-ray}, whereas  the second statement of \ref{teo:A} corresponds to statement (b) in Proposition \ref{prop-caracteritzacio-frontera}.
\end{proof}

\section{Landing and non-landing rays: Proof of \ref{teo:Indec}}\label{sect-5-non-escaping}

We shall discuss now the landing properties of the dynamic rays defined in the previous section. More precisely, we devote the section to prove \ref{teo:Indec}, which asserts that for uncountably many sequences the dynamic ray  land at some point; while for uncountable many others the dynamic ray does not land and its accumulation set (in the Riemann sphere) is an indecomposable continuum. 

We proceed as follows. First of all, we define precisely what we mean for a ray to land, introducing the notion of \textit{landing set}. We also require the notion of \textit{non-escaping set} to relate the accumulation set of a dynamic ray with the non-escaping points having the same itinerary. Afterwards, we classify the sequences $ \underline{s}\in\Sigma_2 $ according to the nature of its landing set, resulting in the different landing behaviours claimed in \ref{teo:Indec}.

\begin{defi}{\bf (Landing set of a ray)}\label{def-end-set}	
		Let $ {\underline{s}}\in\Sigma_2 $ and let $ \gamma_{\underline{s}}^\infty $ be the dynamic ray of sequence $ {\underline{s}} $. We define the {\em landing set} $ L_{\underline{s}} $ of the ray $ \gamma_{\underline{s}}^\infty $ as the set of values $ w\in \widehat{\mathbb{C}} $ for which there is a sequence $ \left\lbrace t_{n}\right\rbrace _n\subset \mathbb{R} $ such that $ t_n\to  +\infty$ and $\gamma^\infty_{\underline s}(t_n)\to w $, as $ n\to\infty $. If $ L_{\underline{s}}=\left\lbrace w\right\rbrace  $,  we say that the dynamic ray $ \gamma_{\underline{s}}^\infty $ \textit{lands} at $ w $.
\end{defi}

Observe that, by Proposition \ref{prop-caracteritzacio-frontera}(b), $ \gamma^\infty_{\underline{s}}\cup L_{\underline{s}}$ contains all the points in $ \partial U $ with itinerary $ \underline{s} $, so \[
\gamma^\infty_{\underline{s}}\cup L_{\underline{s}}= \left\lbrace z\in\partial U\colon I(z)={\underline{s}}\right\rbrace .
\] Therefore, all non-escaping points with itinerary $ {\underline{s}}$ are in $ L_{\underline{s}} $, but {a priori} $ L_{\underline{s}}$ may contain escaping points. This leads us to define the following set.

\begin{defi}{\bf (Non-escaping set)}\label{def-non-escap-set}	
	Let $ {\underline{s}}\in\Sigma_2 $. We define the {\em non-escaping set} $ W_{\underline{s}}$ as the set of points in $ \widehat{S} $ with itinerary $ {\underline{s}} $ which do not escape to infinity.
\end{defi}

	Clearly,  $ W_{\underline{s}}\subset L_{\underline{s}}\cap \mathbb{C} $. Since all escaping points are in a ray, we have $ W_{\underline{s}}= \overline{\gamma^\infty_{\underline{s}}}\smallsetminus \gamma^\infty_{\underline{s}} $. Moreover, $ L_{\underline{s}} $ is always non-empty, compact and connected, whereas $ W_{\underline{s}} $ may be empty. 
	
	We start by describing $ L_{\underline{s}} $ and $ W_{\underline{s}} $ for eventually constant sequences.

\begin{lemma}{\bf (Eventually constant sequences)}\label{lemma-eventually-constant}
	Let $ {\underline{s}}\in\Sigma_2 $.
Then, $ L_{\underline{s}} =\left\lbrace \infty\right\rbrace $ if, and only if, $ {\underline{s}}$ is  eventually constant. In this case, $ W_{\underline{s}}=\emptyset $.
\end{lemma}
\begin{proof}
Recall that $ \gamma^\infty_{\overline{0}}= L^+$ and $ \gamma^\infty_{\overline{1}}=L^-$, so $ L_{\overline{0}}= L_{\overline{1}}=\left\lbrace \infty\right\rbrace $. Since preimages of curves landing at $ \infty $ are again curves landing at $ \infty $ and hairs with eventually constant sequence are the preimages of $ L^\pm $, one implication is proven.

Now, assume $ \underline{s} $ is a non-eventually constant sequence, and $ \gamma^\infty_{\underline{s}} $ lands at $ \infty $. Then, $ \gamma^\infty_{\underline{s}} $ divides $ S $ into two regions: $ R_1, R_2$. The absorbing domain $ V $ is contained in one of them, say $ R_1 $, so $ R_1\cap U\neq \emptyset $. We claim that $ R_2\cap U\neq \emptyset $. Indeed, $ R_2\cap \widehat{S}\neq \emptyset $, because the points that leave $ S $ after applying $ f $ are the ones enclosed by $ f^{-1}(L^\pm)\cap S $, and $ \gamma^\infty_{\underline{s}}  $ is not a preimage of $ L^\pm $ (see Fig. \ref{fig-absorbing}, \ref{fig-frontera-U}). The fact that $ U= \textrm{Int}(\widehat{S}) $ (Corol. \ref{corol-charctS}) gives that $ R_2\cap U\neq \emptyset $.
This is a contradiction because $ U $ is connected.
\end{proof}

The goal for the remaining part of the section is to describe the landing and the non-escaping sets for non-eventually constant sequences. First, we deal with the dynamics of the non-escaping points, whose orbit may be bounded or oscillating. It turns out that this only depends on  its itinerary. Moreover, for certain types of sequence, we have a great control on the  orbit of the ray and the non-escaping set, as the following results show.
\begin{defi}{\bf (Types of sequences)}\label{def-bounded-seq}
	Let $ {\underline{s}}\in\Sigma_2 $ be a non-eventually constant sequence. We say that $ {\underline{s}} $ is {\em oscillating} if it contains arbitrarily large sequences of 0's or 1's. Otherwise, we say that $ {\underline{s}} $ is {\em bounded}.
\end{defi}
\begin{prop}{\bf (Dynamics on the non-escaping sets)}\label{prop-dynamics-non-escaping}
Let $ {\underline{s}}\in\Sigma_2 $ and let $ W_{\underline{s}} $ be its corresponding non-escaping set. Then, $ \left\lbrace f^n(W_{\underline{s}})\right\rbrace _n $ is contained in a compact set if and only if $ {\underline{s}}$ is a bounded sequence. In this case, 
 there exists $ R>0 $ such that $f^n (\gamma^\infty_{\underline{s}} )\subset \left\lbrace \textrm{\em Re } z <R\right\rbrace $ and $ f^n(W_{\underline{s}}) \subset \left\lbrace \left| \textrm{\em Re } z\right|  <R\right\rbrace$.
\end{prop}
\begin{proof}
Assume first that $ {\underline{s}}\in\Sigma_2 $ is a bounded sequence and $ z\in W_{\underline{s}} $. Then, there exists $ N>0 $ such that $ {\underline{s}} $ does not contain more than $ N $ consecutive 0's and $ N $ consecutive 1's. Take $R \coloneqq F^{-N}(0) $, where $ F(t)=t-e^{-t} $. We claim that $ \textrm{Re }f^n(z)\leq R  $ for all $ n $. Indeed, if it is not the case, there must exist $ n_0 $ such that $ \textrm{Re }f^{n_0}(z)>R $. Then, since $ F $ is increasing, we have \[ \textrm{Re }f^{n_0+1}(z)> \textrm{Re }f^{n_0}(z)-e^{\textrm{Re }f^{n_0}(z)}>R-e^{-R}=F(R).\] Repeating the argument inductively, we get \[\textrm{Re }f^{n_0+N}(z)> \textrm{Re }f^{n_0+N-1}(z)-e^{\textrm{Re }f^{n_0+N-1}(z)}>F^{N-1}(R)-e^{-F^{N-1}(R)}=F^N(R)=0
.\] Therefore, by Lemma \ref{lemma-properties-Sij} (a),  either $ \left\lbrace f^{n_0+k}(z)\right\rbrace _{k=0}^N\subset \Omega_0 $ or $ \left\lbrace f^{n_0+k}(z)\right\rbrace _{k=0}^N\subset \Omega_1 $, so $ \underline{s} $ has $ N+1 $ consecutive 0's. Therefore, $\textrm{Re } f^n(z)< R$, for all $ n\geq 0 $. We note that the constant $ R $ has been chosen to depend only on $ N $ (but not on the particular point $ z\in W_{\underline{s}} $), hence it holds $\textrm{Re } f^n(z)< R$, for all $ n\geq 0 $ and $ z\in W_{\underline{s}} $.

We claim that, under the conditions described above, if $ R $ has been chosen large enough, we also have $ \textrm{Re }f^n(z)>-R $ for all $ n\geq 0 $ and $ z\in W_{\underline{s}} $, and hence $ \left\lbrace f^n(W_{\underline{s}})\right\rbrace _n $ is contained in a compact set.

Without loss of generality, we can assume $ R>3 $ and large enough so that if $ z\in\partial U $ and $ \textrm{Re }z<-R $, then $ \left| \textrm{Im }z\right| <\left( 0, \frac{\pi}{3}\right)$ or $\frac{2\pi}{3} <\left| \textrm{Im }z\right| <1 $. We note that this is possible since $ \phi_1(L^+) $ is a curve landing at $ -\infty $ from both sides, approaching tangentially $ L^- $ and $ \mathbb{R} $; and $ \phi_0(L^-) $ also at lands at $ -\infty $, but approaching $ L^+ $ and $ \mathbb{R} $ (see e.g. Fig. \ref{fig-absorbing}). Then, in order to show that $ \textrm{Re }f^n(z)>-R $ for all $ n\geq 0 $ and $ z\in W_{\underline{s}} $, we proceed by contradiction: let us assume that there exists $ z\in W_{\underline{s}} $ and $ n_0\geq 0 $ such that $ \textrm{Re }f^{n_0}(z)<-R $. Then, since $ z\in W_{\underline{ s}} $, we can assume that $ \textrm{Re }f^{n_0+1}(z)>-R $. Hence, $ 0<\left| \textrm{Im }f^{n_0}(z)\right| <\frac{\pi}{3}$. Then, \[\textrm{Re }f^{n_0+1}(z)=\textrm{Re }f^{n_0}(z)+ e^{-\textrm{Re }f^{n_0}(z)}\cos (\textrm{Im }f^{n_0}(z))\geq \]\[ \geq\textrm{Re }f^{n_0}(z)+\frac{1}{2}e^{-\textrm{Re }f^{n_0}(z)}>-\textrm{Re }f^{n_0}(z)>R. \]This contradicts the assumption that $\textrm{Re } f^n(z)< R$, for all $ n\geq 0 $ and $ z\in W_{\underline{s}} $, proving one implication.

For the other implication, let us assume that $ z\in W_{\underline{s}} $ has a bounded orbit, and let us prove that then $ {\underline{s}} $ is bounded. Let $ R>0 $ be such that $ -R\leq \textrm{Re }f^n(z)\leq R $, for all $ n\geq 0 $, and let $ \varepsilon>0 $ be such that $ \textrm{dist }(f^n(z), V)>\varepsilon $, for all $ n\geq 0 $. We note that, in this case, if $ f^n(z)\in \Omega_{00}\cup\Omega_{01} $, $ \left| \textrm{Im }f^n(z)\right|> \frac{\pi}{2}+\varepsilon$. Let $ M =\left| e^{-R}\cos (\frac{\pi}{2}+\varepsilon)\right| $, and let $ N $ be such that $ R-N M <-R $. We claim that $ \underline{s} $ cannot have more than $ N $ consecutive 0's. On the contrary, assume $ \left\lbrace f^n(z) \right\rbrace _{n=0}^N\subset\Omega_{00}$. Then,
\[\textrm{Re }f^n(z)=\textrm{Re }f^n(z)+e^{-\textrm{Re }f^n(z)}\cos(\textrm{Im }f^n(z))<\textrm{Re }z-M,\] for $ 0\geq n\geq N-1 $,
so \[\textrm{Re }f^N(z)<\textrm{Re }z-NM<-R.\]

Therefore, $ \underline{s} $ cannot have more than $ N $ consecutive 0's. A similar argument can be used to prove that   $ \underline{s} $ cannot have more than $ N $ consecutive 1's; and this proves the other implication.

The existence of  $ R>0 $ such that $f^n (\gamma^\infty_{\underline{s}} )\subset \left\lbrace \textrm{\em Re } z <R\right\rbrace $ and $ f^n(W_{\underline{s}}) \subset \left\lbrace \left| \textrm{\em Re } z\right|  <R\right\rbrace$  is deduced from the previous reasoning, taking into account that escaping points with bounded itinerary cannot go arbitrarily far to the right, since they have to be in $ \Omega_{01} $ (or $ \Omega_{10} $) in a  bounded number of steps.
\end{proof}

Next, we use this control on the dynamic rays and the non-escaping sets for bounded sequences to prove that the non-escaping set is actually a point where the dynamic ray lands.

\begin{prop}{\bf (Rays with bounded sequence land)}\label{prop-bounded-rays-land}
	Let 	$ {\underline{s}}\in \Sigma_2 $ be a bounded sequence. Then, there exists a point $ w_{\underline{s}}\in\mathbb{C} $ such that \[L_{\underline{s}}=W_{\underline{s}}=\left\lbrace w_{\underline{s}} \right\rbrace ,\]  i.e. the dynamic ray $ \gamma^\infty_{\underline{s}}(t) $ lands at the point $ w_{\underline{s}} $.
\end{prop}
\begin{proof} First, let us prove that $ W_{\underline{s}}$ consists of a single point.  By Proposition \ref{prop-dynamics-non-escaping}, $ W_{\underline{s}}$ is compact. Assume, on the contrary that $ W_{\underline{s}} $ consists of more than one point, so $ \textrm{diam}_\rho(W_{\underline{s}}) >0$. Recall that $ f $ is uniformly expanding in any compact set $ K \subset S\smallsetminus\overline{V}$ with respect to $ \rho $ (see Rmk. \ref{remark_expansion}). Taking $ K $ to be $ W_{\underline{s}} $, we have $ \textrm{diam}_\rho(f^n(W_{\underline{s}}))\to \infty $, which contradicts the fact that $\left\lbrace  f^n(W_{\underline{s}})\right\rbrace _n$ is contained in a compact set (Prop. \ref{prop-dynamics-non-escaping}). Therefore, $ W_{\underline{s}} $ must consist only of one point, so $ W_{\underline{s}}=\left\lbrace w_s\right\rbrace $.
		
To end the proof, it is enough to show that $ L_{\underline{s}}$ cannot contain any escaping point. Indeed, this would imply, together with the previous lemma, that $ L_{\underline{s}}\subset \left\lbrace w_{\underline{s}}, \infty\right\rbrace  $ and, since $ L_{\underline{s}} $ is connected and it cannot be equal to $ \infty $, necessarily $ L_{\underline{s}}=\left\lbrace w_{\underline{s}}\right\rbrace  $.
	
	By Proposition \ref{prop-dynamics-non-escaping}, $ f^n(\gamma^\infty_{\underline{s}} )$ and $f^n (W_{\underline{s}})$ are contained in the half-plane $  \left\lbrace \textrm{Re } z <R\right\rbrace $. Assume the dynamic ray $ \gamma^\infty_{\underline{s}} $ accumulates at an escaping point $ z $. Since $ z $ is escaping and has itinerary $ \underline{s} $, by Theorem \ref{teo-dynamic-ray}(a), there exists $ n_0\geq 0$ such that $ f^{n_0}(z)\in\gamma_{\sigma^{n_0}({\underline{s}})} $ and $ \textrm{Re } f^{n_0}(z) <-R $. 
	
	We note that $ {\sigma^{n_0}({\underline{s}})} $ is also a bounded sequence satisfying that $ f^n(\gamma^\infty_{\sigma^{n_0}({\underline{s}})} )\subset \left\lbrace \textrm{Re } z <R\right\rbrace$; and $ f^{n_0}(z) $ is escaping and 
 $ f^{n_0}(z)\in L_{\sigma^{n_0}({\underline{s}})}$. 	Therefore, there exists 
	 an increasing sequence $ \left\lbrace t_n\right\rbrace_n  \subset \mathbb{R}$ and $ w_n\coloneqq \gamma^\infty_{\underline{s}}(t_n)\to f^{n_0}(z) $, as $ n\to\infty $. Let us choose some $ m $ such that $ t_m\geq -2 $, and hence $ w_m\in\gamma_{\underline{s}}^\infty\smallsetminus\gamma_{\underline{s}} $, and $ \textrm{Re } w_m <-R $. Since $ w_m $ is not in the escaping tail, there exists $ M>0 $ such that $ \textrm{Re } f^M (w_m) >-2+\pi $, $ M $ being the minimal integer satisfying this property. Hence, $ \textrm{Re } f^{M-1}(w_m) <-R $, so $ \textrm{Re } f^{M}(w_m)> R $ (since $ \left| f(z)\right| > \textrm{Re }z $, as shown in the proof of Prop. \ref{teo-curves-of-escaping-points}). 
	 
	 Therefore the property $ f^n(\gamma^\infty_{\sigma^{n_0}({\underline{s}})} )\subset \left\lbrace \textrm{Re } z <R\right\rbrace$ does not hold, leading to a contradiction.
\end{proof}

To end the section, we prove that rays with oscillating sequences do not always land. 
In fact, we are going to prove that, for uncountably many sequences, $ L_{\underline{s}} $ is an indecomposable continuum which contains the ray $ \gamma_{\underline{s}}^\infty $. We follow the ideas of Rempe (\cite[Thm. 3.8.4]{tesi-lasse-rempe}, \cite[Thm. 1.2]{rempe2007}). 

\begin{prop}{\bf (Some rays do not land)}\label{prop-non-landing}
There exist uncountably many dynamic rays $ \gamma^\infty_{\underline{s}} $ which do not land. 
\end{prop}
\begin{proof}
	First, by Lemma \ref{lemma-eventually-constant}, if we show that, for a non-eventually constant sequence $ {\underline{s}} $, the landing set $ L_{\underline{s}} $ contains $ \infty $, then the ray $ \gamma^\infty_{\underline{s}} $ do not land. Hence, our goal is to construct a non-eventually constant sequence $ {\underline{s}}$ with $ \infty\in L_{\underline{s}} $. 
	
	Let us denote by $ \overline{0}^n $ a block of $ n $ zeroes and by $ \overline{0}$ an infinite block of zeroes. Then, the itinerary $ {\underline{s}} $ that we construct will be of the form ${\underline{s}}= 1\overline{0}^{n_1}1\overline{0}^{n_2}1\overline{0}^{n_3}\dots$ for an infinite sequence $ \left\lbrace n_j\right\rbrace _j$. We choose the $  n_j $'s inductively among countably many choices in each step, leading to uncountably many non-landing rays at the end. 
	
	Assume $ n_1,\dots, n_{j-1} $ have been chosen, and consider the sequence $ {\underline{s}}^j=1\overline{0}^{n_1}\dots1\overline{0}^{n_{j-1}}1\overline{0}$. Then, $ \gamma^\infty_{{\underline{s}}^j} $ is a preimage of $ L^+ $, so it lands at $ \infty $ in both ends. Let us choose $ t_j >-2$ such that $ \left| \gamma_{{\underline{s}}^j}(t_j)\right| >j $. By Theorem \ref{teo-dynamic-ray} (c), there exists $ N_j\in\mathbb{N} $ such that $ \left| \gamma_{{\underline{s}}}(t_j)\right| \geq j $ for all $ s\in\Sigma_2({\underline{s}}^j, N_j) $. We choose $ n_j\geq N_j $.
	
	Let $ {\underline{s}} $ be the sequence constructed in this way. Then, $ {\underline{s}} $  is clearly non-eventually constant, and $ \infty\in L_{\underline{s}} $, since $ \gamma^\infty_{\underline{s}}(t_j)\to\infty $, as $ j\to\infty $,  proving that the ray $ \gamma^\infty_{\underline{s}} $ does not land. Evidently, by symmetry, the same construction interchanging 0's by 1's also gives non-landing rays.
\end{proof}
	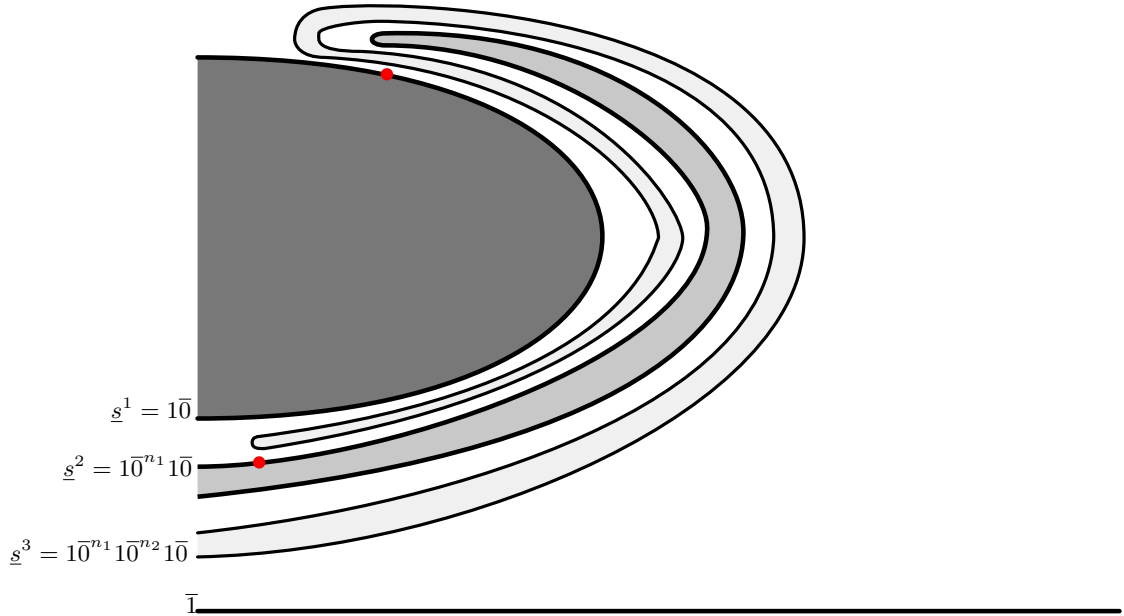
\begin{figure}[htb!]\centering
	\includegraphics[width=14cm]{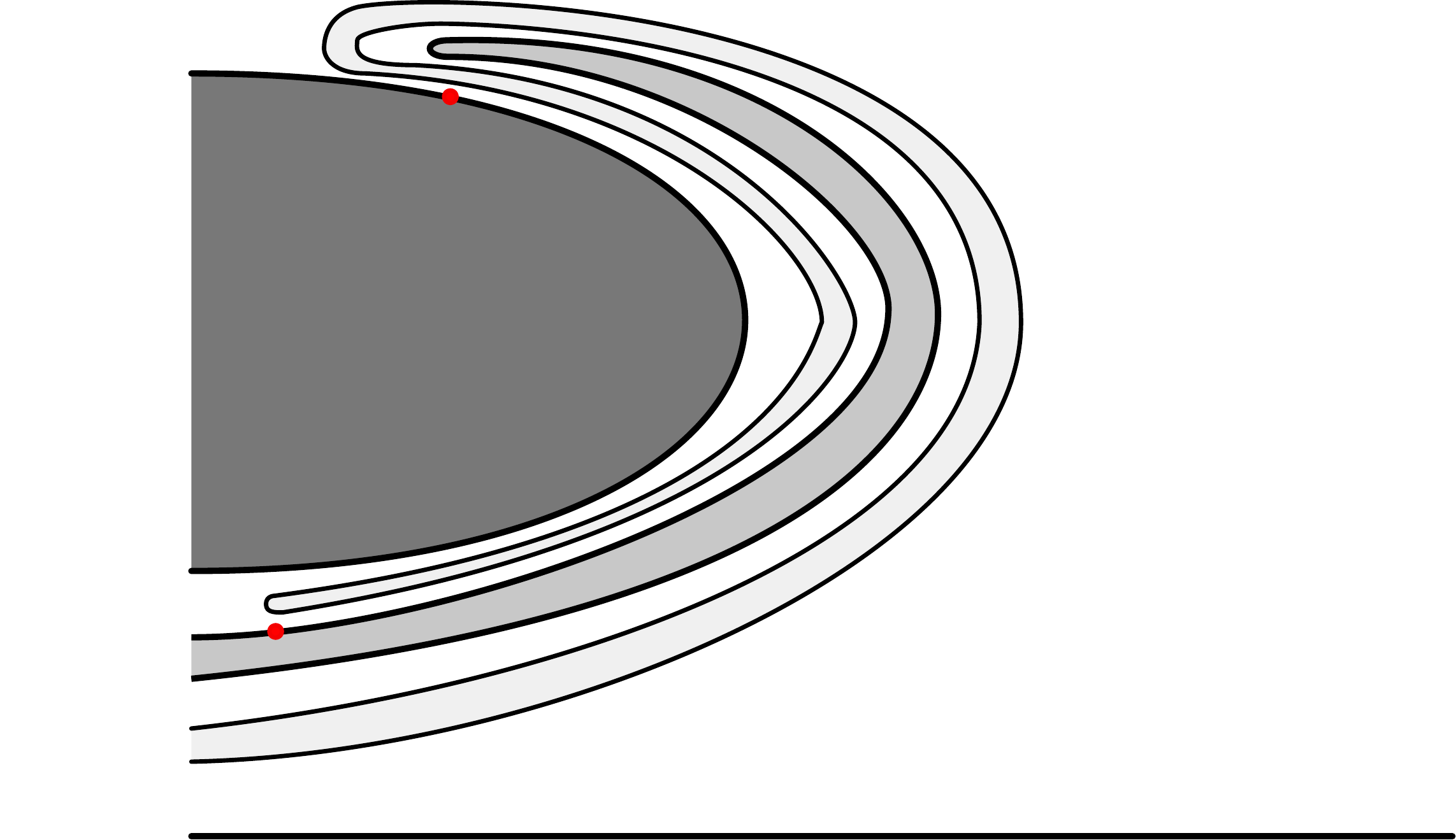}
	\setlength{\unitlength}{14cm}
	\put(-0.95, 0.185){\footnotesize ${\underline{s}}^1= 1\overline{0} $}
	\put(-0.995, 0.13){\footnotesize ${\underline{s}}^2= 1\overline{0}^{n_1}1\overline{0} $}
	\put(-1.045, 0.05){\footnotesize ${\underline{s}}^3= 1\overline{0}^{n_1}1\overline{0}^{n_2}1\overline{0} $}
	\put(-0.88, 0){\footnotesize $\overline{1} $}
	\caption{\footnotesize Schematic representation of the construction of the non-landing ray $ \gamma_{{\underline{s}}}^\infty $, to give a geometric intuition of the proof, showing the first three steps of the induction. The sequence on the right indicates the itinerary of the ray. The first ray that is constructed is the one of sequence $ {\underline{s}}^1 $, which is a preimage of $ L^+ $. In red, it is marked the point $ \gamma_{{\underline{s}}^1}^\infty(t_1) $. In the next step of the induction, it is chosen $ {\underline{s}}_2 $ in such a way that $ \gamma_{{\underline{s}}^2}^\infty $ gets \textit{close} to  $ \gamma_{{\underline{s}}^1}^\infty(t_1) $, so $ \gamma_{{\underline{s}}^2}^\infty $ wraps along $ \gamma_{{\underline{s}}^1}^\infty $. This wrapping is precisely what makes that, in the limit, we get a non-landing ray.}
\end{figure}
\begin{corol}{\bf (Some landing sets are indecomposable continua)}\label{corol-indec}
	The landing set $ L_{\underline{s}} $ of the non-landing rays of Proposition \ref{prop-non-landing} is an indecomposable continuum.
\end{corol}
\begin{proof}
	To prove that $ L_{\underline{s}}  $ is an indecomposable continuum, we shall invoke Curry's Theorem \ref{teo-curry}, after checking that $ L_{\underline{s}} $ does not separate the plane and that $ \gamma^\infty_{\underline{s}}\subset L_{\underline{s}}$.
	
		On the one hand, let us observe that effectively $ L_{\underline{s}}$ cannot separate the plane. We follow the same argument as in the proof of Lemma \ref{lemma-eventually-constant}. Indeed, if $ L_{\underline{s}}$ separates $ \mathbb{C} $, it should also separate the strip $ S $. Let $ R_1 $ be the connected component of $ S\smallsetminus L_{\underline{s}} $ that contains the absorbing domain $ V $, so $ R_1\cap U\neq \emptyset $. Let $ R_2 $ be any other component of $ S\smallsetminus L_{\underline{s}} $. We claim that $ R_2\cap U\neq \emptyset $. Indeed, $ R_2\cap \widehat{S}\neq \emptyset $, because the points that leave $ S $ after applying $ f $ are the ones enclosed by $ f^{-1}(L^\pm)\cap S $, and $ L_{\underline{s}}  $ is not a preimage of $ L^\pm $. The fact that $ U= \textrm{Int}(\widehat{S}) $ gives that $ R_2\cap U\neq \emptyset $.
	This is a contradiction because $ U $ is connected. We note that this argument not only proves that $ L_{\underline{s}}$ cannot separate the plane, but also that neither $ \gamma^\infty_{\underline{s}} $ nor $ \overline{\gamma^\infty_{\underline{s}}} $ can separate the plane.
	
	On the other hand, the proof that $ \gamma^\infty_{\underline{s}}\subset L_{\underline{s}} $ follows the idea of Rempe (\cite[Lemma 3.3]{rempe2007}) based on the fact that dynamic rays accumulate among them. Indeed, $ L_{\underline{s}} $ cannot intersect any dynamic ray different from $ \gamma^\infty_{\underline{s}} $. In particular, $ L_{\underline{s}} $ does not intersect the dynamic rays $ \gamma^\infty_{\underline{r}^n} $, defined by \[ \underline{r}^n \coloneqq s_0s_1\dots s_{n-1}r_n s_{n+1}s_{n+2}\dots,\]where $ r_n=0 $, if $ s_n=1 $, and $ r_n=1 $, if $ s_n=0 $. By Theorem \ref{teo-dynamic-ray} (c), it is clear that $ \gamma^\infty_{\underline{r}^n}\to \gamma^\infty_{\underline{s}} $, as $ n\to\infty $, uniformly on every interval $ \left( -\infty, t_0\right]  $, $ t_0\in\mathbb{R} $. Moreover, from the fact that $ \underline{s} $ is not eventually constant and escaping tails are ordered vertically following the (inverse) lexicographic order, it follows that $ \left\lbrace \gamma^\infty_{\underline{r}^n}\right\rbrace _n $ approximates $  \gamma^\infty_{{\underline{s}}}$ from above and from below. Therefore, we redefine the previous sequences as $ \underline{r}^{n,+}\coloneqq \underline{r}^m$, if $ m\leq n $ is the maximal such that $ \underline{r}^m>\underline{s}$ in the inverse lexicographic order; and $ {\underline{r}^{n,-}}\coloneqq \underline{r}^m$, if $ m\leq n $ is the maximal such that $ \underline{r}^m<\underline{s} $ in the inverse lexicographic order. Hence, the sequence of rays $ \left\lbrace \gamma_{\underline{r}^{n,+}}^\infty \right\rbrace_n $ approximates $ \gamma^\infty_{\underline{s} }$ from above; and $ \left\lbrace \gamma_{\underline{r}^{n,-}}^\infty \right\rbrace_n $ from below.
	
	Now, assume that $ \gamma^\infty_{\underline{s}}\not\subset L_{\underline{s}} $, so we can find $ t_0 $ such that $ \varepsilon\coloneqq \textrm{dist }(\gamma_{\underline{s}}^\infty(t_0), L_{\underline{s}})>0 $. Since $ \infty\in L_{\underline{s}}$ and points in $ L_{\underline{s}} $ must have itinerary $ {\underline{s}} $, it follows that $ L_{\underline{s}} $ is contained in the connected component $ U_n $ of \[\mathbb{C}\smallsetminus \left( D(\gamma_{\underline{s}}^\infty(t_0),\varepsilon)\cup \gamma^\infty_{\underline{r}^{n,+}}\cup \gamma^\infty_{\underline{r}^{n,-}}\right), \] which contains $ \gamma^\infty_{\underline{s}}(t) $, for all $ t\leq t_1 $, for some $ t_1<t_0 $. Therefore, $ L_{\underline{s}}\subset\bigcap\limits_n U_n\subset \gamma_{\underline{s}}^\infty(\left( -\infty, t_0\right] ) $. In such a case, $ \overline{\gamma_{{\underline{s}}}^\infty} $ would separate the plane into (at least) two different connected components, what we have proved before that it is not possible. Therefore, $ \gamma^\infty_{\underline{s}}\subset L_{\underline{s}} $, as desired.

	Then, it follows from Curry's Theorem \ref{teo-curry} that $ L_{\underline{s}} $ is an indecomposable continuum, as desired. 
\end{proof}

Finally, we prove \ref{teo:Indec}.
\begin{proof}[Proof of \ref{teo:Indec}]
	The existence of uncountably many rays that land follows from Proposition \ref{prop-bounded-rays-land} (observe that there are uncountably many bounded sequences), whereas the existence of uncountably many non-landing rays follows from Proposition \ref{prop-non-landing}. On Corollary \ref{corol-indec}, we prove that the accumulation set of such non-landing rays is an indecomposable continuum.
\end{proof}
\section{Accessibility from $ U $ of points in $ \partial U $: Proof of \ref{teo:B}}\label{sect-6-accessibility}

This section is devoted to the proof of \ref{teo:B}, which relates   the accessibility from $ U $ with the previously  studied sets: the escaping set, the non-escaping sets and the landing sets. In particular, \ref{teo:B} asserts that all boundary points in the escaping set are non-accessible, while points in $ \partial U $ having a bounded orbit are accessible. 

First of all, let us choose as a Riemann map the function $ \varphi\colon\mathbb{D}\to U $ such that $ \varphi (0)=0 $ and $ \varphi(\mathbb{R}\cap\mathbb{D})=\mathbb{R} $, as in \cite{baker-dominguez}. 
With this choice, the associated inner function is 
\[g(z)=\dfrac{3z^2+1}{3+z^2}.\] It is easy to check that the
Denjoy-Wolff point of $ g $ is 1. Moreover, since $ g $ is a Blaschke product of degree 2 (and hence there are no critical points in the unit circle),
$ g _{|\partial\mathbb{D}}$ is a 2-to-1 covering of $ \partial \mathbb{D} $, being 1 the only fixed point. In particular, the preimages of 1 under $ g $ are itself and $ -1 $, since $ \varphi(\mathbb{R}\cap\mathbb{D})=\mathbb{R} $ and $ f(-\infty)=+\infty$.

Let us consider the following subsets of the (closed) unit disk\[ D_0\coloneqq \overline{\mathbb{D}}\cap \left\lbrace \textrm{Im }z> 0\right\rbrace \hspace{1cm}D_1\coloneqq \overline{\mathbb{D}}\cap \left\lbrace \textrm{Im }z< 0\right\rbrace,\] as shown in Figure \ref{fig-radial-limits}. We define the itinerary for a point in $ \partial\mathbb{D} $ in the following way. 

\begin{defi}{\bf (Itineraries in $ \partial\mathbb{D} $)}\label{def-itinerary-D}
	Let $ e^{i\theta}\in\partial\mathbb{D} $. If $ g^n(e^{i\theta})\neq 1 $, for all $ n\geq0 $, then the \textit{itinerary} of $ e^{i\theta}$ is defined as the sequence $ \mathscr{S}(e^{i\theta})={\underline{s}}=\left\lbrace s_n\right\rbrace _n\in \Sigma_2 $ satisfying $ g^n(e^{i\theta})\in D_{s_n} $. 
	
	If there exists $ n_0\geq 0 $ such that $ g^{n_0}(e^{i\theta})=1 $, then the \textit{itineraries} of $ e^{i\theta}$, $ \mathscr{S}(e^{i\theta}) $, are defined as the two sequences $ {\underline{s}}^j = \left\lbrace s^j_n\right\rbrace _n\in \Sigma_2$, $ j=0,1 $, satisfying $ g^n( e^{i\theta})\in D_{s_n} $ for $ n\leq n_0-2$, $ s^0_{n_0-1}=1 $, $ s^1_{n_0-1}=0 $ and $ s^j_n=j $, for $ n\geq n_0 $.
\end{defi}

Hence, we have just defined a multivalued function \[\mathscr{S}\colon\partial \mathbb{D}\longrightarrow\Sigma_2.\] We note that, since every point in $ \partial\mathbb{D} $ has an itinerary, the domain of $ \mathscr{S} $ is $ \partial\mathbb{D} $. Moreover,  we claim that $ \mathscr{S} $ is injective, i.e. that two different points in the unit circle cannot have the same itinerary.
This is
due to the expansiveness of the map $ g_{|\partial\mathbb{D}} $. Indeed,
\[g'(z)= \dfrac{-16z}{(3z^2+1)^2},\] and hence, for $ e^{i\theta}\in\partial\mathbb{D} $, it holds \[\left| g'(e^{i\theta})\right| = \dfrac{16\left| e^{i\theta}\right| }{\left| 3e^{i 2\theta}+1\right| ^2}\geq \dfrac{16 }{ (3\left|e^{i 2\theta}\right| +1)^2}\geq 1.\]

We also shall consider its inverse \[\mathscr{S}^{-1}\colon\Sigma_2\longrightarrow\partial \mathbb{D},\] which is a single-valued function. Moreover, $ \mathscr{S}^{-1} $ is  surjective, but not injective, and commutes with the shift map $ \sigma $ in $ \Sigma_2 $.

Since $ \mathscr{S} $ is only multivalued when considering eventual preimages of 1, it follows that $ \mathscr{S} $ is a bijection if we restrict ourselves to non-eventually constant sequences in $ \Sigma_2 $ and points in $ \partial\mathbb{D} $ which are not eventual preimages of $ 1 $.

The following proposition is the key result which  relates itineraries in $\partial \mathbb{D} $ and in $ \widehat{S} $, and will clarify the choice of the itineraries in $ \partial\mathbb{D} $.
\begin{prop}{\bf (Correspondence between itineraries)}\label{prop_correspondence_it}
 Let $ e^{i\theta}\in\partial \mathbb{D} $. If $ g^n(e^{i\theta})\neq 1 $, for all $ n\geq0 $, and $ {\underline{s}}=\mathscr{S}  (e^{i\theta}) $ then $ Cl(\varphi, e^{i\theta})=\overline{\gamma_{\underline{s}}^\infty}  $. If there exists $ n_0\geq 0 $ such that $ g^{n_0}(e^{i\theta})=1 $ and $ \left\lbrace {\underline{s}}^0, {\underline{s}}^1\right\rbrace  =\mathscr{S}  (e^{i\theta}) $, then $ Cl(\varphi, e^{i\theta})= \overline{\gamma_{{\underline{s}}^0}^\infty} \cup \overline{\gamma_{{\underline{s}}^1}^\infty}   $.
\end{prop}

\begin{proof} 

Observe that, according to the chosen Riemann map $ \varphi\colon \mathbb{D}\to U $, it holds that
 $ \varphi(\textrm{Int } D_0)\subset {\Omega_0} $ and  $ \varphi(\textrm{Int } D_1)\subset {\Omega_1} $ (see Fig. \ref{fig-radial-limits}). Moreover, $ \varphi((-1,1))=\mathbb{R}\subset U $. 
 
 Hence, if $ e^{i\theta} \in\partial\mathbb{D}$ and $ e^{i\theta}\notin \left\lbrace -1,1\right\rbrace  $, then $ e^{i\theta}\in D_i$, and so does a neighbourhood of $ e^{i\theta} $ in $\overline{\mathbb{D}}$. Hence, $ Cl(\varphi, e^{i\theta})\subset \Omega_i $, for some $ i\in \left\lbrace 0,1\right\rbrace  $. By continuity of $ g $, every sequence in $ \mathbb{D} $ converging to $ e^{i\theta} $ maps under $ g $ to a sequence converging to $ g(e^{i\theta}) $. If $ e^{i\theta} \in\partial\mathbb{D}$ is not a preimage of 1, then $ g(e^{i\theta})\notin \left\lbrace -1,1\right\rbrace  $, so $ f(Cl(\varphi, e^{i\theta})\cap\mathbb{C})\subset \Omega_j $, for some $ j\in \left\lbrace 0,1\right\rbrace  $.  Repeating inductively the same argument, we get that the itinerary of $ e^{i\theta} $ determines completely the itinerary of points in $ Cl(\varphi, e^{i\theta}) $, so \[ Cl(\varphi, e^{i\theta})\subset \left\lbrace z\in\partial U\colon I(z)=\mathscr{S} (e^{i\theta})\right\rbrace \cup \left\lbrace \infty\right\rbrace  ,\] if $ e^{i\theta} \in\partial\mathbb{D}$ is not an eventual preimage of 1.
 
 On the other hand, consider $ 1\in\partial\mathbb{D} $, $ \mathscr{S} (1)=\left\lbrace {\overline{0}}, {\overline{1}}\right\rbrace $. We note that, for any sequence of points $ \left\lbrace w_k\right\rbrace _k\subset D_0 $ converging to 1, and for all $ n\geq 0 $, there exists $ k_0=k_0(n) $ such that $  \left\lbrace g^n( w_k)\right\rbrace _{k\geq k_0}\subset D_0  $; and we observe that 1 is the only point in $ \partial \mathbb{D} $ with this property. Similarly, if $z\in \partial U $ and for any sequence $ \left\lbrace z_k\right\rbrace _k\subset \Omega_0 $ converging to $ z $,  for all $ n\geq 0 $ there exists $ k_0 $ such that $  \left\lbrace f^n( z_k)\right\rbrace _{k\geq k_0}\subset \Omega_0  $, then $ z\in L^+ $. Therefore,  for any sequence $ \left\lbrace w_n\right\rbrace _n\subset D_0 $, $ w_n\to 1 $,  any accumulation point of $   \left\lbrace \varphi (w_n)\right\rbrace _n $ must be in $ L^+\cup \left\lbrace \infty\right\rbrace  $.
 The analogous argument works similarly with $ D_1 $ and $ L^- $.  
 Hence, \[Cl(\varphi, e^{i\theta})\subset L^+\cup L^-\cup \left\lbrace \infty\right\rbrace= \left\lbrace z\in\partial U\colon I(z)\in \left\lbrace {\overline{0}}, {\overline{1}}\right\rbrace\right\rbrace \cup \left\lbrace \infty\right\rbrace . \] 
 
Therefore, if $ e^{i\theta} $ is an eventual preimage of 1, and hence  $\mathscr{S}  (e^{i\theta})= \left\lbrace {\underline{s}}^0, {\underline{s}}^1\right\rbrace  $, it holds
 \[ Cl(\varphi, e^{i\theta})\subset \left\lbrace z\in\partial U\colon I(z)\in \left\lbrace {\underline{s}}^0, {\underline{s}}^1\right\rbrace\right\rbrace \cup \left\lbrace \infty\right\rbrace  .\]

We note that, given two different sequences $ \underline{ r}, \underline{ s} \in\Sigma_2$, the sets of points in $ \partial U $ having these itineraries are disjoint, i.e.\[\left\lbrace z\in\partial U\colon I(z)\in  {\underline{r}} \right\rbrace\cap\left\lbrace z\in\partial U\colon I(z)\in   {\underline{s}}\right\rbrace=\emptyset,\]since a point in $ \partial U $ has a unique itinerary. Moreover, any point $ z\in \partial U $ must belong to at least one cluster set, hence the previous three inclusions are in fact equalities.
 The fact that all points in $ \partial U $ are in the closure of a hair ends the proof of the proposition.
\end{proof}

	\begin{figure}[htb!]\centering
	\includegraphics[width=15cm]{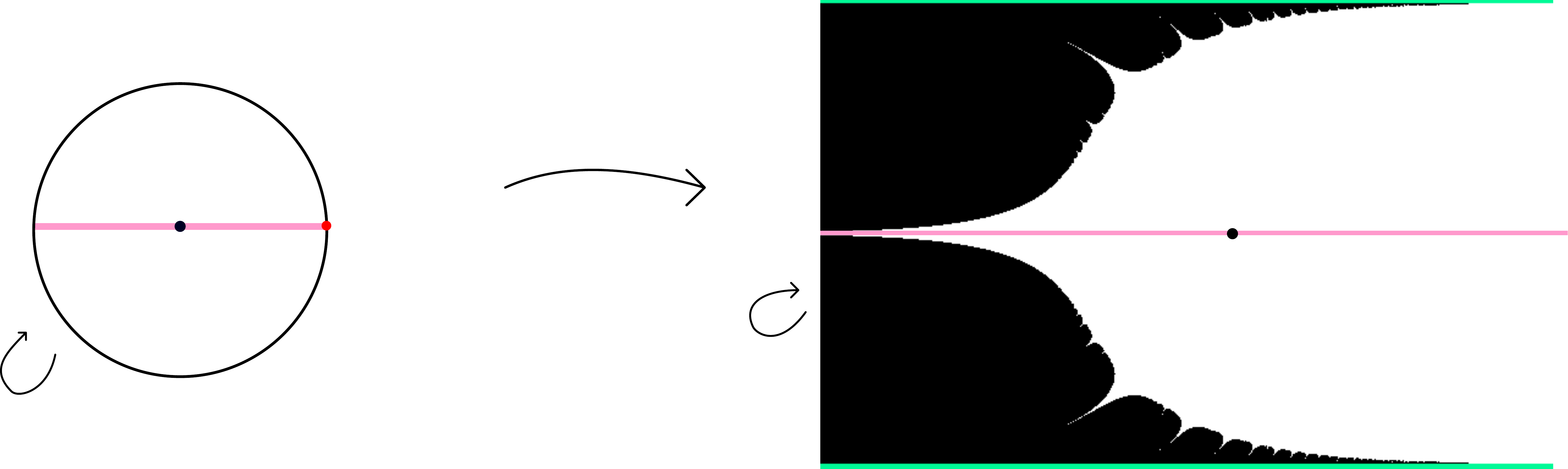}
	\setlength{\unitlength}{15cm}
	\put(-0.63, 0.205){$ \varphi $}
	\put(-1.016, 0.045){ $g $}
	\put(-0.54, 0.1){$f$}
	\put(-0.52, 0.28){ $L^+$}
	\put(-0.52, 0.0){ $L^-$}
	\put(-0.05, 0.23){ $\Omega_0$}
	\put(-0.05, 0.07){ $\Omega_1$}
	\put(-0.85, 0.18){ $D_0$}
	\put(-0.85, 0.11){ $D_1$}
	\caption{\footnotesize Representation of the Riemann map $ \varphi\colon\mathbb{D}\to U $, which fixes the real axis. The regions $ D_0, D_1, \Omega_0 $ and $ \Omega_1 $ are also represented, and it is clear that $ \varphi(D_0)\subset {\Omega_0} $ and  $ \varphi(D_1)\subset {\Omega_1} $ implying the correspondence between itineraries.}\label{fig-radial-limits}
\end{figure}
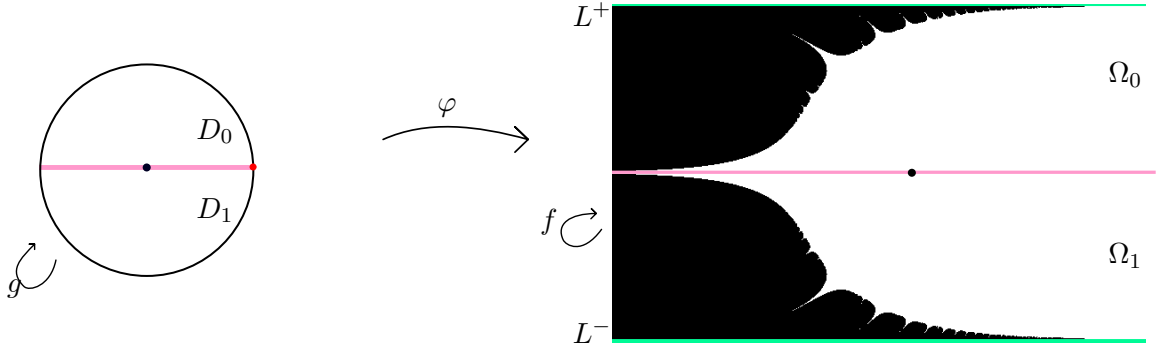

Let us observe that the previous proposition gives, in particular, a way to compute the impression of the prime end at 1, alternative to the one in \cite[Thm. 6.1]{baker-dominguez}.
\begin{corol}{\bf (Prime end at 1)}\label{lemma-prime-end}
	The prime end of $ U $ which corresponds by the Riemann map $ \varphi $ to 1 has the impression $ L^+\cup L^-\cup \left\lbrace \infty\right\rbrace  $. Equivalently, $ Cl(\varphi, 1)=L^+\cup L^-\cup \left\lbrace \infty\right\rbrace  $.
\end{corol}

The previous correspondence between itineraries and the fact that in each cluster set $ Cl(\varphi, e^{i\theta}) $ there is at most one accessible point, imply that there is at most one accessible point per itinerary. In particular,  in each hair and its landing set there is at most one accessible point. 

A first study on accessibility and radial limits was carried out by Baker and Domínguez, characterizing the accesses to infinity.

\begin{thm}{\bf (Accesses to infinity, {\normalfont\cite{baker-dominguez}})}\label{lemma-accesess} Accesses from $ U $ to infinity are characterized by the eventual preimages of 1, i.e.
	\[{\left\lbrace e^{i\theta}\colon \varphi^*(e^{i\theta})=\infty\right\rbrace} ={\left\lbrace e^{i\theta}\colon g^n(e^{i\theta})=1\textrm{, for some }n\geq0\right\rbrace}.\]
\end{thm}

Next, we prove \ref{teo:B}, which asserts that escaping points are non-accessible from $ U $, while points in $ \partial U $ having a bounded orbit are all accessible from $ U $. Using the Correspondence Theorem \ref{correspondence-theorem} between accesses and radial limits, we rewrite the statement of \ref{teo:B} as follows.

\begin{named}{Theorem C}\begin{enumerate}[label={\em (\alph*)}]
		\item Let $ e^{i\theta}\in\partial \mathbb{D} $ such that the radial limit $ z\coloneqq\varphi^*(e^{i\theta}) $ exists. Then, $ z $ is non-escaping.
		\item Let $ z\in \partial U $ be a point whose orbit is bounded. Then, there exists $ e^{i\theta}\in\partial \mathbb{D} $ such that $ \varphi^*(e^{i\theta})=z $, i.e. $ z $ is accessible from $ U $.
	\end{enumerate}
\end{named}
\begin{proof}\begin{enumerate}[label={(\alph*)}]
		\item 	The proof is based on the one developed by Baker and Domínguez in  \cite[Thm. 6.3]{baker-dominguez}.
		
		Assume $ z\coloneqq\varphi^*(e^{i\theta}) $ is an escaping point and let us define  the open set\[W\coloneqq \left\lbrace z\in S\colon \textrm{Re }z<-2\textrm{ and } \left| \textrm{Im }z\right| >\frac{\pi}{2}\right\rbrace. \]
		
		Iterating the function if needed, we can assume $f^n (z) \in W$, for all $ n\geq 0 $. Since the radial segment \[ \varphi_\theta\coloneqq \left\lbrace \varphi(re^{i\theta})\colon r\in (0,1)\right\rbrace \] lands at $ z $, one can choose $ r_0\in (0,1) $ such that $ \gamma\coloneqq\left\lbrace \varphi(re^{i\theta})\colon r\in (r_0,1)\right\rbrace \subset W  $. For points in $ W $ we have $ \textrm{Re }f(z)<\textrm{Re } z $. Hence, since $ \gamma $ is connected and $ f^n(z)\in W $ for $ n\geq 0 $, we have $ f^n(\gamma)\subset W $, for all $ n\geq 0 $. This is a contradiction because $ \gamma\subset U $, so points in $ \gamma $ must converge to $ +\infty $.
		
		\item First, we note that, by the the results in Section \ref{sect-5-non-escaping}, the only points in $ \widehat{S} $ with bounded orbit are endpoints $w_{\underline{s}} $ for bounded sequences $ {\underline{s}}\in\Sigma_2 $.
		  Therefore, the goal is to prove that, if $ e^{i\theta_{\underline{s}}}\in\partial \mathbb{D} $ has itinerary $ {\underline{s}}\in\Sigma_2 $, and $ {\underline{s}} $ is a bounded sequence, then $ \varphi^*(e^{i\theta_{\underline{s}}})=w_{\underline{s}} $. We note that the radial cluster set $ Cl_\rho(\varphi, e^{i\theta_{\underline{s}}}) $, which is connected, is contained in the cluster set $ Cl(\varphi, e^{i\theta_{\underline{s}}}) $ (see Sect. \ref{sect-2-prelim}), and for a bounded sequence, it holds \[Cl(\varphi, e^{i\theta_{\underline{s}}})= \overline{\gamma_{{\underline{s}}}^\infty}=\gamma_{{\underline{s}}}^\infty\cup \left\lbrace w_{\underline{s}}\right\rbrace\cup \left\lbrace \infty\right\rbrace, \] by Propositions \ref{prop-bounded-rays-land} and \ref{prop_correspondence_it}.
		  
	 Hence, it is enough to show that, if $ {\underline{s}} $ is a bounded sequence, then the radial cluster set $ Cl_\rho(\varphi, e^{i\theta_{\underline{s}}}) $ cannot contain any escaping point. 
	
Recall that $ g_{|\partial \mathbb{D}} $ is conjugate to the doubling map. Moreover, since $ {\underline{s}} $ contains at most $ N $ consecutive $ 0 $'s and $ 1 $'s, there exist $ 0<\theta_1 < \theta_2<\pi$ such that $ \theta_1 $ and $ \theta_2 $ are eventual preimages of 1 and  $ g^n(e^{i\theta_{\underline{s}}})\in \left( e^{i\theta_1}, e^{i\theta_2} \right) \cup  \left( e^{-i\theta_2}, e^{-i\theta_1} \right) $. Then, $ \varphi_{\theta_1}$ and $ \varphi_{\theta_2} $ are curves starting at 0 and landing at $ -\infty $ approaching $ L^+ $. Since $ \varphi $ is a bijection, $ f^n (\varphi_{\theta_{\underline{s}}})   $ is contained in the region bounded by $ \varphi_{\theta_1}$ and $ \varphi_{\theta_2} $ and its reflection along the real axis. Therefore, there exists $ R>0 $ such that, if we consider the open set $ W $ defined as before and  \[
W'\coloneqq \left\lbrace z\in S\colon \textrm{Re }z<-R\textrm{ and } \left| \textrm{Im }z\right| <\frac{\pi}{2}\right\rbrace,\]
then $ f^n(\varphi_{\theta_{\underline{s}}}) \cap W'=\emptyset $, for all $ n\geq 0 $. Compare with Figure \ref{fig-radial-limits-4}.
	
		\begin{figure}[htb!]\centering
		\includegraphics[width=15cm]{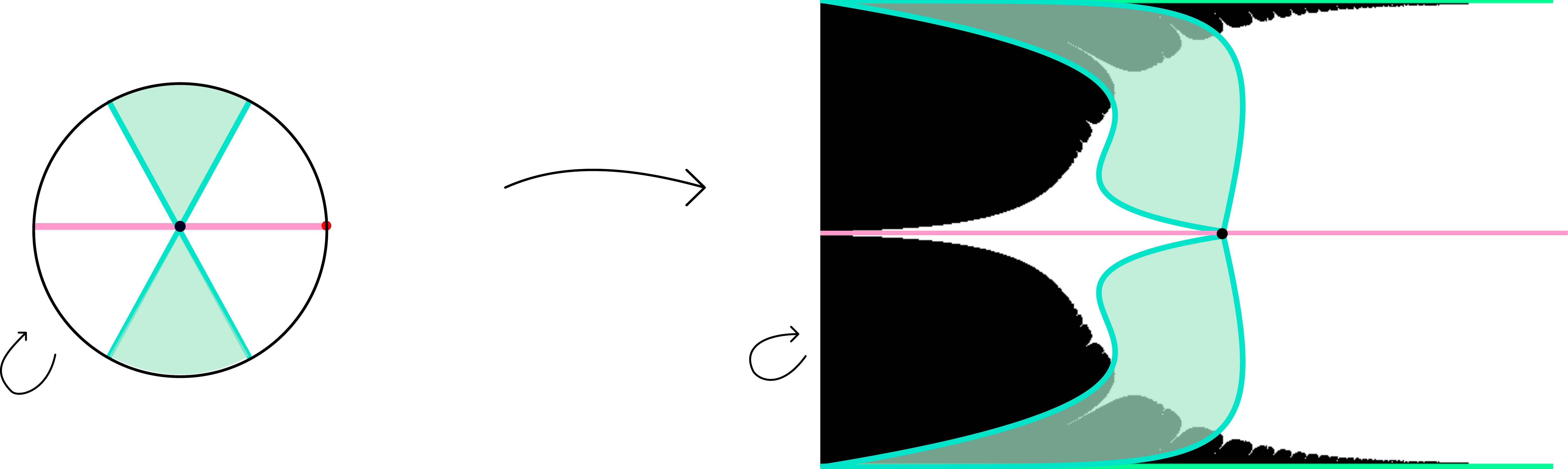}
		\setlength{\unitlength}{15cm}
		\put(-0.63, 0.205){$ \varphi $}
		\put(-1.016, 0.045){ $g $}
		\put(-0.54, 0.07){$f$}
		\put(-0.52, 0.28){ $L^+$}
		\put(-0.52, 0.0){ $L^-$}
		\put(-0.85, 0.24){ $e^{i\theta_0}$}
		\put(-0.95, 0.24){ $e^{i\theta_1}$}
		\put(-0.21, 0.22){ $\varphi_{\theta_0}$}
		\put(-0.33, 0.17){ $\varphi_{\theta_1}$}
		\caption{\footnotesize Schematic representation of the region bounded by $ \varphi_{\theta_1}$ and $ \varphi_{\theta_2} $ and its reflection along the real axis, where  $ f^n(\varphi_{\theta_{\underline{s}}})$ is contained, for all $ n\geq 0 $.}
		\label{fig-radial-limits-4}
		\end{figure}
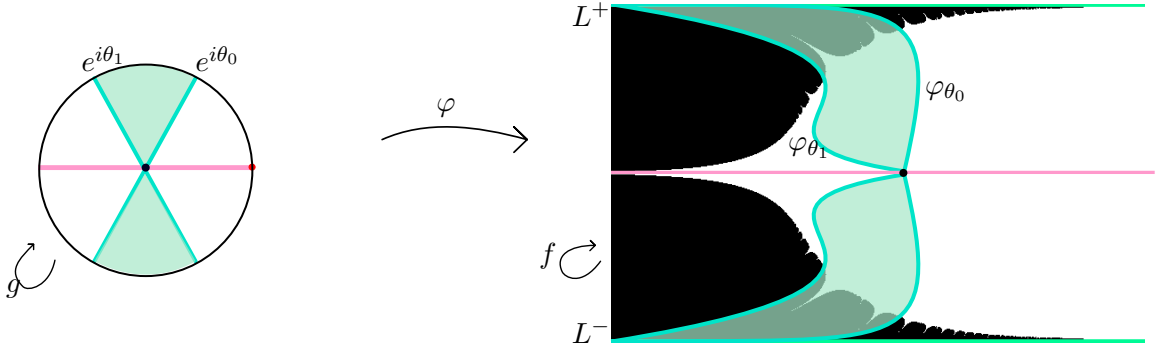
			
	Assume the radial cluster set contains an escaping point $ z $. Iterating the function if needed, we can assume $ \textrm{Re }f^n(z)<-R $, for all $ n\geq 0 $, so $ z\in W $. 
	Then, there exists a sequence of real numbers $ \left\lbrace t_n\right\rbrace_n  $ such that $ t_n\to +\infty$ and $ z_n\coloneqq \gamma^\infty_{\underline s}(t_n)\to z $, as $ n\to\infty $. Without loss of generality, since $ z\in W $, we shall assume $ \left\lbrace z_n \right\rbrace _n\subset W$. For points in $ W $ we have $ \textrm{Re }f(z)<\textrm{Re } z $, so they either belong to $ W $ or to $ W' $. But $ W' $ has been defined so that  $ f^n(\varphi_{\theta_{\underline{s}}}) \cap W'=\emptyset $, so 
	$ \left\lbrace f^k (z_n)\right\rbrace \in W$ for all $ k\geq0 $: a contradiction, since $ \left\lbrace z_n\right\rbrace _n\subset U $, and points in $ U $ converge to $ +\infty $.
		\end{enumerate}
\end{proof}

\begin{remark} 
	Alternatively,  \ref{teo:B} can be seen as a consequence from the results of 
	\cite{benini-rempe}. Indeed, in 	\cite[Sect. 6]{benini-rempe}, it is proved that, for functions in class $ \mathcal{B} $ and bounded postsingular set, accessible points in the boundary of an invariant Fatou component coincide with the endpoints of the hairs lying in its boundary (Remark 6.11). Such result can be applied to $ h(w)=we^{-w} $, semiconjugate to $ f(z)=z+e^{-z} $ (Sect. \ref{sec-3-dynamics-f}), to deduce that points with bounded orbit are accessible from $ U $, since they are the endpoint of a hair in $\partial U $. 
	
	Nevertheless, although \ref{teo:B} can be seen as a consequence of this more general result, it relies strongly on the study of the landing sets of the dynamic rays, carried out in the previous section, which has to be done specifically for our function.
	Moreover, our construction shows explicitly the relation between the dynamics of the inner function in $ \partial\mathbb{D} $ and the dynamics of $ f $ in $ \partial U $, which was the main goal of the paper. 
\end{remark}

\section{Periodic points in $ \partial U $: Proof of \ref{teo:C}}\label{sect-periodic-points}
  This last section of the paper is dedicated to prove \ref{teo:C}, which asserts that periodic points are dense in $ \partial U $. Although it is known that periodic points are dense in the Julia set, if we restrict ourselves to the boundary of a Baker domain, it is not known, in general, the existence of a single periodic point. 
  
  The general argument used to prove that periodic points are dense in the Julia set (e.g. \cite[Thm. III.3.1]{carlesongamelin}) cannot be used, since it gives no control about the resulting periodic point. The proof we present allows us to find a periodic point in any neighborhood of any point in $ \partial U $, whose orbit is entirely contained in $ S $, and hence implying that the periodic point is in $ \partial U $.
  
  \begin{named}{Theorem D} Periodic points are dense in $ \partial U $.
  \end{named}
  \begin{proof}
  	In view of Theorem \ref{teo-dynamics-boundary-baker}, it is enough to approximate $ z\in\partial U $ having a dense orbit by periodic points in $ \partial U $. Let us fix $ \varepsilon>0 $ and consider the disk $ D(z, \varepsilon) $. Without loss of generality, we can assume $ D(z, \varepsilon)\subset S $ and $ D(z,\varepsilon)\cap\overline{V}=\emptyset $, where $ V $ is the absorbing domain defined in Section \ref{sec-3-dynamics-f}. We also assume $ \varepsilon<1 $.
  	
  	Recall that $ f $ is expanding in  $ S\smallsetminus \overline{V} $ and uniformly expanding in any left-half plane intersected with it (see Rmk. \ref{remark_expansion}). In particular, the map is uniformly expanding in $ S\cap \left\lbrace \textrm{Re }z<-2+\varepsilon\right\rbrace  $ with   constant of expansion $ \lambda>1 $. 
  	
  	Take $ n_0 >0$ such that $ \lambda^{n_0} >2 $. Since the orbit of $ z $ is assumed to be dense in $ \partial U $, it visits infinitely many times $ S\cap \left\lbrace \textrm{Re}z<-2\right\rbrace  $. Let $ n_1 $ be such that \[\#\left\lbrace n<n_1\colon \textrm{Re } f^{n}(z)<-2 \right\rbrace\geq n_0  .\]   Since the orbit of $ z $ is dense, there exists $ n_2 >n_1$ with $z_{n_2}\coloneqq f^{n_2}(z)\in D(z,\varepsilon) $. Then, $ \phi_{s_0}\circ\dots\circ\phi_{s_{n_2-1}} (z_{n_2})=z $, for a suitable choice of $ s_0,\dots, s_{n_2-1}\in \left\lbrace 0,1\right\rbrace  $.
  	
  	We claim that $ \phi_{s_0}\circ\dots\circ\phi_{s_{n_2-1}} (D(z,\varepsilon))\subset D(z,\varepsilon) $. Indeed, since $ D(z,\varepsilon)\cap\overline{V}=\emptyset $, we have $ D(z, \varepsilon)= D_\rho (z, \varepsilon) $, for the $ \rho $-distance defined in \ref{def-rho-distance}. The forward invariance of $ V $ gives $ \phi_{s_0}\circ\dots\circ\phi_{s_{n}}(D(z,\varepsilon))\subset S\smallsetminus \overline{V} $, for all $ n\geq 0 $. Moreover, since inverses are contracting, if $  \phi_{s_0}\circ\dots\circ\phi_{s_{n}}(z)\in S\cap \left\lbrace \textrm{Re }z<-2\right\rbrace  $, we have $\phi_{s_0}\circ\dots\circ\phi_{s_{n}}(D(z,\varepsilon)) \subset S\cap \left\lbrace \textrm{Re }z<-2+\varepsilon\right\rbrace  $. Hence, after applying $ n_2 $ inverses, since the iterated preimages of $ D(z,\varepsilon) $ are contained in $ \left\lbrace \textrm{Re }z<-2+\varepsilon\right\rbrace  $ at least $ n_0 $ times, $ \rho $-distances in $ D(z,\varepsilon) $ are contracted by a factor less than $ \frac{1}{\lambda^{n_0}} $.
  	Therefore we have:
  	\[\rho(\phi_{s_0}\circ\dots\circ\phi_{s_{n_2-1}}(z), z)=\rho(\phi_{s_0}\circ\dots\circ\phi_{s_{n_2-1}}(z), \phi_{s_0}\circ\dots\circ\phi_{s_{n_2-1}}(z_{n_2}))\leq \frac{1}{\lambda^{n_0}}\rho(z,z_{n_2})\leq \frac{1}{2}\varepsilon.\]
  	Now let $ w\in D(z,\varepsilon) $, then
  	\[ \rho(\phi_{s_0}\circ\dots\circ\phi_{s_{n_2-1}}(w),\phi_{s_0}\circ\dots\circ\phi_{s_{n_2-1}}(z))\leq \frac{1}{\lambda^{n_0}}\rho(w,z)\leq \frac{1}{2}\varepsilon. \]
  	Therefore, applying the triangle inequality, one deduces that $\phi_{s_0}\circ\dots\circ\phi_{s_{n_2-1}}(w) \in D(z,\varepsilon) $, for any $ w\in D(z,\varepsilon) $, as desired.
  	
  	Finally, observe that $ \rho(\phi_{s_0}\circ\dots\circ\phi_{s_{n_2-1}} $ is well-defined in $  \overline{D(z,\varepsilon)}$, and \[ \phi_{s_0}\circ\dots\circ\phi_{s_{n_2-1}} (\overline{D(z,\varepsilon)})\subset \overline{D(z,\varepsilon)} .\]
  	Hence,  Brouwer fixed-point theorem guarantees the existence of a fixed point $ z_0 $ for $ \phi_{s_0}\circ\dots\circ\phi_{s_{n_2-1}} $ in $ D(z,\varepsilon) $. This point is periodic for $ f $. Moreover, since its orbit is all contained in $ S $, we have $ z_0\in\partial U $, by Proposition \ref{prop-caracteritzacio-frontera}. This ends the proof of \ref{teo:C}.
  \end{proof}

\printbibliography

@article{bfjk19,
	author = {Barański, K. and Fagella , N. and Jarque, X. and Karpińska, B.},
year = {2019},
pages = {679-706},
title = {Escaping points in the boundaries of Baker domains},
volume = {137},
journal = {Journal d'Analyse Mathématique},
shorthand={BFJK19}
}

@article{bfjk15,
	author = {Barański, K. and Fagella , N. and Jarque, X. and Karpińska, B.},
	year = {2017},
	pages = {1835–1867},
	title = {Accesses to infinity from Fatou components},
	volume = {369},
	number={3},
	journal = {Transactions of the American Mathematical Society},
	shorthand={BFJK17}
}

@article{bfjk15-absorbing,
	author = {Barański, K. and Fagella , N. and Jarque, X. and Karpińska, B.},
	year = {2015},
	pages = {144-162},
	title = {Absorbing sets and Baker domains for holomorphic maps},
	volume = {92},
	number={1},
	journal = {Journal of the London Mathematical Society (2)},
	shorthand={BFJK15}
}

@article{r3s,
	author = {Rottenfu{\ss}er, G. and Rückert, J. and Rempe, L. and Schleicher, D.},
	year = {2010},
	pages = {77-125},
	title = {Dynamic rays of bounded-type entire functions},
	volume = {173},
	journal = {Annals of Mathematics},
	shorthand={RRRS10}
}

@article{baranski,
	author = {Barański, K.},
	year = {2007},
	pages = {33-59},
	title = {Trees and hairs for some hyperbolic entire maps of finite order},
	volume = {257},
	number={1},
	journal = {Mathematische Zeitschrift}
}

@article{devaney-indecomposable,
	author = {Devaney, R.},
	year = {1993},
	pages = {627-634},
	title = {Knaster-like continua and complex dynamics},
	volume = {13},
	number={4},
	journal = {Ergodic Theory and Dynamical Systems}
}

@article{devaney-jarque-indecomposable,
	author = {Devaney, R. and Jarque, X.},
	year = {2002},
	pages = {1-12},
	title = {Indecomposable continua in exponential dynamics},
	volume = {6},
	journal = {Conformal Geometry and Dynamics}
}

@article{devaney-jarque-rocha-indecomposable,
	author = {Devaney, R. and Jarque, X. and Moreno Rocha, M.},
	year = {2005},
	pages = {3281-3292},
	title = {Indecomposable continua and Misiurewicz points in exponential dynamics},
	volume = {15},
	number={10},
	journal = {International Journal of Bifurcation and Chaos in Applied Sciences and Engineering},
		shorthand={DJM05}
}

@article{baranski-karpinska,
	author = {Barański, K. and Karpińska, B.},
	year = {2007},
	pages = {391-415},
	title = {Coding trees and boundaries of attracting basins for some entire maps},
	volume = {20},
	number={2},
	journal = {Nonlinearity}
}

@article{efjs,
	author = {Evdoridou, V. and Fagella , N. and Jarque, X. and Sixsmith, D.},
	year = {2019},
	pages = {536-550},
	title = {Singularities of inner functions associated with hyperbolic maps},
	volume = {477},
	journal = {Journal of Mathematical Analysis and Applications},
		shorthand={EFJS19}
}

@article{fh,
	author = {Fagella , N.  and Henriksen, C.},
	year = {2006},
	pages = {379-394},
	title = {Deformation of Entire Functions with Baker Domains},
	volume = {15},
	journal = {Discrete and Continuous Dynamical Systems},
}

@book{milnor,
	ISBN = {9780691124889},
	author = {J. Milnor},
	publisher = {Princeton University Press},
	title = {Dynamics in One Complex Variable. Third Edition. },
	year = {2006}
}

@book{carlesongamelin,
	ISBN = {0387979425},
	author = {L. Carleson and T. W. Gamelin},
	publisher = {Springer-Verlag},
	title = {Complex dynamics },
	year = {1993}
}

@article{cowen,
	author = {C. C. Cowen},
	year = {1981},
	pages = {69-95},
	title = {Iteration and the solution of functional equations for functions in the unit disk},
	volume = {265},
	number={1},
	journal = {Transactions of the American Mathematical Society},
}

@article{konig,
	author = {H. König},
	year = {1999},
	month = {02},
	pages = {153-170},
	title = {Conformal conjugacies in Baker domains},
	volume = {59},
	number={1},
	journal = {Journal of the London Mathematical Society},
}

@article{bargmann,
	author = {D. Bargmann},
	year = {2008},
	pages = {1-36},
	title = {Iteration of inner functions and boundaries of components of the Fatou set},
	journal = {Transcendental Dynamics and Complex Analysis. Cambridge University Press},
}

@article{bergweiler,
	author = {W. Bergweiler},
	year = {1993},
	pages = {151-188},
	title = {Iteration of meromorphic functions},
	volume = {29},
	number={2},
	journal = {Bulletin of the American Mathematical Society },
}

@article{bergweiler95,
	author = {W. Bergweiler},
	year = {1995},
	pages = {251-256},
	title = {On the Julia set of analytic self-maps of the punctured plane},
	volume = {15},
	journal = {Analysis},
}

@article{baker1984,
	author = {I. N. Baker},
	year = {1984},
	pages = {337-360},
	title = {Wandering domains in the iteration of entire functions},
	volume = {47},
	journal = {Proceedings of the Journal Mathematical Society},
}

@book{pommerenke,
	ISBN = {3540547517},
	title = {Boundary behaviour of conformal maps},
	publisher = {Springer-Verlag, Berlin},
	author = {C. Pommerenke},
	year = {1992}
}

@article{fatou1920,
	author = {P. Fatou},
	year = {1920},
	pages = {208–314},
	title = {Sur les équations fonctionnelles},
	volume = {48},
	journal = {Bulletin de la Société Mathématique de France },
}

@article{goldberg-devaney,
	author = { R. L. Devaney and L. R. Goldberg },
	year = {1987},
	pages = {253-266},
	title = {Uniformization of attracting basins for exponential maps},
	volume = {55},
	number={2},
	journal = {Duke Mathematical Journal},
}

@article{rippon-stallard,
	author = { P. J. Rippon and  G. M. Stallard},
	year = {2018},
	pages = {801–810},
	title = {Boundaries of univalent Baker domains},
	volume = {134},
	number={2},
	journal = {Journal d'Analyse Mathématique},
}

@article{rippon-bakerdomains,
	author = { P. J. Rippon},
	year = {2008},
	pages = {371-395},
	title = {Baker domains},
	volume = {348},
	journal = {London Math. Soc. Lecture Note Ser.},
	editor={Cambridge University Press},
}

@article{rippon-bakerdomainsmeromorphic,
	author = { P. J. Rippon},
	year = {2006},
	pages = {1225-1233},
	title = {Baker domains of meromorphic functions},
	volume = {26},
	number={4},
	journal = {Ergodic Theory and Dynamical Systems},
}

@article{rippon-stallard-familiesbakersI,
	author = { P. J. Rippon and  G. M. Stallard},
	year = {1999},
	pages = {1005-1012},
	title = {Families of Baker domains. I},
	volume = {12},
	number={4},
	journal = {Nonlinearity},
}

@article{rippon-stallard-familiesbakersII,
	author = { P. J. Rippon and  G. M. Stallard},
	year = {1999},
	pages = {67-78},
	title = {Families of Baker domains. II},
	volume = {3},
	journal = {Conformal Geometry and Dynamics},
}

@article{baker-dominguez,
	author = {I. N. Baker and Domínguez, P. },
	year = {1999},
	pages = {437-464},
	title = {Boundaries of unbounded Fatou components of entire functions},
	volume = {24},
	journal = {Annales Academiae Scientiarum Fennicae Mathematica},
}

@article{doering-mañé,
	author = {Doering, C. and Mañé, R.},
	year = {1991},
	pages = {1-79},
	title = {The Dynamics of Inner Functions},
	volume = {Volume 3},
	journal = {Ensaios de Matemática (SBM)}
}

@article{eremenko-lyubich,
	author = {Eremenko, A. and Lyubich, M.},
	year = {1992},
	pages = {989-1020},
	title = {Dynamical properties of some classes of entire functions},
	volume = {42},
	number={4},
	journal = {Annales de l'Institut Fourier}
}

@book{Steinmetz+2011,
	author = {N. Steinmetz},
	title = {Rational Iteration: Complex Analytic Dynamical Systems},
	year = {1993},
	publisher = {De Gruyter},
	ISBN = {3110137658}
}

@article{rempe2007,
	author = {Rempe, L.},
	year = {2007},
	pages = {353-369},
	title = {On nonlanding dynamic rays of exponential maps},
	volume = {32},
	journal = {Annales Academiae Scientiarum Fennicae}
}

@article{schleicher-zimmer,
	author = {Schleicher, D. and Zimmer, J.},
	year = {2003},
	pages = {380-400},
	title = {Escaping points of exponential maps},
	volume = {67},
	number={2},
	journal = {Journal of the London Mathematical Society}
}

@phdthesis{tesi-lasse-rempe,
	  author =       {Rempe, L.},
	title =        {Dynamics of exponential maps},
	school =       {Christian-Albrechts-Universität Kiel},
	year =         {2003},
	type =         {},
}

@article{curry,
	author = {Curry, S.B.},
	year = {1991},
	pages = {145-151},
	title = {One-dimensional nonseparating continua with disjoint $\varepsilon$-dense subcontinua },
	volume = {32},
	number={2},
	journal = {Topology and its Applications}
}

@article {benini-rempe,
	AUTHOR = {Benini, A. M. and Rempe, L.},
	TITLE = {A landing theorem for entire functions with bounded
	post-singular sets},
	JOURNAL = {Geometric and Functional Analysis},
	VOLUME = {30},
	YEAR = {2020},
	NUMBER = {6},
	PAGES = {1465-1530},
}

@article {leti,
	AUTHOR = {Pardo-Sim\'{o}n, L.},
	TITLE = {Criniferous entire maps with absorbing {C}antor bouquets},
	JOURNAL = {Discrete and Continuous Dynamical Systems. Series A},
	VOLUME = {42},
	YEAR = {2022},
	NUMBER = {2},
	PAGES = {989-1010},
	shorthand={Par22}
}

@article {bodelon,

    AUTHOR = {Bodel\'{o}n, C. and Devaney, R. L. and Hayes, M. and
Roberts, G. and Goldberg, L. R. and Hubbard, J. H.},
TITLE = {Hairs for the complex exponential family},
JOURNAL = {International Journal of Bifurcation and Chaos in Applied
Sciences and Engineering},
VOLUME = {9},
YEAR = {1999},
NUMBER = {8},
PAGES = {1517--1534},
}

@book {mesuraharmonica,
	AUTHOR = {Bracci, F. and Contreras, M. D. and D\'{\i}az-Madrigal,
	S.},
	TITLE = {Continuous semigroups of holomorphic self-maps of the unit
	disc},
	year={2020},
	SERIES = {Springer Monographs in Mathematics},
	PUBLISHER = {Springer},
	ISBN = {9783030367817},
		shorthand={BCD20}
}

@book {harmonicmeasure2,
	AUTHOR = {Garnett, J. B. and Marshall, D. E.},
	TITLE = {Harmonic measure},
	SERIES = {New Mathematical Monographs},
	VOLUME = {2},
	PUBLISHER = {Cambridge University Press, Cambridge},
	YEAR = {2005},
	ISBN = {9780521470186},
}

@incollection {DevaneyExp,
	AUTHOR = {Devaney, R. L.},
	TITLE = {Complex dynamics and entire functions},
	BOOKTITLE = {Complex dynamical systems ({C}incinnati, {OH}, 1994)},
	SERIES = {Proc. Sympos. Appl. Math.},
	VOLUME = {49},
	PAGES = {181--206},
	PUBLISHER = {Amer. Math. Soc., Providence, RI},
	YEAR = {1994},
}

\end{document}